\documentclass[11pt]{article}
\usepackage[T1]{fontenc}
\usepackage{lmodern}

\usepackage{amsmath,amsthm,amssymb}
\usepackage[usenames,dvipsnames]{xcolor}
\usepackage{enumerate}
\usepackage{graphicx}
\usepackage{cite}
\usepackage{comment}
\usepackage{oands}
\usepackage{tikz}
\usepackage{changepage}
\usepackage{bbm}
\usepackage{mathtools}
\usepackage[margin=1in]{geometry}
\usepackage[pagewise,mathlines]{lineno}
\usepackage{appendix}
\usepackage{stmaryrd}
\usepackage{multicol}
\usepackage{microtype}
\usepackage[colorinlistoftodos]{todonotes}
\usepackage{dsfont}
\usepackage{mathrsfs}  
\usepackage{multirow}
\usepackage{graphicx,subfigure}
\usepackage{array}
\usepackage{bm}
\newcolumntype{P}[1]{>{\centering\arraybackslash}p{#1}}

\usepackage[pdftitle={Power-law bounds for increasing subsequences in Brownian separable permutons
	and homogeneous sets in Brownian cographons},
  pdfauthor={Jacopo Borga, William Da Silva, and Ewain Gwynne},
colorlinks=true,linkcolor=NavyBlue,urlcolor=RoyalBlue,citecolor=PineGreen,bookmarks=true,bookmarksopen=true,bookmarksopenlevel=2,unicode=true,linktocpage]{hyperref}

\usepackage[capitalise,noabbrev]{cleveref}

\theoremstyle{plain}
\newtheorem{thm}{Theorem}[section]
\newtheorem{cor}[thm]{Corollary}
\newtheorem{lem}[thm]{Lemma}
\newtheorem{prop}[thm]{Proposition}
\newtheorem{conj}[thm]{Conjecture}

\def\@rst #1 #2other{#1}
\newcommand\MR[1]{\relax\ifhmode\unskip\spacefactor3000 \space\fi
  \MRhref{\expandafter\@rst #1 other}{#1}}
\newcommand{\MRhref}[2]{\href{http://www.ams.org/mathscinet-getitem?mr=#1}{MR#2}}

\theoremstyle{definition}

\newtheorem{remark}[thm]{Remark}

\numberwithin{equation}{section}

\newcommand{\dsb}{\begin{adjustwidth}{2.5em}{0pt}
\begin{footnotesize}}
\newcommand{\dse}{\end{footnotesize}
\end{adjustwidth}}

\newcommand{\ssb}{\begin{adjustwidth}{2.5em}{0pt}}
\newcommand{\sse}{\end{adjustwidth}}

\newcommand{\aryb}{\begin{eqnarray*}}
\newcommand{\arye}{\end{eqnarray*}}
\def\alb#1\ale{\begin{align*}#1\end{align*}}
\def\allb#1\alle{\begin{align}#1\end{align}}
\newcommand{\eqb}{\begin{equation}}
\newcommand{\eqe}{\end{equation}}
\newcommand{\eqbn}{\begin{equation*}}
\newcommand{\eqen}{\end{equation*}}

\newcommand{\BB}{\mathbbm}

\newcommand{\op}{\operatorname}

\newcommand{\wt}{\widetilde}

\newcommand{\mcl}{\mathcal}

\newcommand{\eps}{\varepsilon}

\DeclareMathOperator{\Perm}{Perm}
\DeclareMathOperator{\Graph}{Graph}

\newcommand{\bbP}{\mathbb{P}}
\newcommand{\bbE}{\mathbb{E}}
\newcommand{\bbN}{\mathbb{N}}

\newcommand{\efrak}{\mathfrak{e}}
\newcommand{\sfrak}{\mathfrak{s}}
\newcommand{\calS}{\mathcal{S}}

\newcommand{\sfS}{\mathsf{S}}

\newcommand{\Ffrak}{\mathfrak{F}}
\newcommand{\lstar}{\lambda_{\ast}}

\newcommand{\calE}{\mathcal{E}}

\newcommand{\calG}{\mathcal{G}}

\let\originalleft\left
\let\originalright\right
\renewcommand{\left}{\mathopen{}\mathclose\bgroup\originalleft}
\renewcommand{\right}{\aftergroup\egroup\originalright}

\DeclareMathOperator{\LIS}{LIS}
\DeclareMathOperator{\LHS}{LHS}

\DeclareMathOperator{\LIN}{LIN}
\DeclareMathOperator{\LCL}{LCL}
\DeclareMathOperator{\Leb}{Leb}

\title{Power-law bounds for increasing subsequences in Brownian separable permutons and homogeneous sets in Brownian cographons}
 \date{ }
 \author{
\begin{tabular}{c} Jacopo Borga\\[-5pt]\small Stanford University \end{tabular}
\begin{tabular}{c} William Da Silva\\[-5pt]\small University of Vienna \end{tabular} 
\begin{tabular}{c} Ewain Gwynne\\[-5pt]\small University of Chicago \end{tabular} 
}

\begin{document}

\maketitle
\thispagestyle{empty}
\vspace{-0.5cm}

\begin{abstract}
The Brownian separable permutons are a one-parameter family -- indexed by $p\in(0,1)$ -- of universal limits of random constrained permutations. We show that for each $p\in (0,1)$, there are explicit constants $1/2 < \alpha_*(p) \leq \beta^*(p) < 1$ such that the length of the longest increasing subsequence in a random permutation of size $n$ sampled from the Brownian separable permuton is between $n^{\alpha_*(p) - o(1)}$ and $n^{\beta^*(p) + o(1)}$ with probability tending to 1 as $n\to\infty$. 
In the symmetric case $p=1/2$, we have $\alpha_*(p) \approx 0.812$ and $\beta^*(p)\approx 0.975$. We present numerical simulations which suggest that the lower bound $\alpha_*(p)$ is close to optimal in the whole range $p\in(0,1)$. 
	
Our results work equally well for the closely related Brownian cographons. In this setting, we show that for each $p\in (0,1)$, the size of the largest clique (resp.\ independent set) in a random graph on $n$ vertices sampled from the Brownian cographon is between $n^{\alpha_*(p) - o(1)}$ and $n^{\beta^*(p) + o(1)}$ (resp.\ $n^{\alpha_*(1-p) - o(1)}$ and $n^{\beta^*(1-p) + o(1)}$) with probability tending to 1 as $n\to\infty$.

Our proofs are based on the analysis of a fragmentation process embedded in a Brownian excursion introduced by Bertoin (2002). We expect that our techniques can be extended to prove similar bounds for uniform separable permutations and uniform cographs.
\end{abstract}

\vspace{-0.5cm}

\begin{figure}[h!]
	\centering
	\includegraphics[width=.26\textwidth]{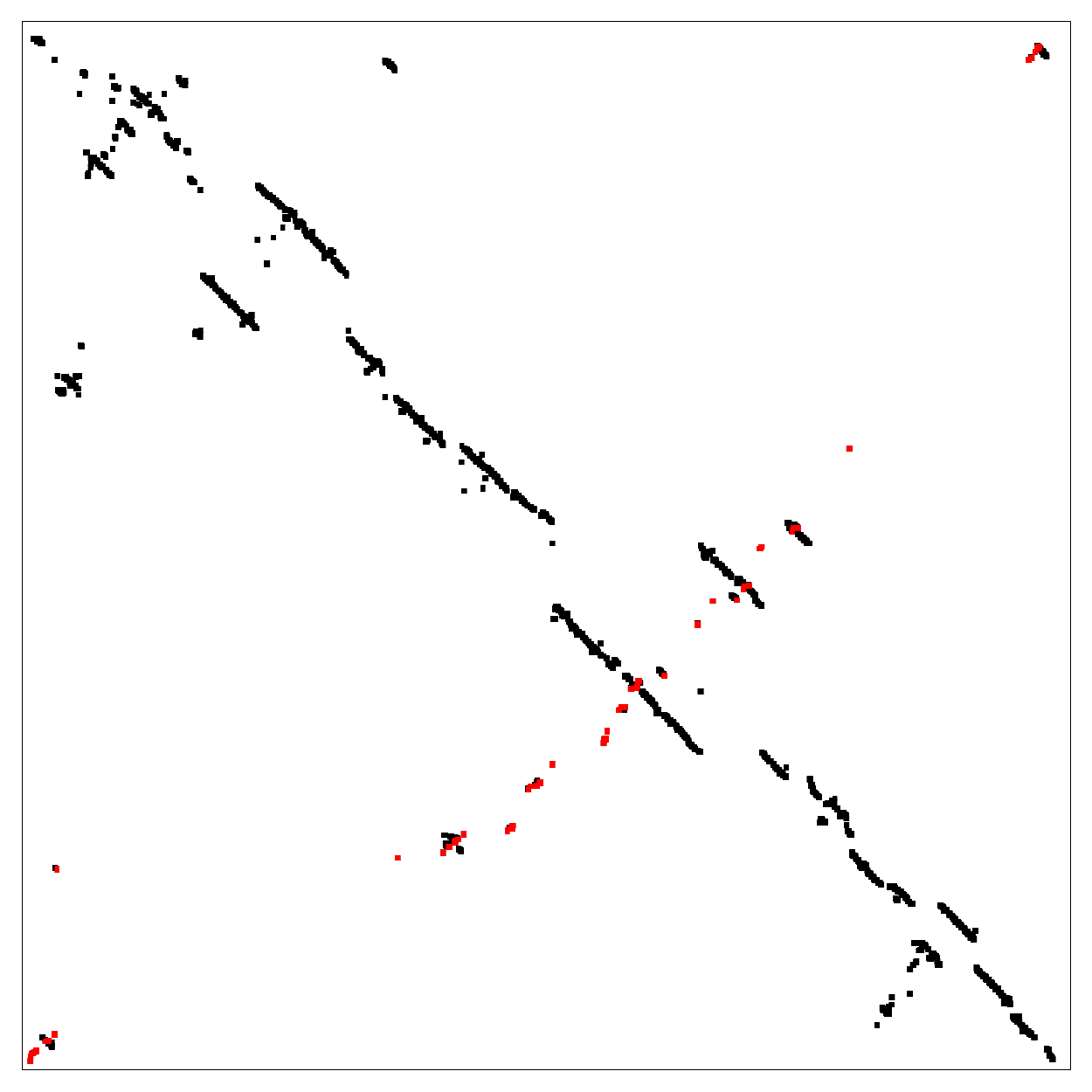}
	\includegraphics[width=.26\textwidth]{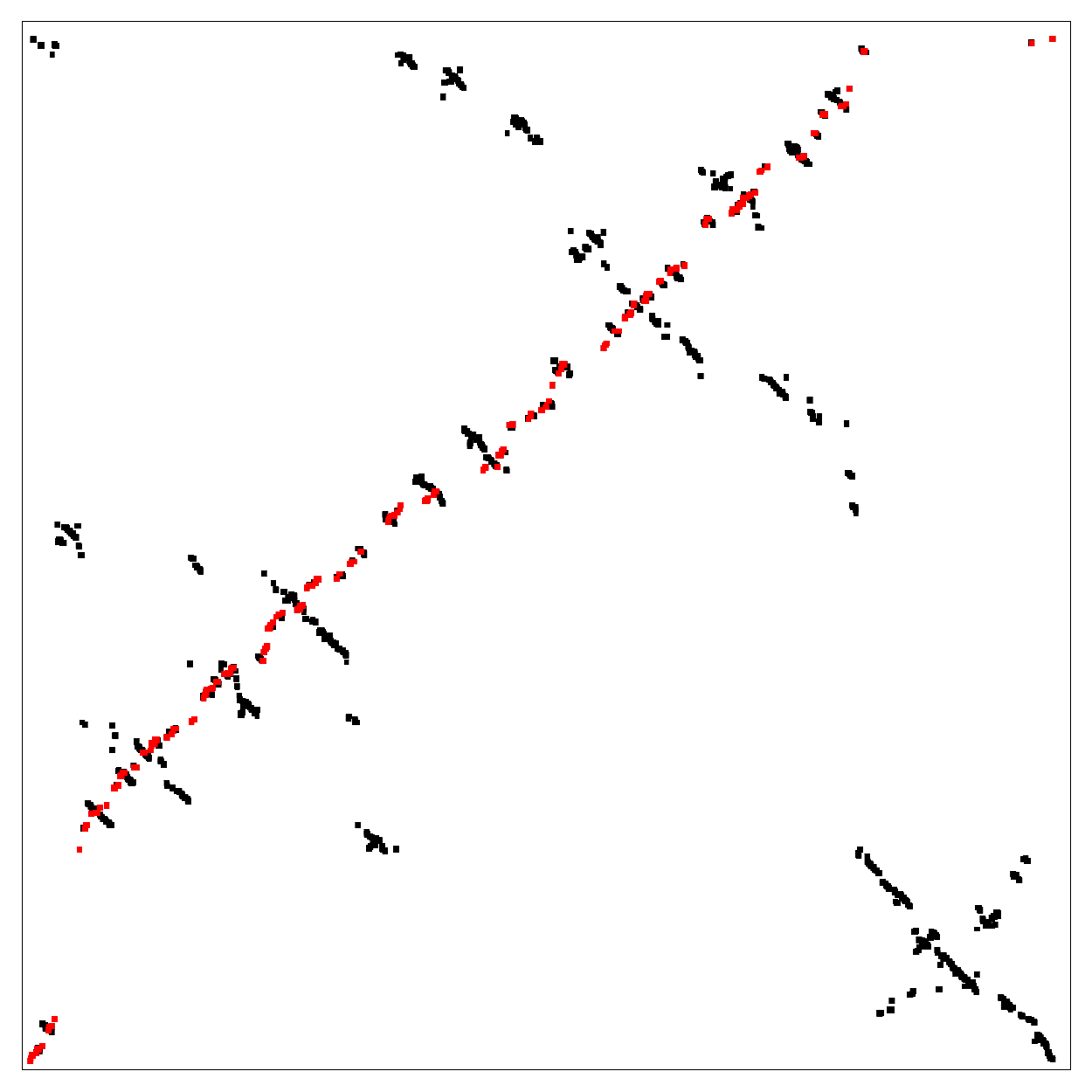} 
	\includegraphics[width=.26\textwidth]{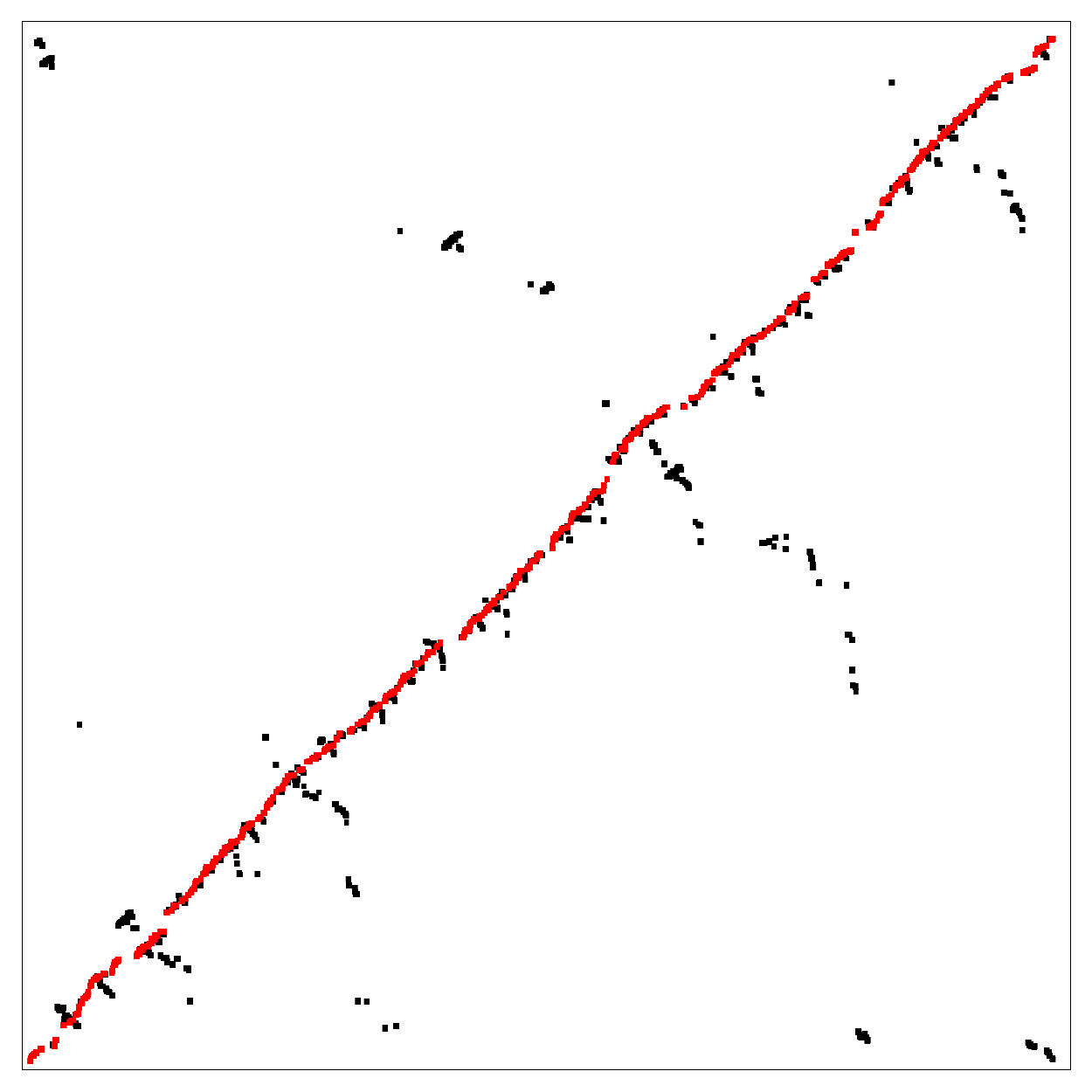}
	\includegraphics[width=.26\textwidth]{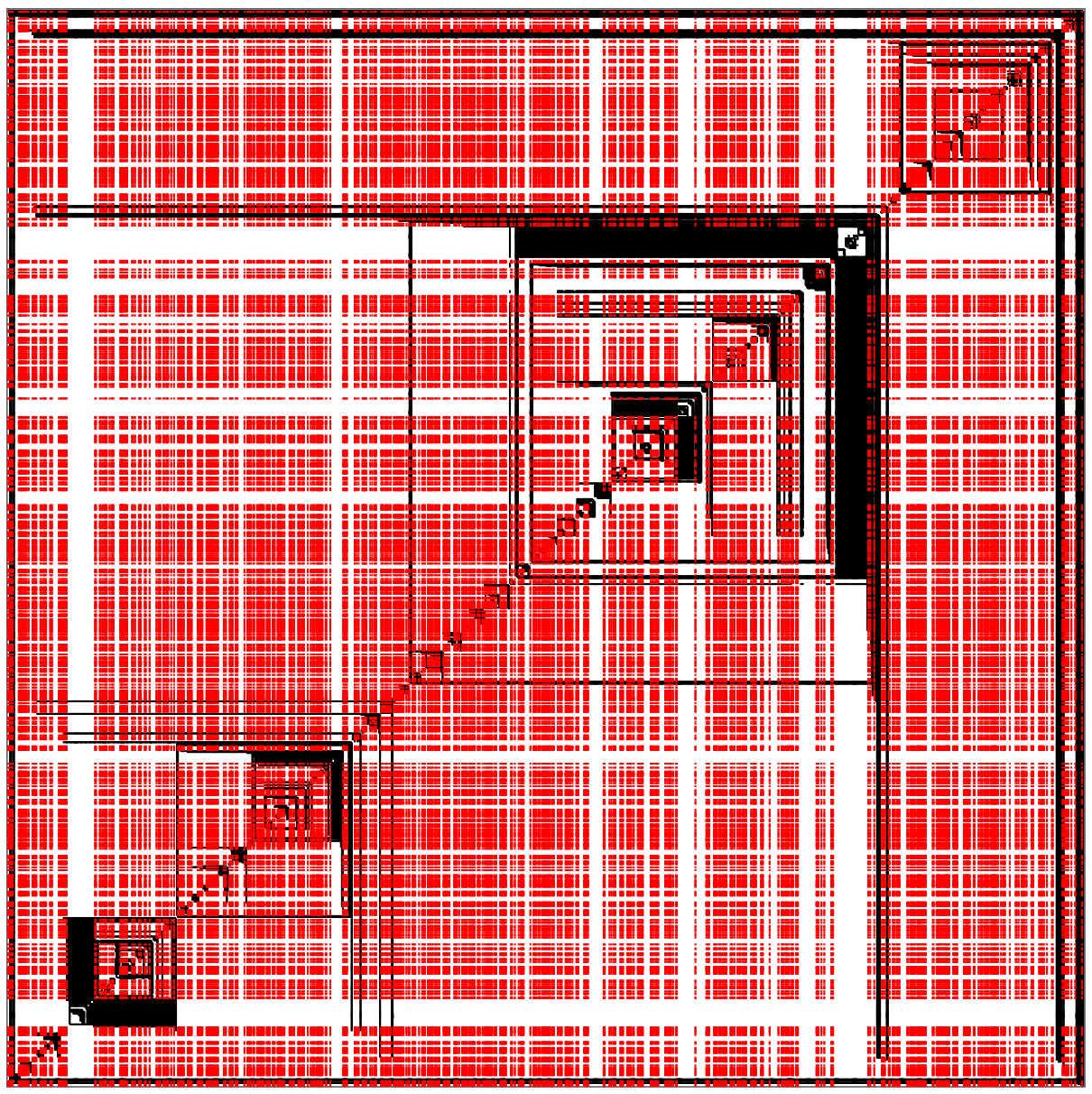}  
	\includegraphics[width=.26\textwidth]{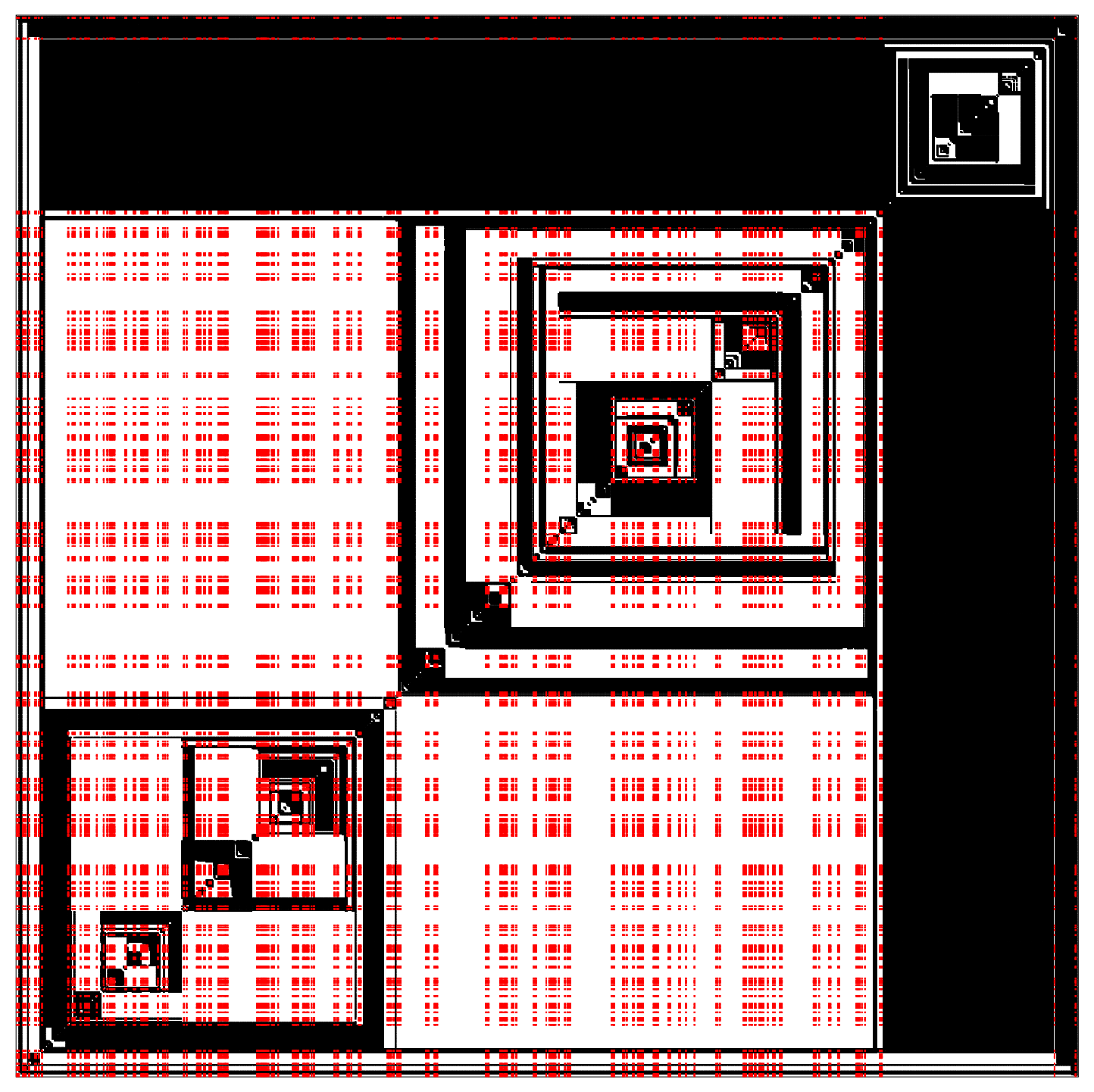}
	\includegraphics[width=.26\textwidth]{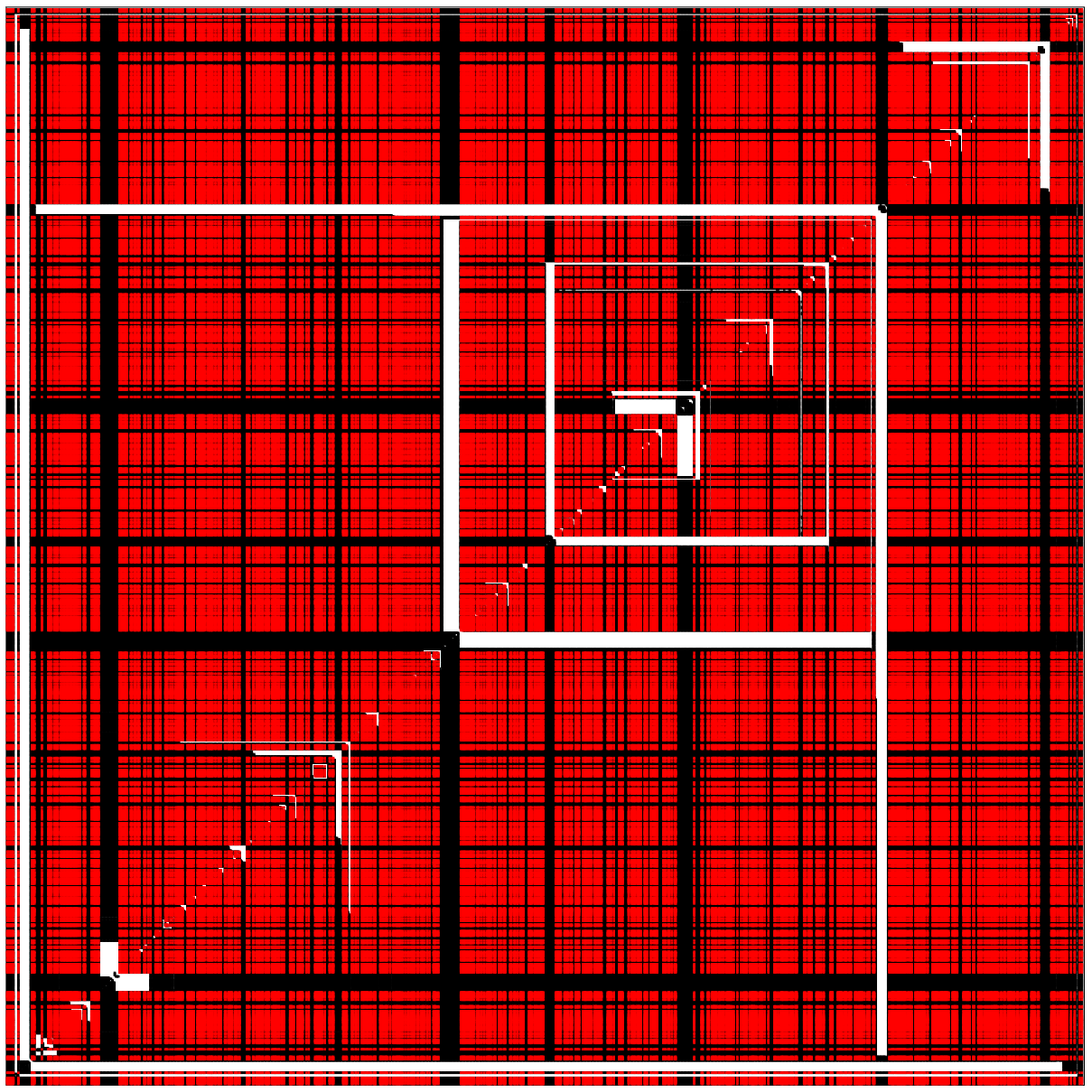} 
	\vspace{-0.25cm}
	\caption{\label{fig-simulations} \footnotesize{\textbf{Top:} The diagram of three large permutations (in black) sampled from the Brownian separable permuton with parameter $p=0.2,0.5,0.9$ (from left to right). In red we highlighted one longest increasing subsequence. \textbf{Bottom:} The adjacency matrix of three large graphs (ones are plotted in black) sampled from the Brownian cographon with parameter $p=0.2,0.5,0.9$ (from left to right). In red we highlighted one largest homogeneous set. In the first two samples it is an independent set, while in the third case it is a clique.}}
\end{figure}

 

\tableofcontents

\bigskip

\noindent\textbf{Acknowledgments.} We thank four anonymous referees for helpful comments on an earlier version of this article. We thank Jean Bertoin for some helpful discussion on Lévy and fragmentation processes and Valentin Féray for some helpful discussion on the problems investigated in this paper. We are also grateful to Arka Adhikari, Élie Aïdékon, Omer Angel, Matija Bucic, Amir Dembo and Lucas Teyssier for interesting discussions.
E.G.\ was partially supported by a Clay research fellowship. 
W.D.S. acknowledges the support of the two Austrian Science Fund (FWF) grants on “Scaling limits in random conformal geometry” (DOI: 10.55776/P33083) and “Emergent branching structures in random geometry” (DOI: 10.55776/ESP534).

\section{Introduction}
\label{sec-intro}

The length of the longest increasing subsequence in a random permutation and the size of the largest homogeneous set (\textit{i.e.}\ a clique or an independent set) in a random graph are two of the classical problems at the interface of combinatorics and probability theory, with connections to several other areas of mathematics. 

In this paper, we investigate these classical problems in the setting of universal Brownian-type permutations and graphs, \textit{i.e.}\ for the \emph{Brownian separable permutons} \cite{bassino-separable-permuton,bbfgp-universal} and the \emph{Brownian cographons} \cite{bassino2022random,stufler2021graphon}. These objects are the universal limits of various random permutations and graph families.

In the following sections, we first discuss our results on permutations (\cref{sect:perm}) and then on graphs (\cref{sect:graphs}). In both sections, we briefly review the literature around the questions addressed in this paper. We then present two conjectures and some potential extensions of our work (\cref{sext:conj}). Here, we also discuss a few open problems and additional motivation for our work coming from random geometry and the study of planar maps.
Finally, we give an overview of the techniques used to establish our main results, explaining how they can potentially be used to answer similar questions in the continuum setting (\cref{sect:proof-tech}).

\subsection{Brownian separable permuton results}\label{sect:perm}

\subsubsection{Permutons and the Brownian separable permutons}\label{sect:broe_sep}

A Borel probability measure $\mu$ on the unit square $ [ 0,1 ] ^ 2$ is a \textbf{permuton} if both of its marginals are uniform, that is, $\mu([a,b]\times[0,1]) = \mu( [0,1] \times [a,b] )=b-a$ for all $0\le a < b\le 1$.
To a permutation $ \sigma $, we can associate a permuton $\mu_{\sigma}$ which is equal to $n$ times the Lebesgue measure on the union of the squares $\{[\frac{i-1}{n}, \frac{i}{n}]\times  [\frac{\sigma(i)-1}{n}, \frac{\sigma(i)}{n}]: i \in [n]\}.$ 
For a sequence of permutations $\{\sigma_n\}_{n\in \BB N}$, we say that $\sigma_n$ \textbf{converges in the permuton sense} to a limiting permuton $\mu$ if the permutons $\mu_{\sigma_n}$ converge weakly to $\mu$. The set of permutons equipped with the topology of weak convergence of measures is a compact metric space. The theory of permutons has seen many recent developments at the interface between discrete mathematics, probability theory, and statistics, see for instance, \cite{grubel2022ranks} for a survey. 

\medskip

We now recall the construction, due to Maazoun \cite{maazoun-separable-permuton}, of the Brownian separable permuton with parameter $p\in[0,1]$ in terms of a Brownian excursion with i.i.d.\ coin flips. Under $\bbP$, we call \textbf{signed excursion} a pair $(\efrak, \sfrak, p)$ consisting of a Brownian (normalized) excursion $\efrak$, together with an independent sequence $\sfrak$ of i.i.d.\ $p$-coins $\{\oplus,\ominus\}$, \textit{i.e.}\ $\bbP(\oplus)=p=1-\bbP(\ominus)$. 
One should think of the sequence $\sfrak$ as being indexed by the local minima of $\efrak$.\footnote{For the technicalities involved in indexing an i.i.d.\ sequence by	this random countable set, see \cite[Section 2.2]{maazoun-separable-permuton}.}

We define the following random relation $\vartriangleleft_{\efrak, \sfrak, p}$ on $[0,1]$:
conditional on $\efrak$, if $x,y\in[0,1]$, with $x<y$, and $\min_{[x,y]}\efrak$ is reached at a unique point which is a strict local minimum $\ell_{x,y}\in(x,y)$ then

\begin{equation}\label{eq:exc_to_perm}
	\begin{cases}
		x\vartriangleleft_{\efrak, \sfrak, p} y, &\quad\text{if}\quad \sfrak(\ell_{x,y})=\oplus,\\
		y\vartriangleleft_{\efrak, \sfrak, p} x, &\quad\text{if}\quad \sfrak(\ell_{x,y})=\ominus.\\
	\end{cases}
\end{equation}
Standard properties of the Brownian excursion ensure the existence of a random subset $ \mcl R_\efrak \subset [0,1]$ such that the complement has a.s.\ Hausdorff dimension $1/2$ -- \textit{i.e.}\ $\bbP(\dim([0,1]\setminus\mcl R_\efrak)=1/2)=1$, where $\dim(\cdot)$ denotes the Hausdorff dimension of a set -- and for every $x,y\in  \mcl R_\efrak$ with $x<y$, $\min_{[x,y]}\efrak$ is reached at a unique point which is a strict local minimum. In particular, the restriction of $\vartriangleleft_{\efrak, \sfrak, p}$ to $\mcl R_\efrak$ is a total order (see \cref{lem:prop_support_perm} below).
Setting
\begin{equation}\label{eq:level_function_sep}
	\psi_{\efrak, \sfrak, p}(t)\coloneqq\Leb\left( \big\{x\in[0,1]|x \vartriangleleft_{\efrak, \sfrak, p} t\big\}\right),\quad \forall t\in[0,1],
\end{equation}
then the (biased) \textbf{Brownian separable permuton} is the push-forward of the Lebesgue measure on $[0,1]$ via the mapping $(\mathbb{I},\psi_{\efrak, \sfrak, p})$, where $\mathbb{I}$ denotes the identity. That is,
\begin{equation}\label{eq:sep_perm}
	\bm{\mu}_p(\cdot)\coloneqq (\mathbb{I},\psi_{\efrak, \sfrak, p})_{*}\Leb(\cdot)=\Leb\left(\{t\in[0,1]|(t,\psi_{\efrak, \sfrak, p}(t))\in \cdot \,\}\right).
\end{equation}
Heuristically, $\psi_{\efrak, \sfrak, p}$ is the \emph{continuum permutation} of the elements in the interval $[0,1]$ induced by the order $\vartriangleleft_{\efrak, \sfrak, p}$ and $ \bm{\mu}_p$ is the diagram of $\psi_{\efrak, \sfrak, p}$. We stress that $\bm{\mu}_p$ is a \emph{random} permuton.

\medskip

The Brownian separable permutons were first introduced while studying random separable permutations \cite{bassino-separable-permuton}, \textit{i.e.}\ permutations avoiding the patterns $2413$ and $3142$ (see the Wikipedia page on separable permutations for further details and many properties of these permutations). The authors of the latter paper showed that uniform random separable permutations converge in distribution to $ \bm\mu_{1/2}$ when the size of the permutations tends to infinity. Later, it has been proved that this convergence result is in some sense universal: uniform permutations in proper substitution-closed classes \cite{bbfgp-universal, bbfs-tree-sep} or classes having a finite combinatorial specification for the substitution decomposition \cite{bassino2022scaling} converge in distribution (under some technical assumptions) to $ \bm\mu_{p}$, where the parameter $p$ depends on the chosen class.
These papers initiated a line of research around random Brownian fractal-type permutons, see for instance \cite{borga-skew-permuton,bgs-meander}.

\subsubsection{The length of the longest increasing subsequence}\label{sect:lis}

There is a vast literature devoted to the asymptotic behavior of the length of the longest increasing subsequence $\LIS(\sigma_n)$ for various types of large random permutations $\sigma_n$. 
For uniform random permutations $\sigma_n$, the study of $\LIS (\sigma_n)$ was initiated in the 1960s by Ulam~\cite{Ulam-lis}. In this case, one has $\LIS (\sigma_n) \sim 2\sqrt n$~\cite{Hammersley-lis,Vervsik-lis,Logan-lis,Aldous-lis}. 
The strongest known result is due to Dauvergne and Vir\'{a}g~\cite{Dauvergne-lis}, who showed that the scaling limit of the longest increasing subsequence in a uniform permutation is the \emph{directed geodesic} of the \emph{directed landscape}. 
The study of $\op{LIS}(\sigma_n)$ is connected with many other problems in combinatorics and probability theory, such as last passage percolation and random matrix theory; see the book of  Romik~\cite{Romik-lis} for an overview. 
In recent years, many extensions beyond uniform permutations have been considered, for instance:
\begin{itemize}
	\item when $\sigma_n$ is a uniform pattern-avoiding permutation \cite{Deutsch-lis,Madras-lis,Mansour-lis,bassino-lis};
	\item when $\sigma_n$ follows the \emph{Mallows distribution}, or is a product of such random permutations \cite{Starr-lis-mallows,Bhatnagar-lis-mallows,Basu-lis-mallows,zhong2023length};
	\item when $\sigma_n$ is conjugacy-invariant with few cycles, in particular \emph{Ewens-distributed}~\cite{Kammoun-lis};
	\item when $\sigma_n$ is sampled\footnote{A precise definition of what it means to sample a permutation from a measure is given in \cref{sect:main_res_perm}.} from a probability measure of the unit square $[0,1]^2$ having a density which satisfies certain regularity/divergence conditions \cite{deu-lim,deuschel1999increasing,dubach2023locally}. 
	\item when $\sigma_n$ is sampled from a random Brownian-type permuton, such as the Brownian separable permutons \cite{bassino-lis} or the \emph{skew Brownian permutons} \cite{bgs-meander}.
\end{itemize}

To the best of our knowledge, the only papers proving non-trivial power-law bounds for the length of the longest increasing subsequence are the recent work of Dubach \cite{dubach2023locally} and the works on the Mallows model of Bhatnagar and Peled, and Zhong \cite{Bhatnagar-lis-mallows,zhong2023length}.
Dubach \cite{dubach2023locally} built a family of permutons $\mu_\alpha$ for $1/2 < \alpha < 1$, with a density satisfying certain types of divergence, and which have the interesting property that a sequence of random permutations sampled from $\mu_\alpha$ has a longest increasing subsequence with growth rate equivalent to $n^\alpha$.  On the other hand, the authors of \cite{Bhatnagar-lis-mallows,zhong2023length} looked at random permutations $\sigma_n$ distributed according to the Mallows distribution with respect to various distances. They showed that rescaling the so-called \emph{scale parameter} $\beta$ with $n$ in a specific way (made explicit in the papers), one obtains that $\LIS(\sigma_n)$ is of order $n^\alpha$ for some $1/2 < \alpha < 1$ (and in some cases they prove exact limit theorems).

\subsubsection{Main results}\label{sect:main_res_perm}

Given a permuton $\mu$, sample $n$ independent points $Z_1, \dots , Z_n$ in the unit
square $[0, 1]^2$ according to $\mu$. These $n$ points induce a random permutation $\sigma$: for any $i,j\in[n]:=\{1,\dots,n\}$, let $\sigma(i) = j$ if the point with $i$-th lowest $x$-coordinate has $j$-th lowest $y$-coordinate (this is well-defined since the marginals of a permutons are uniform and so almost surely there are no points with the same $x$- or $y$-coordinates). We denote this permutation by $\Perm(\mu,n)$ and call it the \textbf{random permutation induced by the permuton $\mu$} of size $n$.  

This definition can be naturally extended to the case of random permutons. For more details, see \textit{e.g.}\ \cite[Section 2.1]{borga-thesis}. It is important to note that $\Perm(\mu,n)$ converges in distribution in the permuton sense to $\mu$ when $n$ tends to infinity, as shown in \cite[Lemma 2.3]{bbfgp-universal}.

\medskip

Recall that $\LIS(\cdot)$ denotes the length of the longest increasing subsequence in a permutation. Let $p\in(0,1)$ and recall that $\bm{\mu}_p$ is the Brownian separable permuton of parameter $p$. We restrict our analysis to the case $p\in(0,1)$, since the cases when $p=0$ and $p=1$ are degenerate: the Brownian separable permutons $\bm\mu_0$ (resp.\ $\bm\mu_1$) is the Lebesgue measure on the decreasing (resp.\ increasing) diagonal of the unit square. 

We are interested in studying $\LIS(\Perm(\bm{\mu}_p,n))$. See \cref{fig-simulations} for some simulations.
The results of \cite{bassino-lis} show that the following convergence holds in probability for all $p\in(0,1)$,
\begin{equation*}
	\frac{\LIS(\Perm(\bm{\mu}_p,n))}{n}\to 0.
\end{equation*}
It is simple to show that $\LIS(\Perm(\bm{\mu}_p,n))$ is bounded below by $\sqrt n$ with high probability (see the discussion below Theorem 1.7 in \cite{bassino-lis}). Our first main result shows that $\LIS(\Perm(\bm{\mu}_p,n))$ has an asymptotic behavior that is strictly different from the two bounds above. 

\begin{thm}\label{thm:upper_lower_perm}
	There exist two explicit functions $\alpha_* :(0,1)\to(1/2,1)$ and $\beta^* :(0,1)\to(1/2,1)$ such that for all $p\in(0,1)$,
	\begin{itemize}
		\item $1/2<\alpha_*(p)\leq\beta^*(p)<1$;
		\item for each $\alpha < \alpha_*(p)$ and each $\beta > \beta^*(p)$, it holds with probability tending to one as $n\to\infty$ that 
		\begin{equation*}
			n^\alpha \leq \LIS(\Perm(\bm{\mu}_p,n)) \leq n^\beta.
		\end{equation*}
	\end{itemize}	
\end{thm}

See \cref{fig-Upper-lower-bound2} for a plot of the graphs of the two functions $\alpha_*(p)$ and $\beta^*(p)$ from \cref{thm:upper_lower_perm} and a table of some of their values.  
See also \cref{rem:expr-bounds} at the end of this section for explicit formulas for these two functions.

\begin{figure}[ht!]
	\begin{minipage}{8.7cm}
		\includegraphics[width=\textwidth]{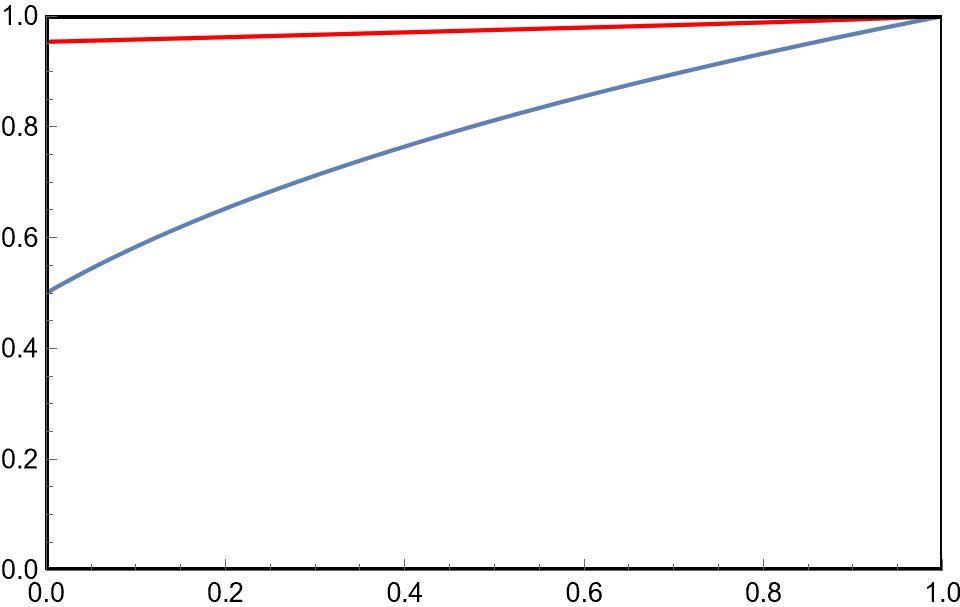}
	\end{minipage}
	\hspace{0.2cm}
	\begin{minipage}{9cm}
		\begin{tabular}{ |P{2cm}||P{2cm}|P{2cm}| }
			\hline
			\multicolumn{3}{|c|}{Numerical values for $\alpha_*(p)$ and $\beta^*(p)$} \\
			\hline
			$p$ & $\alpha_*(p)$ & $\beta^*(p)$\\
			\hline
			0.1 & 0.584 & 0.959\\
			0.2 & 0.653 & 0.963\\
			0.3 & 0.712 & 0.967\\
			0.4 & 0.765 & 0.971\\
			0.5 & 0.812 & 0.975\\
			0.6 & 0.855 & 0.980\\ 
			0.7 & 0.895 & 0.985\\
			0.8 & 0.932 & 0.991\\ 
			0.9 & 0.967 & 0.996\\
			\hline
		\end{tabular}
	\end{minipage}
	\caption{\label{fig-Upper-lower-bound2} \textbf{Left:} The plot of our bounds from \cref{thm:upper_lower_perm} as functions of $p\in(0,1)$: $\alpha_*(p)$ is in blue and $\beta^*(p)$ is in red. \textbf{Right:} Some numerical values of the bounds $\alpha_*(p)$ and $\beta^*(p)$.}
\end{figure}

We expect that $\LIS(\Perm(\bm{\mu}_p,n)) = n^{d(p) + o(1)}$ with probability tending to 1 as $n\to\infty$, for some exponent $d(p)\in[\alpha_*(p),\beta^*(p)]$. Numerical simulations suggest that $d(p)$ should be very close to $\alpha_*(p)$ for all $p\in(0,1)$. See \cref{conj:first-conj} below for a more precise statement and \cref{sect:num-sim} for more details on the numerical simulations.

We also expect that it is possible to transfer the bounds of \cref{thm:upper_lower_perm} to uniform separable permutations, see the text just before \cref{conj:sep-and-cograph} for further discussion.

\cref{thm:upper_lower_perm} shows that the exponent $d(p)$ is strictly bigger than $1/2$ and strictly smaller than $1$ for all $p\in(0,1)$. Note that the results (and techniques) in \cite{bassino-lis} are insufficient to establish either of the two bounds. Indeed, their results show that $\LIS(\Perm(\bm{\mu}_p,n))$ is sublinear in $n$ but, for instance, they do not exclude a potential behavior such as $n/\log(n)$.

Note also that our results give the first proof that the growth rate of $\LIS(\Perm(\bm{\mu}_p,n))$ depends on $p$, at least to some extent. This is because $\alpha_*(p) > \beta^*(\wt p)$ when $p$ is close to 1 and $\wt p$ is close to zero. 

We conclude this section by giving the explicit description of the functions $\alpha_*(p)$ and $\beta^*(p)$. 

\begin{remark}\label{rem:expr-bounds}
	For all $p\in(0,1)$,
	\begin{equation*}\label{eq:alpha_lam}
		\alpha_*(p)=1-\lstar(p),
	\end{equation*}
	where $\lstar(p)$ is the only positive solution (see \cref{eq:lapl_reg} for further explanations) to the equation $\Phi^\sfS(-\lstar(p)) = -2(1-p) \sqrt{\frac{2}{\pi}}$, with 
	\[\Phi^\sfS(q)=\int_0^{\infty} (1-\mathrm{e}^{-qx})(\mathds{1}_{x\in (0, \log 2]}+p\mathds{1}_{x\in (\log 2, +\infty)}) \frac{2\,\mathrm{e}^x\mathrm{d}x}{\sqrt{2\pi(\mathrm{e}^x-1)^3}}.\]
	Note that $\alpha_*(0^+) = 1/2$ and $\alpha_*(1^-) = 1$, and $\alpha_*(p)$ is strictly increasing in $p\in(0,1)$. In the symmetric case ($p=1/2$), we get $\alpha_*(1/2) \approx 0.812$.
	On the other hand, for all $p\in(0,1)$,
	\begin{equation*}
		\beta^*(p)=1-\lambda^*(p),
	\end{equation*}
	where $\lambda^*(p)=\sup_{\beta\in(0,\log(2)),\delta>0}\min\left\{\beta\delta\,,\,\sup_{\gamma<0}\{\gamma \delta + \kappa^{*}_{\gamma,\mathrm{e}^{-\beta}}(p)\}\right\}$ and $\kappa^{*}_{\gamma,r}(p)$ is the only \emph{positive} solution (see the discussion below \eqref{eq: sec 4 def kappa} for further explanations) to the equation 
	\begin{equation}\label{eq:expr_kappa}
		\Phi(-\kappa^{*}_{\gamma,r}(p))-2 (1-p)(1-\mathrm{e}^{\gamma})
		\sqrt{\frac{2}{\pi}}\frac{r^{-1}-2}{\sqrt{r^{-1}-1}}=0,
	\end{equation} 
	where $\Phi(q) = 2\sqrt{2} \frac{\Gamma(q+1/2)}{\Gamma(q)}$. 
	See \cref{sect:proof-tech} for some explanations on the origin of the latter expressions. 
\end{remark}

\subsection{Brownian cographon results}\label{sect:graphs}

\subsubsection{Graphons and the Brownian cographons}\label{sect:grphons}

A \textbf{graphon} is an equivalence class of measurable functions $W:[0, 1]^2 \to \{0, 1\}$ which are symmetric (\textit{i.e.}\ $W(x,y)=W(y,x)$ for all $x,y\in[0,1]$), under the equivalence relation $\sim$, where $W \sim U$ if there exists an invertible, measurable, Lebesgue measure preserving function $\phi : [0, 1] \to [0, 1]$ such that $W(\phi(x), \phi(y)) = U(x, y)$ for almost every $x, y \in [0, 1].$

Intuitively, a graphon is a continuous analog of the adjacency matrix of a graph, viewed up to relabeling its continuous vertex set. To every graph $G$ with $n$ labeled vertices, one can naturally associate a corresponding graphon:
\begin{align*}
	W_G:  &[0, 1]^2 \to \{0, 1\},\\
	&(x,y)\mapsto\,\,\,A_{\lceil nx \rceil,\lceil ny \rceil},
\end{align*}
where $A$ is the adjacency matrix of the graph $G$. Note that any relabeling of the vertices of $G$ gives the same graphon $W_G$, and so the definition extends to unlabeled graphs.
It is possible to define the so-called \emph{cut metric}, first on functions and then on graphons. The cut metric induces a notion of convergence for graphons (and so for graphs). Roughly speaking, the graphon convergence is the convergence of the rescaled adjacency matrix with respect to the cut metric. Graphon convergence has been first studied in \cite{borgs2008convergent} and developed into a vast topic in graph combinatorics, see
\cite{lovasz2012large} for an overview of this field of research.

\medskip

Given the signed excursion $(\efrak, \sfrak,p)$ introduced in \cref{sect:broe_sep}, the \textbf{Brownian cographon} $\bm{W}_p$ of parameter $p\in[0,1]$ is defined (following \cite{bassino2022random}) as the equivalence class of random functions 
\begin{align}\label{eq:Brownian-cog}
	\bm{W}_p: &[0, 1]^2 \to \{0, 1\},\\
	&(x,y)\mapsto\,\,\,\mathds{1}_{\sfrak(\ell_{x,y})=\oplus}, \notag
\end{align}
where, as in \cref{sect:broe_sep}, if $(x,y)\in\mcl R_\efrak^2$ and $x<y$, we denote by $\ell_{x,y}\in[x,y]$ the unique strict local minimum $\min_{[x,y]}\efrak$. If $x$ or $y$ is not in $\mcl R_\efrak$ then we arbitrarily set $\sfrak(\ell_{x,y})=\oplus$. This choice does not change the law of $\bm{W}_p$ because a.s.\ $\Leb([0,1]^2\setminus \mcl R_\efrak^2)=0$. We stress the fact that $\bm{W}_p$ is a \emph{random} graphon.

\medskip

A \textbf{cograph} is a graph avoiding the path with four vertices as induced subgraph (see the Wikipedia page for several other equivalent characterizations of cographs and their computational properties).
The Brownian cographon $\bm W_{1/2}$ has been proven to be the limit in the graphon sense of uniform random cographs when the number of vertices tends to infinity
\cite{bassino2022random,stufler2021graphon}. Some first universality results for the Brownian cographons have been recently established in \cite{lenoir2023graph}, but we expect more to come.

\subsubsection{The size of the largest homogeneous set}\label{sect:lhs}

An \textbf{independent set} of a graph $G$ is a subset of its vertices such that every two distinct vertices in the subset are not adjacent, while a \textbf{clique} of a graph $G$ is a subset of its vertices such that every two distinct vertices in the subset are adjacent. For a graph $G$, a subset of its vertices is called \textbf{homogeneous} if it is either an independent set or a clique. The \emph{Erd\H{o}s--Hajnal conjecture} \cite{erdos1977spanned,erdos1989ramsey} states
that every \textbf{$H$-free} graph of size $n$ (\textit{i.e.}\ a graph avoiding a given subgraph $H$ as induced subgraph) has a homogeneous set of polynomial size\footnote{We recall that every graph of size $n$ has a homogeneous set of size $\log(n)$, and this is optimal up to a constant. Determining the optimal constant $C$ such that every graph of size $n$ has a homogeneous set of size at least $C \log(n)$ is equivalent to the computation of diagonal Ramsey numbers, see \cite{campos2023exponential} for the best-known bounds}. in $n$, \textit{i.e.}\ of size $c\cdot n^{\alpha(H)}$ for some $\alpha(H)\in(0,1]$. This conjecture is still open, see \cite{chudnovsky2014erdos} for a survey and \cite{bucic2023loglog} for the best-known bound. 

We emphasize that cographs play a key role in this conjecture. Indeed, since a homogeneous set in a graph induces a cograph, and all cographs of size $n$ have a homogeneous set of size at least $\sqrt n$ \cite{erdos1989ramsey}, then the Erd\H{o}s--Hajnal conjecture takes the following equivalent form: Every $H$-free graph of size $n$ contains a cograph of polynomial size as an induced subgraph. In fact, this reformulation is one of the most classical ways to attack the conjecture, see for instance \cite{erdos1989ramsey}.

Our motivations for studying homogeneous sets in the Brownian cographons are multiple. Specifically, they come from the above reformulation of the Erd\H{o}s--Hajnal conjecture, from its probabilistic version \cite{loebl2010almost,kang2014most} discussed below, from the recent graphon convergence results towards the Brownian cographons mentioned earlier, and from the recent developments in \cite{bassino-lis}, also explained below.

The authors of \cite{loebl2010almost,kang2014most} established that for a large family of graphs $H$, a uniformly random $H$-free graph with $n$ vertices has with high probability a homogeneous set of \emph{linear size}. We highlight that here the homogeneous set has linear size and not only polynomial size as in the original version of Erd\H{o}s--Hajnal conjecture (\textit{i.e.}\ $\alpha(H)$=1).
When this holds, the graph $H$ is said
to have the \textbf{asymptotic linear Erd\H{o}s-Hajnal property} (see \cite{kang2014most} for a precise definition).
\cite[Section 5]{kang2014most} asked whether a uniform random cograph with $n$ vertices  has the asymptotic linear Erd\H{o}s-Hajnal property.\footnote{Indeed, cographs do not fit into the results of \cite{loebl2010almost,kang2014most}.} This question was answered negatively in \cite[Theorem 1.2]{bassino-lis}, where it was shown that the maximal size of a homogeneous set $\LHS(G_n)$ in a uniform random cograph $G_n$ with $n$ vertices converges to zero in probability when divided by $n$.
The authors left open the question of finding the exact order of magnitude of $\LHS(G_n)$ (see \cite[Remark 1.5]{bassino-lis}), pointing out that $\sqrt{n}$ is a trivial lower bound. Also in this setting, the results of \cite{bassino-lis} do not exclude a potential behavior such as $n/\log(n)$. The latter behavior is not entirely unexpected. Indeed, as pointed out in \cite[Section 1.2]{kang2014most}, $n/\log(n)$ is the asymptotic behavior for $\LHS(\widetilde G_n)$ when $\widetilde G_n$ is a uniform graphs avoiding the path  $P_3$ with three vertices (recall that cographs are graph avoiding the path $P_4$ with four vertices), and for instance graphs avoiding $P_3$ and $P_4$ have the same (exceptional) \emph{coloring number}\footnote{The \emph{coloring number} of a graph (see \cite[Definition 2]{loebl2010almost}) should not be confused with its \emph{chromatic number}.} $2$; see \cite[Section 1, p.\ 4]{loebl2010almost} for further details.

Our results in the next section imply explicit power-law bounds for the analog of $\alpha(H)$ for graphs sampled from the Brownian cographons (\cref{thm:upper_lower_graphons} and \cref{cor:upper_lower_graphons}), and we expect that the same bounds can be proven for uniform cographs (\cref{sext:conj}). The latter result would distinguish the behavior of $\LHS(\cdot)$ on uniform cographs and uniform graphs avoiding $P_3$.

\subsubsection{Main results}\label{sect:main_res_graphons}

Given a graphon $W$, we can consider the \textbf{random graph induced by $W$} of size $n$,
denoted by $\Graph(W,n)$ and defined as follows: consider $n$ vertices
$\{v_1, v_2, \dots, v_n\}$ and, let $(U_1, \ldots , U_n)$ be $n$ i.i.d.\ uniform random variables in $[0, 1]$.
We connect the vertices $v_i$ and $v_j$ with an edge if and only if $W(U_i,U_j)=1$.
This definition can be naturally extended to the case of random graphons (see \cite[Section 3.2]{bassino2022random} for further details). Note that $\Graph(W,n)$ converges in distribution to $W$ when $n$ tends to infinity, as shown in \cite[Lemma 3.9]{bassino2022random}.

\medskip

Recall that $\LHS(\cdot)$ denotes the size of the largest homogeneous set in a graph. We also denote by $\LIN(\cdot)$ the size of the largest independent set in a graph and by $\LCL(\cdot)$ the size of the largest clique in a graph (the latter two quantities are usually denoted by $\alpha(\cdot)$ and $\omega(\cdot)$ in the literature, but we preferred to adopt a different notation since it is more consistent with the one used in \cref{sect:perm} for permutations). See \cref{fig-simulations} for some simulations of the largest homogeneous set in $\Graph(\bm{W}_p,n)$.

\begin{thm}\label{thm:upper_lower_graphons}
	Let $p\in(0,1)$ and $\bm{W}_p$ be the Brownian cographon of parameter $p$. Let $\alpha_*(p)$ and $\beta^*(p)$ be as in \cref{thm:upper_lower_perm}. Then 
	\begin{itemize}
		\item for each $\alpha < \alpha_*(p)$ and each $\beta > \beta^*(p)$ it holds with probability tending to one as $n\to\infty$ that 
		\begin{equation*}
			n^\alpha \leq \LCL(\Graph(\bm{W}_p,n)) \leq n^\beta;
		\end{equation*}
		\item for each $\overline\alpha < \alpha_*(1-p)$ and each $\overline\beta > \beta^*(1-p)$ it holds with probability tending to one as $n\to\infty$ that 
		\begin{equation*}
			n^{\overline\alpha} \leq \LIN(\Graph(\bm{W}_p,n)) \leq n^{\overline\beta}.
		\end{equation*}
	\end{itemize} 
\end{thm}

In this paper, we will prove \cref{thm:upper_lower_perm} and derive \cref{thm:upper_lower_graphons} from it (see \cref{sect:sampling} for further explanations). We point out that one can equivalently directly prove \cref{thm:upper_lower_graphons} and then derive \cref{thm:upper_lower_perm}, \textit{i.e.} the two theorems are equivalent. We opted for the first strategy since we found slightly simpler to phrase some combinatorial constructions in term of permutations rather than graphs.

As in the case of \cref{thm:upper_lower_perm}, we expect that it is possible to transfer the bounds in \cref{thm:upper_lower_graphons} to uniform cographons. See the text just before \cref{conj:sep-and-cograph} for further discussion.

We have the following immediate consequence of \cref{thm:upper_lower_graphons}.

\begin{cor}\label{cor:upper_lower_graphons}
	Let $p\in(0,1)$ and $\bm{W}_p$ be the Brownian cographon of parameter $p$. Let $\alpha_*(p)$ and $\beta^*(p)$ be as in \cref{thm:upper_lower_perm}, then for each $\widetilde\alpha < \alpha_*(\widetilde p)$ and each $\widetilde\beta > \beta^*(\widetilde p)$ with $\widetilde p=\max\{p,1-p\}$, it holds with probability tending to one as $n\to\infty$ that 
	\begin{equation*}
		n^{\widetilde\alpha} \leq \LHS(\Graph(\bm{W}_p,n)) \leq n^{\widetilde\beta}.
	\end{equation*}
\end{cor}

Note that our results in \cref{thm:upper_lower_graphons} also complement the results of McKinley \cite{mckinley2019superlogarithmic}, where the author constructed certain graphons $W_{\alpha}$ such that the size of the largest clique $\LCL(\Graph(W_\alpha,n))$ behaves like $n^\alpha$ for all $\alpha\in(0,1)$.

\subsection{Conjectures and potential extensions}\label{sext:conj}

The results in Theorems \ref{thm:upper_lower_perm}
and \ref{thm:upper_lower_graphons} establish upper and lower bounds for the polynomial growth of the longest increasing subsequence in permutations sampled from the Brownian separable permutons and for the largest clique and independent set in graphs sampled from the Brownian cographons. We expect the existence of a deterministic \emph{critical exponent} for these quantities (\emph{c.f.}\ \cref{lem:perm_to_graph}).

\begin{conj}\label{conj:first-conj}
	For all $p\in(0,1)$ there exists $d(p)\in[\alpha_*(p),\beta^*(p)]$ such that  
	with probability tending to 1 as $n\to\infty$
	\begin{equation*}
		\LIS(\Perm(\bm{\mu}_p,n))=n^{d(p)+o(1)},
	\end{equation*}
	\begin{equation*}
		\LCL(\Graph(\bm{W}_p,n))=n^{d(p)+o(1)}\qquad \text{and}\qquad \LIN(\Graph(\bm{W}_p,n))=n^{d(1-p)+o(1)}.
	\end{equation*}
\end{conj} 

As already mentioned, numerical simulations suggest that $d(p)$ should be very close to $\alpha_*(p)$ for all $p\in(0,1)$. See \cref{sect:num-sim} for more details. Various heuristic arguments  indicate that our lower bound $\alpha_*(p)$ should \emph{not} be sharp for all $p\in(0,1)$, \textit{i.e.} it should hold that $\alpha_*(p)<d(p)$  for all $p\in(0,1)$.

\medskip

We expect that our estimates for longest increasing subsequences and largest homogeneous sets from Theorems~\ref{thm:upper_lower_perm} and~\ref{thm:upper_lower_graphons} in the case when $p=1/2$ can be transferred to the setting of uniform separable permutations and uniform cographs, respectively. To explain how this might be accomplished, we recall that both uniform separable permutations and uniform cographs can be encoded by random walk excursions with i.i.d. steps and some collections of random signs $\oplus/\ominus$, in a similar way as the Brownian separable permutons and the Brownian cographons can be encoded by a signed Brownian excursion. 

The \emph{KMT coupling theorem} \cite{komlos1975approximation,zaitsev-kmt} states that one can construct a random walk with i.i.d.\ steps and a standard Brownian motion on the same probability space in such a way that with high probability, the difference of their values at each time $n$ is $O(\log n)$. To transfer our bounds to the setting of uniform separable permutons and uniform graphons, a natural approach would be to couple the encoding walks for (the infinite-volume version of) the discrete models with the encoding Brownian motion for (the infinite-volume version of) the Brownian models via KMT. One could then use the coupling to transfer estimates for Brownian motion to estimates for random walk, and finally use local absolute continuity arguments to transfer from an unconditioned walk to a random walk excursion. We point out that a similar approach has been used to prove estimates for random planar maps in~\cite{ghs-map-dist}.  
However, we expect the above KMT coupling argument to require a fair amount of technical work, so we do not carry it out in this paper. 

It would also be interesting to deal with the discrete models themselves to come up with power-law bounds (or, even better, with the exact exponents) for the length of the longest increasing sequence in uniform separable permutations and the size of the largest homogeneous set in uniform cographs.

In fact, we expect an even stronger relationship between discrete and continuum models. 

\begin{conj}\label{conj:sep-and-cograph}
	Let $d(1/2)\in[\alpha_*(1/2),\beta^*(1/2)]$ be as in \cref{conj:first-conj}. Let $\sigma_n$ and $G_n$ be a uniform separable permutation and a uniform separable cograph of size $n$, respectively. Then the following holds with probability tending to 1 as $n\to\infty$:
	\begin{equation*}
		\LIS(\sigma_n)=n^{d(1/2)+o(1)}\qquad\text{and}\qquad \LHS(G_n)=n^{d(1/2)+o(1)}.
	\end{equation*}
\end{conj}

It also seems possible that our methods could be extended to some models of uniform permutations in substitution-closed classes which are well-behaved with respect to the substitution decomposition \cite{bbfgp-universal,bbfs-tree-sep} and to uniform graphs in some classes of graphs which are closed under the substitution operation at the core of the modular decomposition \cite{bassino2022random}.

\subsubsection{Skew Brownian permuton and random planar maps}\label{sect:more-mot}

We now explain certain connections with other models of random permutations, with random geometry, and with the study of planar maps. 

Recently, the first author of this paper constructed in \cite{borga-skew-permuton} a two-parameter family of universal random permutons $\bm{\mu}_{\rho,q}$  -- indexed by $(\rho,q)\in (-1,1]\times[0,1]$ and called \emph{skew Brownian permutons} -- and showed that they are the limits of various models of random permutations \cite{bm-baxter-permutation,borga-strong-baxter}. Additionally, skew Brownian permutons are connected with multiple models of decorated planar maps and SLE-decorated Liouville quantum gravity spheres (see \cite[Section 1.5]{borga-skew-permuton}). 

As shown in \cite[Theorem 1.12]{borga-skew-permuton}, the Brownian separable permutons $\bm{\mu}_p$ studied in this paper coincide with the skew Brownian permutons $\bm{\mu}_{1,1-p}$. The skew Brownian permuton $\bm{\mu}_{-1/2,1/2}$ is the \emph{Baxter permuton} \cite{bm-baxter-permutation}, \textit{i.e.}\ the permuton limit of uniform Baxter permutations. Additionally, in \cite[Section 1.6, Item 6]{borga-skew-permuton} it is conjectured that  $\bm{\mu}_{-1,q}:=\lim_{\rho\to-1}\bm{\mu}_{\rho,q}$ are the Mallows permutons \cite{starr2009thermodynamic, MR3817550}, \textit{i.e.}\ the permuton limits of Mallows permutations. In particular, $\lim_{\rho\to-1}\bm{\mu}_{\rho,1/2}$ should be the uniform Lebesgue measure on $[0,1]^2$, \textit{i.e.}\ the permuton limit of uniformly random permutations.\footnote{Recall from \cite{starr2009thermodynamic} that the Mallows permutons $(\nu(\beta))_{\beta\in \mathbb{R}}$ describe the permuton limit of $q$-Mallows distributed permutations of size $n$ when $q$ scales as $1-\frac{\beta}{n}$. To the best of our knowledge, there are no available conjectures for the exact relation between the $\beta$-parameter for the Mallows permutons $\nu(\beta)$ and the $q$-parameter for the permutons $\bm{\mu}_{-1,q}:=\lim_{\rho\to-1}\bm{\mu}_{\rho,q}$, apart from the specific case $\bm{\mu}_{-1,1/2}=\nu(0)$, where we recall that $\nu(0)$ coincides with the uniform Lebesgue measure on $[0,1]^2$.} See \cref{fig-skew-perm} for a schematic summary.

Combining the results of this paper with the fact that it is known that the length of the longest increasing subsequence in Mallows permutations behaves\footnote{Note that this result is true in the regime when Mallows permutations exhibit a nontrivial permuton limit (\textit{i.e.} when $q$ scales as $1-\frac \beta n$)}, which is the regime of interest to us. like $\sqrt n$ \cite{Starr-lis-mallows}, we propose the following conjecture. Recall the exponent $d(p)$ from \cref{conj:first-conj}.

\begin{conj}\label{conj:random-geom}
	For all $(\rho,q)\in [-1,1]\times(0,1)$, let $\bm{\mu}_{\rho,q}$ be the skew Brownian permuton. There exists a function $\ell(\rho,q):[-1,1]\times(0,1)\to[1/2,1)$ such that with probability tending to 1 as $n\to \infty$,  
	\begin{equation*}
		\LIS(\Perm(\bm{\mu}_{\rho,q},n))=n^{\ell(\rho,q)+o(1)} .
	\end{equation*} 
	Moreover, $\ell(-1,q)=1/2$ for all $q\in (0,1)$, $\ell(1,q)=d(1-q)\in[\alpha_*(1-q),\beta^*(1-q)]$ for all $q\in (0,1)$, and $\ell(\rho,q)$ is continuous, non-increasing in $q$ and non-decreasing in $\rho$.
\end{conj}

We point out that in \cite[Corollary 1.13]{bgs-meander} it was shown that $\LIS(\Perm(\bm{\mu}_{\rho,q},n))$ is sublinear for all $(\rho,q)\in (-1,1)\times(0,1)$.
The exponents $\ell(\rho,q)$ of \cref{conj:random-geom} should also be related to certain directed metrics on random planar maps, see \cite[Remark 1.15]{bgs-meander} for further details. In \cref{fig-skew-perm}, we summarize various models of random permutations and random planar maps which are connected with the skew Brownian permutons.

\begin{figure}[ht!]
	\begin{center}
		\includegraphics[width=1\textwidth]{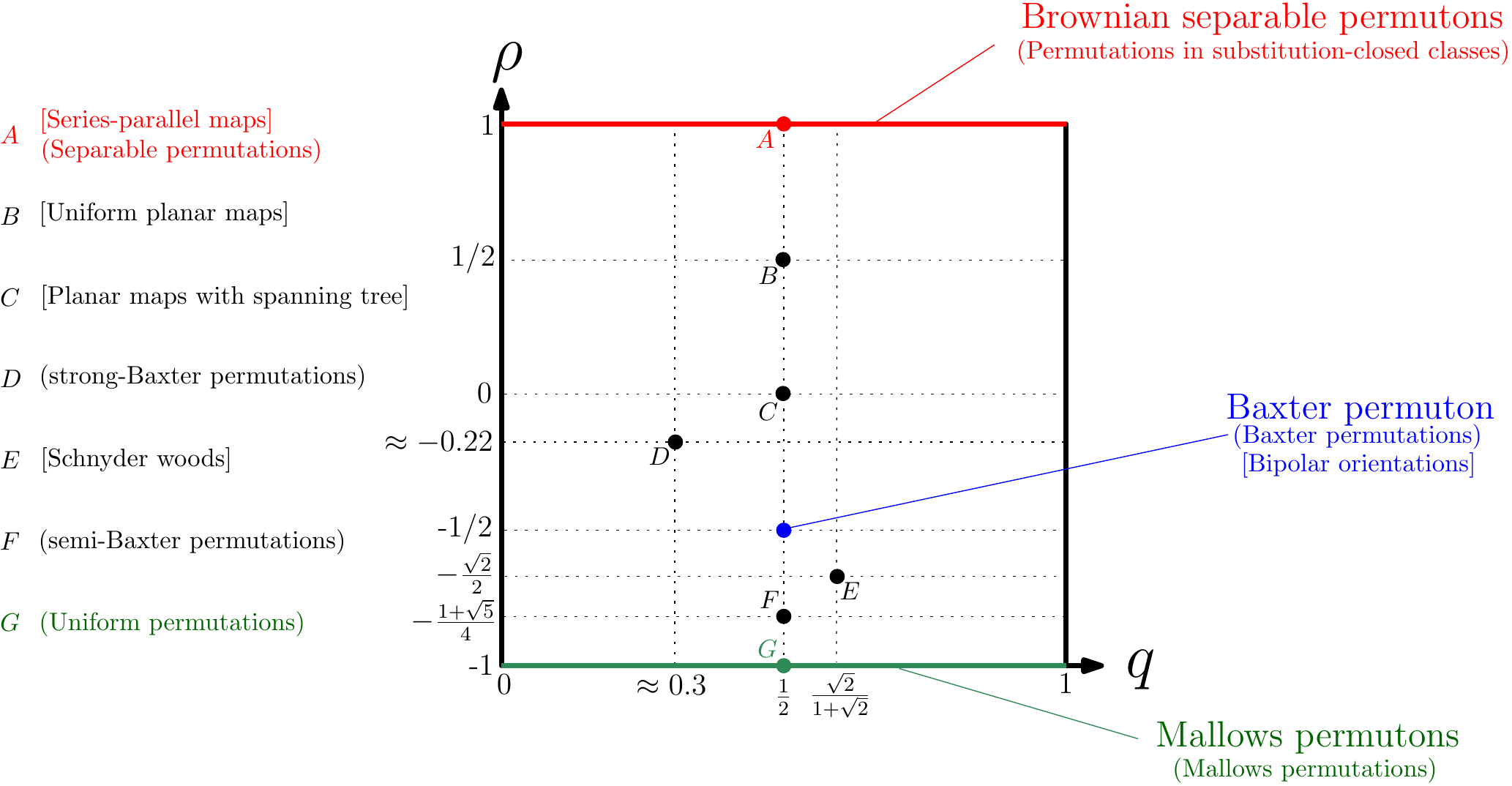}
		\caption{\label{fig-skew-perm} The diagram shows the range of values  $(\rho,q)\in [-1,1]\times(0,1)$ for the parameters of the skew Brownian permutons. At the bottom, in green, we have the Mallows permutons. At the top, in red, we have the Brownian separable permutons. In blue, close to the center, we have the Baxter permuton. Various models of random permutations which are known to converge in the permuton sense to the skew Brownian permutons are indicated between rounded parentheses. Finally, various models of planar maps which are connected to the skew Brownian permutons are indicated between squared parentheses.}
	\end{center}
	\vspace{-3ex}
\end{figure}

\subsection{Proof techniques and ordered subsets of the signed Brownian excursion}\label{sect:proof-tech} 

Recall that we will derive \cref{thm:upper_lower_graphons} from \cref{thm:upper_lower_perm}, so here we focus on the latter theorem. Our results in \cref{thm:upper_lower_perm} follow -- after some non-trivial arguments in Sections \ref{sec: lower bound seq permuton} and \ref{sect:up_bound_seq_perm} -- from some preliminary estimates on the probability of certain events (introduced in the next two sections) related to the signed Brownian excursion $(\efrak,\sfrak,p)$.

\subsubsection{Strategy for the proof of the lower bound}\label{sec: strategy lower bound}

We say that a subset $O\subset[0,1]$ is \textbf{ordered w.r.t.}\  $\vartriangleleft_{\efrak, \sfrak, p}$ if the usual order on $O$ coincides with $\vartriangleleft_{\efrak, \sfrak, p}$ on $O\cap \mcl R_\efrak$.

To obtain the lower bound in \cref{thm:upper_lower_perm}, we first define in \eqref{eq:selection} a selection rule $\sfS$ which determines a large subset $\calS\subset[0,1]$ ordered w.r.t.\  $\vartriangleleft_{\efrak, \sfrak, p}$. Here, \emph{large} refers to a good notion of size, which we could, \textit{e.g.}, take to be Hausdorff dimension.

To define the selection rule $\sfS$, we explore the signed excursion $(\efrak,\sfrak,p)$ by heights.
We call \textbf{branching height} a height which corresponds to a local minimum of $\efrak$: a branching height is said to be \textbf{positive} or \textbf{negative} according to the sign $\oplus$ or $\ominus$ of the corresponding symbol in $\sfrak$. We denote by $\mcl B_{\oplus}$ and $\mcl B_{\ominus}$ the set of positive and negative branching heights respectively, and we set $\mcl B=\mcl B_{\oplus}\cup\mcl B_{\ominus}$. Each branching height $b=\efrak(t_b)$ results in one sub-excursion of $\efrak$ over an interval $(a,c)\subset [0,1]$ splitting into two sub-excursions over the intervals $(a,t_b)$ and $(t_b,c)$ for some $t_b\in(a,c)$. See the red excursions and red intervals in \cref{fig:excursion-split}.

\begin{figure}[ht!]
	\begin{center}
		\includegraphics[width=.85\textwidth]{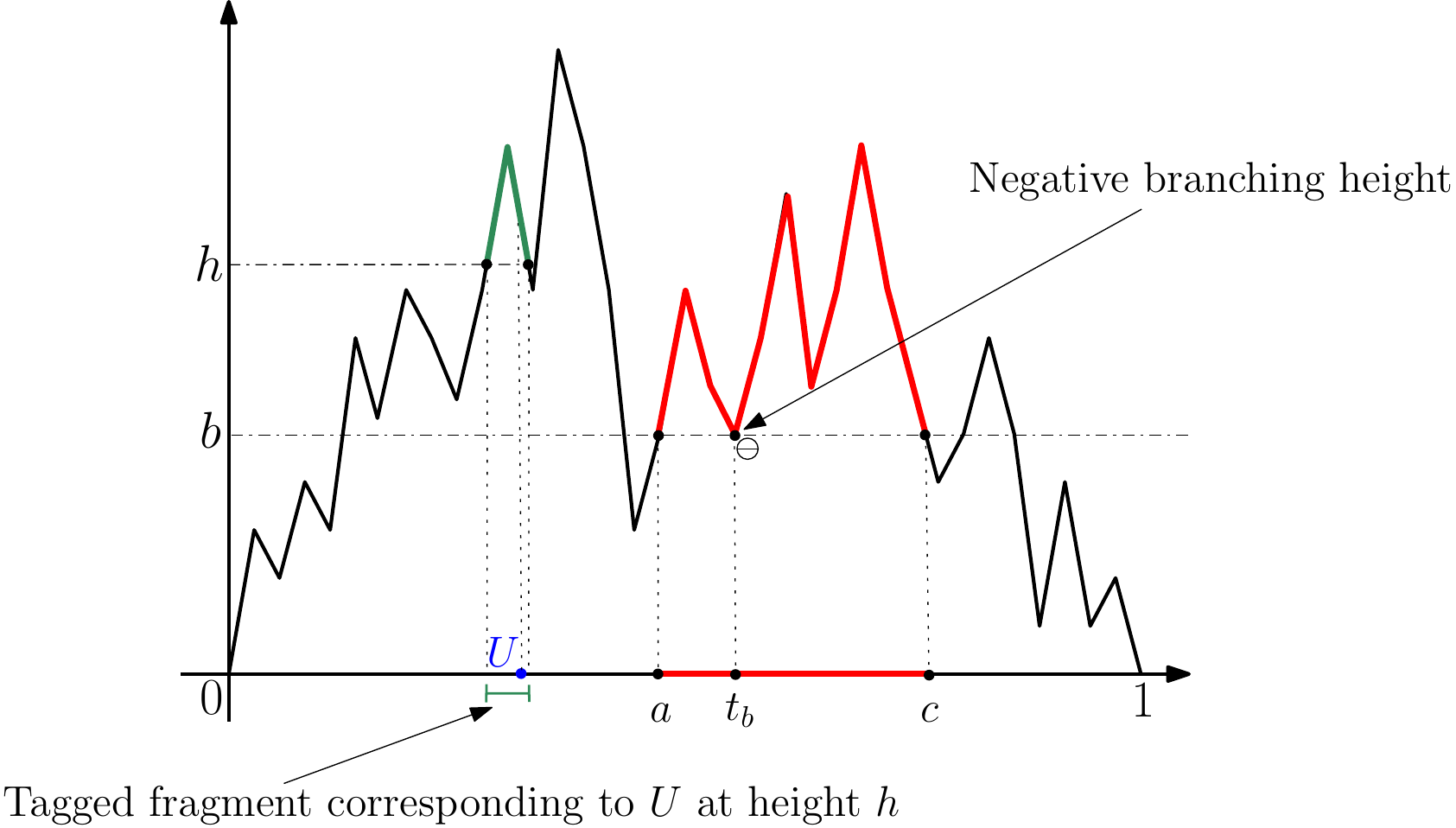}  
		\caption{\label{fig:excursion-split} A sketch for the notation introduced in \cref{sec: strategy lower bound}.}
	\end{center}
	\vspace{-3ex}
\end{figure}

Note that the \emph{optimal} selection rule to find the largest ordered subset of $[0,1]$ w.r.t.\  $\vartriangleleft_{\efrak, \sfrak, p}$ can be informally read as follows. Whenever we hit a positive branching height, we keep both intervals $(a,t_b)$ and $(t_b,c)$ and carry on the exploration in both components of $\efrak$. On the other hand, when reaching a negative branching height, one needs to discard one of the intervals $(a,t_b)$ and $(t_b,c)$ (because in this case, by definition \eqref{eq:exc_to_perm}, $y \vartriangleleft_{\efrak, \sfrak, p} x$ for almost all $x\in(a,t_b)$ and $y\in(t_b,c)$), keeping only the one with the largest ordered subset.

The main issue with the previous considerations is the heavily non-Markovian nature of the selection rule at negative branching heights: in order to pick one of $(a,t_b)$ or $(t_b,c)$ over the other, one \textit{a priori} needs to look into the whole future and see which one of them has the largest ordered subset. Our approach is to replace this selection rule by a Markovian rule. More precisely, we consider the following \textbf{selection rule} $\sfS$: 
\begin{align}\label{eq:selection}
	&\text{ ``Whenever reaching a negative branching height, we discard the smaller}\\ 
	&\quad\,\,\text{(in terms of Lebesgue measure) interval between $(a,t_b)$ and $(t_b,c)$.''}\notag
\end{align}

We then want to show that the selection rule $\sfS$ does not discard too many points, in a sense made precise in the next two paragraphs.

One essential tool in our proof is the analysis of a natural \emph{fragmentation process} \cite{bertoin-fragbook}, or more precisely of a \emph{self-similar interval fragmentation process}, embedded in the Brownian excursion $\efrak$ (see \eqref{eq: def Ffrak} below for a precise definition). A key player in the description of such processes is the so-called \textbf{tagged fragment}, which in our case, consists in looking at the duration of the excursion at height $h\geq 0$ straddling a uniform point $U$ in $(0,1)$. See the green excursion and green interval in \cref{fig:excursion-split}. \cite[Section 4]{bertoin-ssfrag} describes the law of the tagged fragment as a stochastic process with time parameter $h\geq 0$: it is an explicit \emph{positive self-similar Markov process with index $-\frac12$}, see \cref{prop: law tagged} for a precise statement. 

\begin{remark}
	The specific expressions for our bounds $\alpha_*(p)$ and $\beta^*(p)$ in \cref{rem:expr-bounds} are related to the expression of some Laplace exponents of the fragmentation process mentioned above, see for instance \eqref{eq:Phi^S} below.
\end{remark}

In order to show that the selection rule $\sfS$ does not discard too many points, we proceed as follows:
\begin{enumerate}[(a)]
	\item we first estimate the probability that the tagged fragment corresponding to $U$ survives \emph{long enough}, in the sense that it reaches some small value $\eps$ before it (possibly) gets discarded by $\sfS$. In \cref{prop: estimate killing}, we show that the latter probability is asymptotically $c\,\eps^{\lstar(p)}$, for some positive constant $c>0$ (recall from \cref{rem:expr-bounds} that $\lambda_*(p)=1-\alpha_*(p)$).
	\item Then in \cref{prop: two point function estimates}, we extend this estimate to a so-called \emph{two-point function estimate}: we show that the probability that two tagged fragments corresponding to two independent uniform points $U$ and $V$ in $(0,1)$ both survive until they get smaller than $\eps$ is asymptotically $c'\,\eps^{2\lstar(p)}$, for some (other) positive constant $c'>0$.
\end{enumerate}	
The above two estimates turn out to be enough to then deduce the lower bound in \cref{thm:upper_lower_perm}, as shown in \cref{sec: lower bound seq permuton}.

\begin{remark}\label{remark:haus-lower}
	One can also introduce the set of survival times
	\begin{equation*}\label{eq: def calS}
		\calS := \{t\in[0,1], \; t \; \text{is not in any sub-interval which is discarded by }\sfS \}.
	\end{equation*}
	Plainly, $\calS$ is a random set of fractal type and whose law depends on the parameter $p\in(0,1)$ of the signs. Moreover, by construction, we have the following immediate result: Almost surely, $\calS$ is totally ordered w.r.t.\ $\vartriangleleft_{\efrak, \sfrak, p}$.
	
	We believe that the first moment estimate in Item $(a)$ above, together with an upgraded version of the two-point function estimate in Item $(b)$ above and some energy method arguments (see for instance \cite[Theorem 4.27]{morters-peres}), would be enough to prove the following result: 
	
	\emph{
	For all $p\in(0,1)$, almost surely, the Hausdorff dimension of $\calS$ is 
	\[\dim (\calS)=\dim(\calS\cap \mcl R_\efrak)=\alpha_*(p),\]
	where $\alpha_*(p)$ is as in \cref{rem:expr-bounds}. 
	}

	We expect the upgraded version of the two-point function estimate mentioned above to require a fair amount of technical work, so we do not carry it out in this paper.
\end{remark}

\subsubsection{Strategy for the proof of the upper-bound}\label{sec: strategy upper bound}

Our strategy to prove the upper bound in \cref{thm:upper_lower_perm} is in some sense to analyze the \emph{worst-case scenario}.

Recall the notation from the previous section.  We pick (in a manner which is allowed to depend on the signed excursion $(\mathfrak e,\mathfrak s,p)$, plus possibly some additional independent information) an arbitrary subset $O\subset [0,1]^2$ which is ordered w.r.t.\  $\vartriangleleft_{\efrak, \sfrak, p}$. 
For each branching height $b\in \mcl B_{\ominus}$, let $\eta_b$ be the \emph{discarding rule} corresponding to the set $O$, that is, $\eta_b$ is equal to $L$ (for \emph{left}) if $O$ intersects $(t_b,c)$ and $\eta_b$ is equal to $R$ (for \emph{right}) if $O$ intersects $(a,t_b)$ (we make an arbitrary choice if $O$ intersects neither of the two intervals). Our goal is to show that $(\eta_b)_{b\in \mcl B_{\ominus}}$ discards a large subset of $[0,1]$ so that the complement is small. To do that, we upper-bound as $\eps\to 0$ the following probability
\[\bbP\left((\eta_b)_{b\in\mcl B_{\ominus}}\text{ does not discard the fragment corresponding to } U \text{ before it gets smaller than } \eps\right),\]
where $U$ is a uniform point in $[0,1]$ independent from everything else. In \cref{prop:est_up_bound_lis(q)}, we show that the above probability is upper-bounded by $c\cdot\eps^{\lambda^*(p)}$ as $\eps\to 0$ (recall from \cref{rem:expr-bounds} that $\lambda^*(p)=1-\beta^*(p)$). This is done in two main steps: 
\begin{itemize}
	\item Fix $r\in (1/2,1)$. We say that a branching height $b\in \mcl B_{\ominus}$ whose corresponding excursion interval $(a,c)$ contains $U$ is \emph{balanced} (at scale $r$) if $\max\{t_b-a,c-t_b\}\leq r (c-a)$, \textit{i.e.}\ the sub-excursions to the left and right of the local minimum time $t_b$ are of comparable time duration (recall \cref{fig:excursion-split}). The first step, which is carried out in \cref{prop:first step2}, is to upper-bound the probability that $(\eta_b)_{b\in\mcl B_{\ominus}}$ does not discard $U$ before reaching the $m$-th smallest balanced branching height. In particular, we will show that this probability is upper-bounded by $r^m$.
	\item The second step, which is carried out in \cref{prop:second step2}, is to upper-bound the probability that the fragment corresponding to $U$ at the $m$-th balanced branching height has duration smaller than $\eps$. 
\end{itemize}
We then optimize over the parameters $r$ and $m$. 

Finally, we transfer this estimate to the discrete setting in \cref{sect:up_bound_seq_perm} (obtaining the upper bound in \cref{thm:upper_lower_perm}), as follows.
First, we sample a permutation $\Perm(\bm{\mu}_p,n)$ from the Brownian separable permuton $\bm{\mu}_p$. Then, we consider the discarding rule $(\eta_b)_{b\in \mcl B_{\ominus}}$ corresponding to the points of the longest increasing subsequence in $\Perm(\bm{\mu}_p,n)$ (if there are multiple ones, we choose one arbitrarily). That is, at each branching height $b\in \mcl B_{\ominus}$ we discard the interval not containing points of the longest increasing subsequence (we make an arbitrary choice if the longest increasing subsequence intersects neither of the two intervals). We then conclude using the estimates above.

\begin{remark}
	As in Remark \ref{remark:haus-lower}, we expect that one can extend the arguments discussed above to obtain the following result:
		
	\emph{
	Fix $p\in(0,1)$. Almost surely, every set $O\subset [0,1]$ such that $O\cap \mcl R_\efrak$ is totally ordered w.r.t.\  $\vartriangleleft_{\efrak, \sfrak, p}$  satisfies
	\[\dim(O)\leq\beta^*(p),\]
	where $\beta^*(p)$ is as in \cref{rem:expr-bounds}.
	}
		
	However, there are some technical arguments involved in transferring from an estimate for a uniform time $U\in [0,1]$ to a bound for Hausdorff dimension, so for the sake of brevity we do not prove the above statement in this paper.
\end{remark}

\section{Preliminary results}

In this section we gather some preliminary results that are used later in the paper. 
In \cref{sect:prel_res} we discuss some classical properties of Brownian excursions.
In \cref{sect:sampling} we explain an equivalent way to sample permutations and graphs from the Brownian separable permutons and the Brownian cographs and how to derive \cref{thm:upper_lower_graphons} from \cref{thm:upper_lower_perm}.
Finally, in \cref{sect:frah}, we properly introduce fragmentation processes and the tagged fragment in a Brownian excursion.

\subsection{Brownian excursions}\label{sect:prel_res}

We collect some standard properties of a Brownian excursion $\efrak$ and the order $\vartriangleleft_{\efrak, \sfrak, p}$ introduced in \eqref{eq:exc_to_perm} that will be used frequently (and sometimes tacitly) in the paper.

Conditioning on $\efrak$, we say that $t\in[0,1]$ is \textbf{regular} for $\efrak$ if it is \emph{not} a one-sided local minimum of $\efrak$, that is, $t\in[0,1]$ is regular for $\efrak$ if
\[
\forall \eps>0, \, \exists\, t_1\in(t-\eps,t), t_2\in(t,t+\eps) \text{ such that } \efrak(t_1)<\efrak(t) \; \text{and} \; \efrak(t_2)<\efrak(t).
\]
We denote by $\mcl R_\efrak$ the (random) set of regular points of $\efrak$.

\begin{lem}\label{lem:prop_support_perm}
	Fix $p\in(0,1)$. Let $\efrak$ be a (normalized) Brownian excursion and $\vartriangleleft_{\efrak, \sfrak, p}$ be the relation introduced in \eqref{eq:exc_to_perm}. Then, almost surely, the following assertions hold.
	\begin{enumerate}[(a)]
		\item All local minima of $\efrak$ are strict local minima (hence countable), they all have different heights, and they are dense in $[0,1]$.
		\item Local minima never split the excursion into two sub-excursions with equal durations.
		\item For every $x,y\in  \mcl R_\efrak$ with $x<y$, it holds that $\min_{[x,y]}\efrak$ is reached at a unique point which is a strict local minimum.
		\item The relation $\vartriangleleft_{\efrak, \sfrak, p}$ restricted to $\mcl R_\efrak$ is a total order.
		\item $\dim([0,1]\setminus \mcl R_\efrak)=1/2$, where $\dim(\cdot)$ denotes the Hausdorff dimension of a set.
	\end{enumerate}
\end{lem}

\begin{proof}
	The first two items are classical properties of Brownian excursions (see for instance \cite[Chapter XII]{revuz2013continuous}). Then Item (c) follows from the previous properties and the definition of regular points. 
	
	Item (d) is \cite[Lemma 2.5]{maazoun-separable-permuton}. Note that the latter lemma states our claim only for a set of Lebesgue measure one (instead of $\mcl R_\efrak$), but the proof of the lemma proves exactly our claim.
	
	Finally, we prove Item (e). This result is classical in the probabilistic literature, and we include a proof only for the sake of completeness. Let $t$ be a time which is not regular for $\efrak$. Then either there is a rational $q < t$ such that $\efrak$ attains a running minimum at time $t$ when run forward started from time $q$; or there exists a rational $q > t$ such that $\efrak$ attains a running minimum at time $t$ when run backward started from time $q$. For each rational time $q$, the set of times at which $\efrak$ attains a running minimum when run forward (resp.\ backward) from time $q$ has a.s.\ Hausdorff dimension 1/2 (by local absolute continuity between the Brownian excursion and Brownian motion). Therefore the result follows from the countable stability of Hausdorff dimension. 
\end{proof}

\subsection{Sampling permutations and graphs from the Brownian separable permutons and the Brownian cographs}\label{sect:sampling}

Recall from \cref{sect:main_res_perm} that given the Brownian separable permuton $\bm{\mu}_p$, the permutation $\Perm(\bm{\mu}_p,n)$ is obtained as follows: conditioning on $\bm{\mu}_p$,  sample $n$ independent points $Z_1, \dots , Z_n$ in the unit
square $[0, 1]^2$ with distribution $\bm{\mu}_p$. These $n$ points induce a random permutation $\sigma_n=\Perm(\bm{\mu}_p,n)$: for any $i,j\in[n]:=\{1,\dots,n\}$, let $\sigma_n(i) = j$ if the point with $i$-th lowest $x$-coordinate has $j$-th lowest $y$-coordinate. It is simple to realize that the previous permutation can be equivalently obtained as follows: sample $n$ independent uniform points $(U_i)_{i\leq n}$ on $[0, 1]$. Then $\sigma_n$ is the permutation induced by the order of the points $(U_i)_{i\leq n}$ with respect to the order $\vartriangleleft_{\efrak, \sfrak, p}$ introduced in \eqref{eq:exc_to_perm}. Equivalently, for $i<j$,
\begin{eqnarray*}
	\sigma_n(i)<\sigma_n(j) \quad \text{if and only if} \quad \sfrak(\ell_{\bar{U}_i,\bar{U}_j})=\oplus,
\end{eqnarray*}
where $(\bar{U}_i)_{i\leq n}$ is the re-arrangement of $(U_i)_{i\leq n}$ in increasing order and $\ell_{\bar{U}_i,\bar{U}_j}\in[\bar{U}_i,\bar{U}_j]$ is the unique strict local minimum $\min_{[\bar{U}_i,\bar{U}_j]}\efrak$. Recall from \cref{lem:prop_support_perm} that almost surely, for every $x,y\in  \mcl R_\efrak$ with $x<y$, the minimum $\min_{[x,y]}\efrak$ is reached at a unique point which is a strict local minimum and $\mcl R_\efrak$ has Lebesgue measure one, so the previous quantities are almost surely well-defined. Therefore, when convenient, we will denote $\sigma_n=\Perm(\bm{\mu}_p,n)$ also by $\sigma_n=\Perm(\efrak,\sfrak,(U_i)_{i\leq n})$.

Recall now that given the Brownian cographon $\bm{W}_p$ introduced in \eqref{eq:Brownian-cog}, we can consider the random graph induced by $\bm{W}_p$ of size $n$,
denoted by $\Graph(\bm{W}_p,n)$ and defined as follows: consider $n$ vertices
$\{v_1, v_2, \dots, v_n\}$, and let $(U_1, \ldots , U_n)$ be $n$ i.i.d.\ uniform random variables in $[0, 1]$, independent of $\bm{W}_p$.
We connect the vertices $v_i$ and $v_j$ with an edge if and only if $\bm{W}_p(U_i,U_j)=1$. Equivalently, from the definition in \eqref{eq:Brownian-cog}, we connect the vertices $v_i$ and $v_j$ with an edge if and only if $\sfrak(\ell_{U_i,U_j})=\oplus$. 

By comparing the above descriptions of $\Perm(\bm{\mu}_p,n)$ and $\Graph(\bm{W}_p,n)$ and noting that $\bm{\mu}_{1-p}$ has the same law as $\bm{\mu}_{ p}$ when we exchange all the $\oplus$ and $\ominus$ signs in the collection of signs $\sfrak$, we immediately obtain the following result.

\begin{lem} \label{lem:perm_to_graph}
	Fix $p\in(0,1)$. Let $\bm{\mu}_p$ be the Brownian separable permuton of parameter $p$ and let $\bm{W}_p$ be the Brownian cographon of parameter $p$. Then
	\begin{equation*}
		\LCL(\Graph(\bm{W}_p,n))\stackrel{d}{=}\LIS(\Perm(\bm{\mu}_p,n)) \quad \text{ and } \quad \LIN(\Graph(\bm{W}_p,n))\stackrel{d}{=}\LIS(\Perm(\bm{\mu}_{1-p},n)).
	\end{equation*}
\end{lem}

As a consequence of \cref{lem:perm_to_graph}, we see that \cref{thm:upper_lower_graphons} follows immediately from \cref{thm:upper_lower_perm}. We also recall that \cref{cor:upper_lower_graphons} is an immediate consequence of \cref{thm:upper_lower_graphons}.
Hence, in the rest of the paper we will focus on the proof of \cref{thm:upper_lower_perm}.

\subsection{Fragmentation processes and the tagged fragment in a Brownian excursion}\label{sect:frah}
As already mentioned in \cref{sect:proof-tech}, our point of view bears close connections with Bertoin's \textbf{fragmentation processes} \cite{bertoin-fragbook}, which in fact come in as one of the main tools for the derivation of the exponent bounds $\alpha_*(p)$ an $\beta^*(p)$. Such processes describe the behavior of a system of masses which fall apart randomly over time in a Markovian way. In this paper we will be interested in one particular example of so-called \textbf{self-similar interval fragmentation} $\Ffrak_0=(\Ffrak_0(h),h\ge 0)$ defined from a normalized Brownian excursion $\efrak$ as
\begin{equation} \label{eq: def Ffrak}
	\Ffrak_0(h):=\{s\in(0,1), \; \efrak(s)>h\}, \quad h\ge 0.
\end{equation}
See Figure \ref{fig-fragmentation} for an illustration. It is clear that $\Ffrak_0$ is a nested family of open sets in $(0,1)$. Moreover, Brownian scaling implies that $\Ffrak_0$ also enjoys a self-similarity property: for all $r>0$, the process $(r\Ffrak_0(r^{-1/2}h), h\ge 0)$ has the law of the process \eqref{eq: def Ffrak} defined from a Brownian excursion conditioned to have duration $r$. We will denote by $\Ffrak$ the collection of \emph{lengths} of intervals in $\Ffrak_0$. The process $\Ffrak$ has been first introduced in \cite{bertoin-ssfrag}, where it was proved to a be essentially a variant of the Aldous--Pitman fragmentation \cite{aldous-pitman}.

\begin{figure}[ht!]
	\begin{center}
		\includegraphics[width=.5\textwidth]{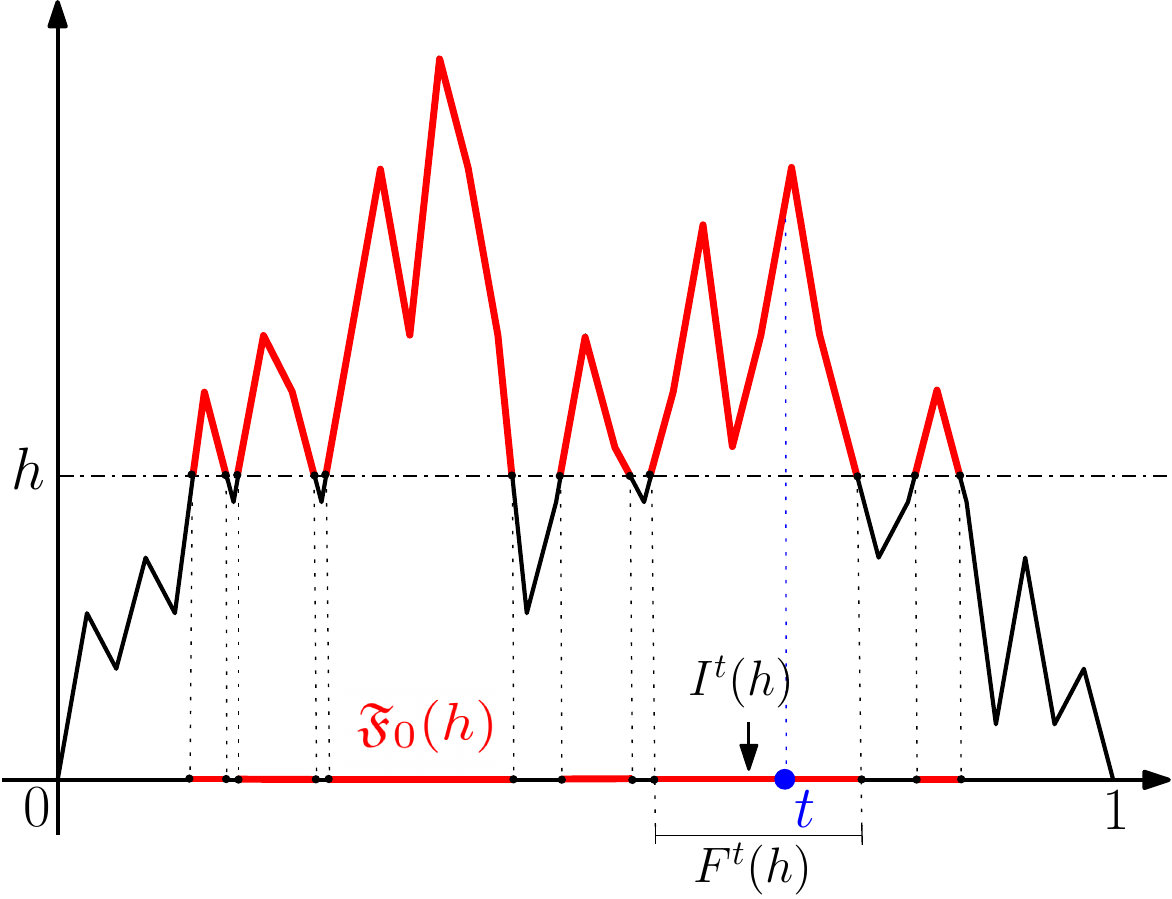}  
		\caption{\label{fig-fragmentation} A sketch for the notation introduced for the self-similar interval fragmentation $\Ffrak_0=(\Ffrak_0(h),h\ge 0)$. In red we highlighted the intervals in $\Ffrak_0(h)$ and the corresponding sub-excursions of $\efrak$ above level $h$. The interval $I^{t}(h)$ in $\Ffrak_0(h)$ containing $t$ has length $F^{t}(h)$ in $\Ffrak$.}
	\end{center}
	\vspace{-3ex}
\end{figure}

A key player in the description of such processes is the so-called \textbf{tagged fragment}, which in our case consists in targeting the fragment straddling a uniform point in $(0,1)$. For $t\in(0,1)$, we denote by $I^{t}(h)$, $h\ge 0$, the interval in $\Ffrak_0(h)$ which contains $t$, and we set $I^{t}(h)=\emptyset$ if $t$ is not contained in any interval of $\Ffrak_0(h)$ (equivalently, $\efrak(t) \leq h$). We further write $F^{t}(h)$ for the length of the interval $I^{t}(h)$, and $h^{t} = \efrak(t)$ for its absorption time at $0$, \textit{i.e.}\ the first height $h$ such that $I^{t}(h)=\emptyset$. The tagged cell for the fragmentation, or simply tagged fragment, is then given by 
\begin{equation} \label{eq: def tagged cell}
F^{U}=(F^{U}(h), 0\le h<h^{U}), \quad \text{where $U$ is uniform in $(0,1)$, independent from $\efrak$.}
\end{equation}
 Then \cite[Section 4]{bertoin-ssfrag} gives the law of $F^U$. Let us recall for completeness that if $(X,P_x)$ is a positive Markov process which under $P_x$ starts at $x>0$, we say that $X$ is \textbf{self-similar with index $\alpha$} if, for all $r>0$, the process $(rX(r^{\alpha}s), s\ge 0)$ under $P_x$ has the law $P_{rx}$.

\begin{prop}[{\cite[Section 4]{bertoin-ssfrag}}]\label{prop: law tagged}
	The process $F^{U}$ is a positive self-similar Markov process with index $-\frac12$. It can be further written in the Lamperti representation as 
	\[
	F^{U}(h) = \exp(-\xi_{\rho(h)}), \quad 0\le h<h^{U},
	\]
	where $\rho(h)$ is the (Lamperti) time-change
	\begin{equation}\label{eq:def_time_change}
	\rho(h) := \inf\left\{u>0, \; \int_0^{u} \mathrm{e}^{-\frac12\xi_r} \mathrm{d}r > h\right\}, \quad 0\le h<h^{U},
\end{equation}
	and $\xi$ is a subordinator with Laplace exponent 
	\begin{equation}\label{eq:def_phi}
		\Phi(q) := -\log \bbE[\mathrm{e}^{-q\xi_1}] = \int_0^\infty (1-\mathrm{e}^{-qx}) \frac{2\mathrm{e}^x}{\sqrt{2\pi(\mathrm{e}^x-1)^3}}\, \mathrm{d}x, \quad q>-\frac12.
	\end{equation}
	That is, $\xi$ has no killing, no drift, and Lévy measure $\Lambda(\mathrm{d}x) := \frac{2\mathrm{e}^x \mathrm{d}x}{\sqrt{2\pi(\mathrm{e}^x-1)^3}}$ on $(0,+\infty)$.
\end{prop}

\begin{remark}\label{rem:lap_exp_expl}
	One can calculate that $\Phi(q) = 2\sqrt{2} \frac{\Gamma(q+1/2)}{\Gamma(q)}$ (see \cite[Equation (12)]{bertoin-ssfrag}). Note also that $\Phi(q)$ is increasing in $q$, $\Phi(0)=0$, $\lim_{q\to (-1/2)^+}\Phi(q)=-\infty$ and $\lim_{q\to +\infty}\Phi(q)=+\infty$.
	The fact that $\Phi(q) = 2\sqrt{2} \frac{\Gamma(q+1/2)}{\Gamma(q)}$, together with \cite[Section 2.3]{krs-subordinators} for instance, implies that $F^{U}$ has the law of the $\frac12$--stable (decreasing) subordinator conditioned to be absorbed continuously at $0$.
\end{remark}

The reason why $F^{U}$ plays a special part in the description of $\Ffrak$ is that it governs the behavior of the size of a \emph{typical} fragment in $\Ffrak$.

We end this section with a technical lemma for subordinators which will be relevant in later sections.

\begin{lem} \label{lem: asymptotics subordinator}
	Let $\eta$ be a subordinator with Laplace exponent $\Psi(q) := -\log \bbE[\mathrm{e}^{-q\eta_1}]$, and $a \geq 0$. Assume that that there exists $a_\ast>0$ such that $\Psi$ extends (analytically) to a neighborhood to the left of $0$ containing $-a_\ast$, $\Psi(-a_\ast) = -a$ and $\Psi'(-a_\ast)<\infty$. For $x>0$, define the first passage time of $\eta$ across $x$,
	\[
	S_x := \inf\{s>0, \; \eta_s>x\}.
	\]
	Then for all $b\geq 0$, there exists a constant $c=c(a,b)\in (0,1)$ such that,
	\begin{equation} \label{eq: lem asymptotics subordinator equivalent}
		\bbE\left[ \mathrm{e}^{-a S_x - b \eta_{S_x}}\right] \underset{x\to \infty}{\sim} c\,\mathrm{e}^{-(a_\ast+b) x}. 
	\end{equation}
	Moreover, for all $b>-a_\ast$, we have the upper-bound
	\begin{equation} \label{eq: lem asymptotics subordinator upper bound}
		\bbE\left[\mathrm{e}^{-a S_x - b \eta_{S_x}}\right]  \le \mathrm{e}^{-(a_\ast+b) x}, \quad \text{for all } 			x>0.
	\end{equation}
\end{lem}

\begin{proof}
	The result is a consequence of exponential tilting. Since $\Psi(-a_\ast) = -a$, the process 
	\[
	M^a_s := \mathrm{e}^{a_\ast \eta_s - a s}, \quad s\ge0,
	\]
	is a martingale. Let $\bbP^a$ denote the tilted probability measure with respect to the martingale $M^a$. This is the measure defined from Kolmogorov's extension theorem by $\mathrm{d}\bbP^a = M^a_s \cdot \mathrm{d}\bbP$ on $\sigma(\eta_r,  r\le s)$ for all $s\ge 0$. Plainly, under $\bbP^a$, $\eta$ is a Lévy process with Laplace exponent $\Psi^a(q) = \Psi(q-a_\ast) + a$. 
	We claim that by an optional stopping type argument, we have 
	\begin{equation}\label{eq:_claim_ch_meas}
		\bbE\left[ \mathrm{e}^{-a S_x - b \eta_{S_x}}\right]
		=
		\bbE^a\left[ \mathrm{e}^{-(a_\ast+b) \eta_{S_x}}\right]\quad \text{if $b\geq 0$},
	\end{equation}
	and
	\begin{equation}\label{eq:_claim_ch_meas2}
		\bbE\left[ \mathrm{e}^{-a S_x - b \eta_{S_x}}\right]\leq\bbE^a\left[ \mathrm{e}^{-(a_\ast+b) \eta_{S_x}}\right] \quad \text{if $-a_\ast < b < 0$} ,
	\end{equation}
	where $\bbE^a$ denotes expectation with respect to $\bbP^a$.
	Indeed, since subordinators are transient (see for instance \cite[Chapter III]{bertoin1996levy}), $S_x$ is a.s.\ finite, whence a.s.\  
	\[
	\lim_{n\to \infty} \mathrm{e}^{-(a_\ast+b) \eta_{n\wedge S_x}} = \mathrm{e}^{-(a_\ast+b) \eta_{S_x}}
	\quad \text{and} \quad 
	\lim_{n\to \infty}\mathrm{e}^{-a (n\wedge S_x)-b\, \eta_{n\wedge S_x}}=\mathrm{e}^{-a  S_x-b\, \eta_{S_x}}.
	\]
	Moreover, since $a_\ast+b\geq 0$, by dominated convergence (the domination is straightforward once we remark that under $\bbP^a$, $\eta$ is still a subordinator, so that $\eta_{n\wedge S_x}\geq 0$),
	\begin{equation}\label{eq:first_step_lim}
		\bbE^a\left[ \mathrm{e}^{-(a_\ast+b) \eta_{S_x}}\right]=\lim_{n\to \infty}\bbE^a\left[ \mathrm{e}^{-(a_\ast+b) \eta_{n\wedge S_x}}\right]=\lim_{n\to \infty}\bbE\left[ \mathrm{e}^{-a (n\wedge S_x)-b\, \eta_{n\wedge S_x}}\right].
	\end{equation}
	For $a\geq 0$ and $b \geq 0$,
	another application of dominated convergence gives the claim in \eqref{eq:_claim_ch_meas}. For $a\geq 0$ and $-a_\ast < b < 0$, we note that by Fatou's lemma
	\[
	\bbE\left[ \mathrm{e}^{-a S_x - b \eta_{S_x}}\right]
	=
	\bbE\left[\liminf_{n\to \infty}\mathrm{e}^{-a (n\wedge S_x)-b\, \eta_{n\wedge S_x}} \right]\leq \liminf_{n\to \infty} \bbE\left[\mathrm{e}^{-a (n\wedge S_x)-b\, \eta_{n\wedge S_x}} \right]\stackrel{\eqref{eq:first_step_lim}}{=}\bbE^a\left[ \mathrm{e}^{-(a_\ast+b) \eta_{S_x}}\right],
	\]
	which is the claim in \eqref{eq:_claim_ch_meas2}.
	
	First of all, the inequality \eqref{eq: lem asymptotics subordinator upper bound} is a trivial consequence of \eqref{eq:_claim_ch_meas} and \eqref{eq:_claim_ch_meas2}, together with the observation that $\eta_{S_x}\ge x$. We now prove the claim in \eqref{eq: lem asymptotics subordinator equivalent}. Since by assumption $\bbE^a[\eta_1]=(\Psi^a)'(0) = \Psi'(-a_\ast) < \infty$, by the renewal theorem \cite[Theorem 1]{bvhs-renewal}, $\eta_{S_x} - x$ converges in distribution under $\bbP^a$ to a limiting non-degenerate random variable as $x\to\infty$. Therefore
	\[
	\bbE^a\left[ \mathrm{e}^{-(a_\ast+b) \eta_{S_x}}\right] \underset{x\to \infty}{\sim} c\,\mathrm{e}^{-(a_\ast+b) x},
	\]
	for some constant $c\in(0,1)$. The statement of the lemma then follows from \eqref{eq:_claim_ch_meas}.
\end{proof}
\begin{remark}
	The constant $c$ can be made explicit from \cite[Theorem 1]{bvhs-renewal}. 
\end{remark}

\section{Estimates for the lower bound}
\label{sec-sec2}

The main goal of this section is to provide first and second moment estimates involving our selection rule $\sfS$ defined in \eqref{eq:selection}, which will be later used in \cref{sec: lower bound seq permuton} to lower bound the length of the longest increasing subsequence in permutations sampled from the Brownian separable permutons. More precisely, we provide in \cref{prop: estimate killing} asymptotics as $\eps\to 0$ for the probability that the tagged fragment \eqref{eq: def tagged cell} survives (in the sense of $\sfS$) until getting smaller than $\eps$, and similar estimates in \cref{prop: two point function estimates} for the two-point function.

\subsection{Embedding the selection rule in the tagged cell}
\label{sec: first moment}

Recall from \cref{sec: strategy lower bound} the following setup. Every local minimum $t_b\in(0,1)$ of the Brownian excursion $\efrak$ corresponds to a so-called branching height $b = \efrak(t_b)$ and comes with a $\oplus$ or $\ominus$ sign given by $\sfrak$. We denote by $\cal B_{\oplus}$ and $\cal B_{\ominus}$ the sets of branching heights respectively associated with $\oplus$ and $\ominus$ signs. Each branching height $b$ is splitting one interval $(a,c)$ into the two sub-intervals $(a,t_b)$ and $(t_b,c)$ (see \cref{fig:excursion-split}, p.\ \pageref{fig:excursion-split}). Our selection rule $\sf S$ is to discard at each \emph{negative} branching height the smaller of the two intervals $(a,t_b)$ and $(t_b,c)$ in terms of Lebesgue measure.
In this paragraph we embed this strategy in the fragmentation $\Ffrak$ and the tagged fragment $F^{U}$ introduced in \eqref{eq: def tagged cell}. For a right-continuous non-negative process $X$, we introduce the notation $\Delta X(t):= X(t)-X(t^-)$ for its possible jump at time $t$. Recall that the branching heights of $\efrak$ are encoded by jumps in the fragmentation process $\Ffrak$. One can therefore enforce the selection rule $\sfS$ in the tagged fragment $F^{U}$ by killing it at the first \emph{negative} branching height $h$ in $(\efrak,\sfrak,p)$ where the size of the other fragment, which is $-\Delta F^{U}(h)$, is larger than $F^{U}(h)$; \emph{c.f.}\ Figure \ref{fig-selection-rule}. Moreover, note that in the notation of \cref{prop: law tagged}, whenever $F^{U}(h)<-\Delta F^{U}(h)$, then after time-change we have that $\Delta \xi_{\rho(h)} >\log 2$. Based on this, we set
\begin{equation} \label{eq: killing H^U}
	H^{U}_{\sfS}:=\inf\{h\in\mcl B_{\ominus}, \; -\Delta F^{U}(h)>F^{U}(h)\},
\end{equation}
and $\overline F^{U}(h):=F^{U}(h)\mathds{1}_{h<H^{U}_{\sfS}}$, $h\geq 0$.

\begin{figure}[ht!]
	\begin{center}
		\includegraphics[width=.65\textwidth]{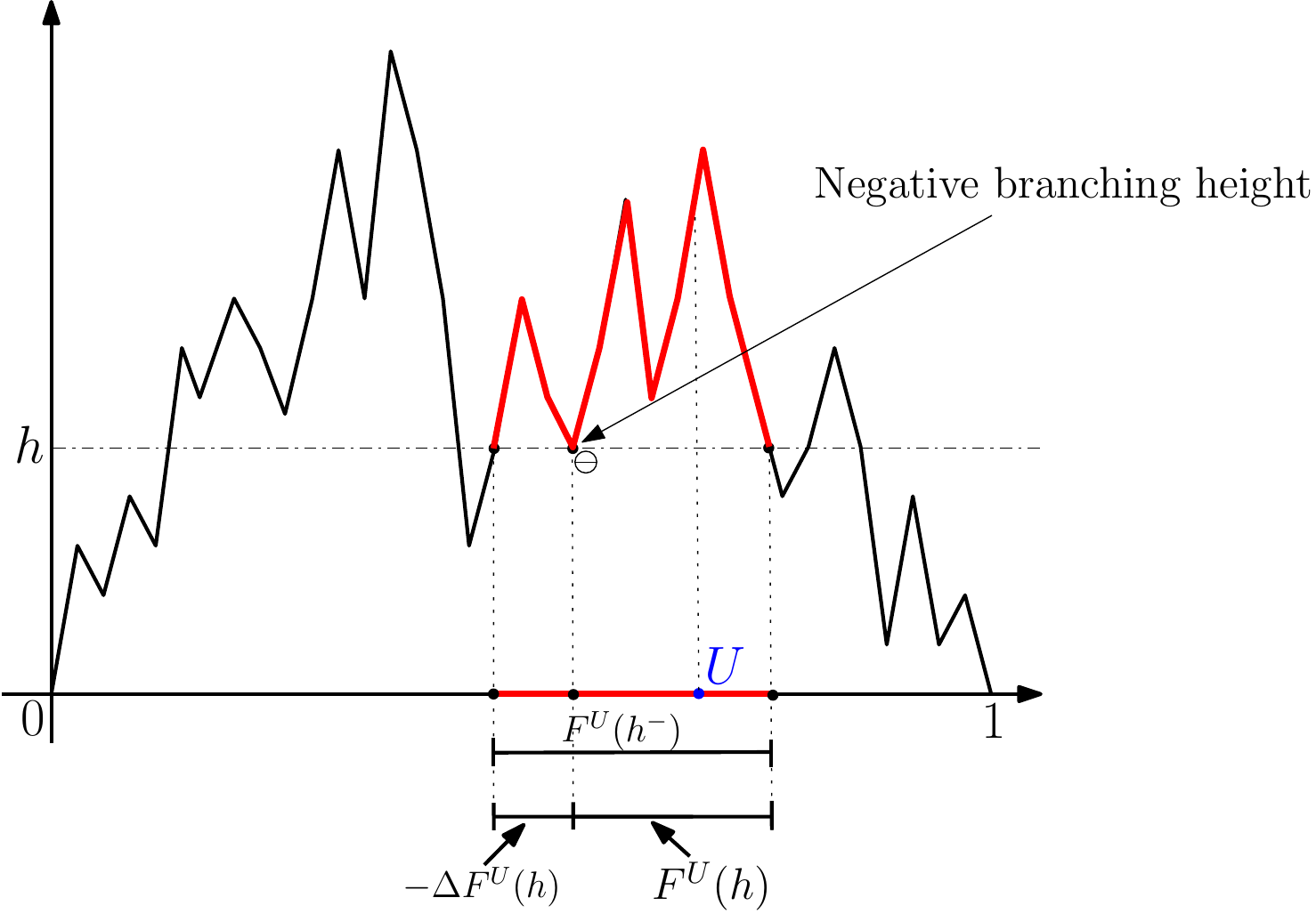}  
		\caption{\label{fig-selection-rule} A sketch explaining how we enforce the selection rule $\sfS$ in the tagged fragment $F^{U}$: at the first negative height when the fragment $-\Delta F^{U}(h)$ is larger than the tagged fragment $F^{U}(h)$, we kill the process $F^{U}$ at zero, \textit{i.e.}\ with the notation in \eqref{eq: killing H^U}, we are considering the process $\overline F^{U}(h)=F^{U}(h)\mathds{1}_{h<H^{U}_{\sfS}}$.}
	\end{center}
	\vspace{-3ex}
\end{figure}

This also motivates the introduction of two new processes $\chi$ and $\overline\xi$ describing the previous construction at the level of the Lévy process $\xi$ in Proposition \ref{prop: law tagged}. For any jump time $s>0$ of $\xi$, the variable $\chi_s$ encodes the $\oplus$ or $\ominus$ sign in $\sfrak$ attached to the branching height $\rho^{-1}(s)$ of $\efrak$. More precisely, conditional on $\xi$, for each $s>0$ such that $\Delta \xi_s>0$, let $\chi_s = 1$ or $\chi_s = 0$ with probability $p$ and $1-p$ respectively (for other times we send $\chi_s:=\lozenge$ to a cemetery state). We take the random variables $\{\chi_s  : \Delta \xi_s > 0\}$ to be conditionally independent given $\xi$.

One can then implement the strategy $\sfS$ on $\xi$, obtaining a new process $\overline\xi$ which is a killed version of $\xi$. Informally $\overline\xi$ is constructed as follows: let $s>0$ be a jump time for $\xi$. If $\chi_s=1$, we do nothing. Otherwise $\chi_s=0$: then we kill $\overline\xi$ (\textit{i.e.}\ we set $\overline\xi_s=\infty$) if, and only if, $\Delta \xi_s>\log 2$.  
More precisely, let $\xi^\sfS$ be a Lévy process with intensity measure on $(0,+\infty)$.
\begin{equation} \label{eq: Lambda^S}
	\Lambda^\sfS(\mathrm{d}x) 
	:= \Lambda(\mathrm{d}x)|_{x\in (0, \log 2]} +p \Lambda(\mathrm{d}x)|_{x\in (\log 2, +\infty)} 
	= (\mathds{1}_{x\in (0, \log 2]}+p\mathds{1}_{x\in (\log 2, +\infty)}) \frac{2\,\mathrm{e}^x\mathrm{d}x}{\sqrt{2\pi(\mathrm{e}^x-1)^3}}.
\end{equation}
We also set 
\begin{equation} \label{eq:Phi^S}
	\Phi^\sfS(q) := \int_0^{\infty} (1-\mathrm{e}^{-qx}) \Lambda^\sfS(\mathrm{d}x), \quad q>-\frac12,
\end{equation}
for the Laplace exponent of $\xi^\sfS$. One may write $\xi = \xi^\sfS +\xi^{\mathsf{K}}$, where $\xi^{\mathsf{K}}$ is an independent Lévy process with Lévy measure supported on $(\log 2,+\infty)$
\[
\Lambda^{\mathsf{K}}(\mathrm{d}x) := (1-p)  \Lambda(\mathrm{d}x)|_{x\in (\log 2, +\infty)} = \frac{2(1-p)\,\mathrm{e}^x\mathrm{d}x}{\sqrt{2\pi(\mathrm{e}^x-1)^3}}.
\]
In this description, as a result of the thinning operation (see\footnote{Informally, the thinning operation is the operation which allows one to select some points in a Poisson point process according to some random rule, obtaining a new Poisson point process with some ``thinned'' intensity measure.} \textit{e.g.} \cite[Section 2.2.2]{baccelli2020random}), $\overline\xi$ has the law of $\xi^\sfS$ killed at the first time $T$ when $\xi^{\mathsf{K}}$ has a jump.
Note also that, in light of \cref{prop: law tagged}, one has that $\overline F^{U}(h)=\exp(-\overline \xi_{\rho(h)})$ for all $h\ge 0$. Additionally, $\Lambda^{\mathsf{K}}$ has finite total mass, so that $T$ is an exponential random variable with parameter 
\begin{equation} \label{eq: lambda}
	\lambda(p) = \int_{\log 2}^\infty \Lambda^{\mathsf{K}}(\mathrm{d}x) = 2(1-p) \sqrt{\frac{2}{\pi}}.
\end{equation}
There is a natural correspondence between $T$ and $H^{U}_\sfS$, which is just given by a Lamperti time-change (recall \eqref{eq:def_time_change}).

\medskip

\subsection{First moment estimate} \label{sec: first moment estimates}

Our first estimate in this subsection concerns the probability that the fragment targeted at the uniform point $U$ survives \emph{long enough}, in the sense that it reaches some small value before it gets (possibly) discarded. Let $\eps>0$, and 
\begin{equation} \label{eq: U hitting eps}
	H^{U}_{\eps} := \inf\{h>0, \; F^{U}(h) < \eps\}.
\end{equation}
Recalling the height $H_{\sfS}^U$ from \eqref{eq: killing H^U}, we also introduce the event 
\begin{equation}\label{eq:key_event}
	\calE_{\eps}^{U} := \left\{H^{U}_{\eps}<H^{U}_{\sfS}\right\},
\end{equation}
\textit{i.e.}, the event that the fragment targeted at the uniform point $U$ gets smaller than $\eps$ before it gets (possibly) discarded.

We fix $p\in(0,1)$ for the rest of the section. All the constants appearing in the next propositions depend on $p$, even if not explicitly stated.

\begin{prop} \label{prop: estimate killing}
	Let $\lstar(p)$ be the only positive solution to the equation $\Phi^\sfS(-\lstar(p)) = -\lambda(p)$, with $\Phi^\sfS$ as in \eqref{eq:Phi^S} and $\lambda(p)$ as in \eqref{eq: lambda}.
	The probability of $\calE_\eps^U$ satisfies
	\begin{equation} \label{eq: Ecal_eps uniform upper bound}
	\bbP(\calE_\eps^U) \le \eps^{\lstar(p)}, \quad\text{for all } \eps>0.
	\end{equation}
	Moreover, there exists a constant $c>0$ such that
	\begin{equation} \label{eq: Ecal_eps equivalent}
	\bbP\big(\calE_{\eps}^{U}\big) \underset{\eps\to 0}{\sim} c\,\eps^{\lstar(p)}.
	\end{equation}
\end{prop}

\begin{remark}\label{eq:lapl_reg}
	We emphasize that the exponent $\lstar(p)$ is the same as the one appearing in \cref{rem:expr-bounds} (see also that remark for some particular values of $\lstar(p)$).
	Since $\Phi^\sfS$ is continuous on $(-1/2,+\infty)$ with $\Phi^\sfS(0)=0$ and $\Phi^\sfS(q) \to -\infty$ as $q\to (-1/2)^+$, it is plain that $\lstar(p)\in(0,\frac12)$. See the left-hand side of \cref{fig-Upper-lower-bound2} for the graph of $\alpha_*(p)\stackrel{\eqref{eq:alpha_lam}}{=}1-\lstar(p)$ for $p\in(0,1)$, obtained by solving numerically the equation $\Phi^\sfS(-\lstar(p)) = -\lambda(p)$.
\end{remark}

\begin{proof}[Proof of \cref{prop: estimate killing}]
	The Lamperti representation provides a natural point of view to address this question (see for instance \cite[Section 2.3]{curmar-swallow} where a similar approach was used in a different context).
	Recall the subordinators $\xi^{\sfS}$ and $\xi^{\mathsf K}$ and the time $T$ introduced just above~\eqref{eq: lambda}. Note that
	\begin{equation} \label{eq: hitting time Lamperti}
		\bbP\big(\calE_{\eps}^{U}\big)
		=
		\bbP\big(H^{U}_{\eps}<H^{U}_{\sfS}\big)
		=
		\bbP\big(T_{-\log \eps} < T\big),
	\end{equation}
	where $T_x:= \inf\{s>0, \; \xi^{\sfS}_s > x\}$ is the first passage time of $\xi^\sfS$ above $x>0$. Now since $T$ is an exponential random variable with parameter $\lambda(p)$ (see \eqref{eq: lambda}), by the independence of $\xi^\sfS$ and $\xi^\mathsf{K}$ (and so of $T_{-\log \eps}$ and $T$), 
	\begin{equation} \label{eq: hitting time exponential}
		\bbP\big(T_{-\log \eps} < T\big)
		=
		\bbE\left[\mathrm{e}^{-\lambda(p) T_{-\log \eps}}\right].
	\end{equation}
	\cref{prop: estimate killing} is then a consequence of \cref{lem: asymptotics subordinator}. First, note that for all $q<1/2$, $(\Phi^\sfS)'(-q) < \infty$ as is easily seen from \eqref{eq: Lambda^S} and \eqref{eq:Phi^S}. Applying the aforementioned lemma for $a = \lambda(p)>0$ and $b=0$, we obtain from \eqref{eq: hitting time Lamperti} and \eqref{eq: hitting time exponential} that 
	\[
	\bbP(\calE_\eps^U) \le \eps^{\lstar(p)},
	\]
	for all $\eps>0$, and
	\[
	\bbP\big(\calE_{\eps}^{U}\big) \underset{\eps\to 0}{\sim} c\, \eps^{\lstar(p)},
	\]
	for some constant $c>0$, where $\lstar(p)$ is the positive solution to the equation $\Phi^\sfS(-\lstar(p)) = -\lambda(p)$. This proves \eqref{eq: Ecal_eps uniform upper bound} and \eqref{eq: Ecal_eps equivalent}.
\end{proof}

We conclude this section with the following moment estimate.

\begin{prop} \label{prop: first moment estimate}
	Let $\zeta \geq 0$. There exists a constant $c=c(\zeta)>0$ such that
	\begin{equation} \label{eq: prop asymptotics first moment}
	\bbE\left[F^{U}(H^{U}_{\eps})^{\zeta} \cdot \mathds{1}_{\calE_{\eps}^{U}} \right]
	\underset{\eps\to 0}{\sim} 
	c\,\eps^{\zeta+\lstar(p)},
	\end{equation}
	where $\lstar(p)$ is the positive solution to $\Phi^\sfS(-\lstar(p)) = -\lambda(p)$.
	Moreover, if $\zeta>-\lambda_*(p)$,
	\begin{equation*} \label{eq: prop asymptotics first moment2}
	\bbE\left[F^{U}(H^{U}_{\eps})^{\zeta} \cdot \mathds{1}_{\calE_{\eps}^{U}} \right]
	\leq  
	\eps^{\zeta+\lstar(p)},
	\quad\text{for all } \eps>0.
	\end{equation*}
\end{prop}

\begin{proof}
	Using the description in \cref{prop: law tagged}, together with the notation at the beginning of this subsection, we have that
	\[
	\bbE\left[F^{U}(H^{U}_{\eps})^{\zeta} \cdot \mathds{1}_{\calE_{\eps}^{U}} \right]
	=
	\bbE\left[\mathrm{e}^{-\zeta \xi^{\sfS}_{T_{-\log \eps}}} \cdot \mathds{1}_{T_{-\log \eps}< T}\right],
	\]
	where we recall that $T_x= \inf\{s>0, \; \xi^{\sfS}_s > x\}$ and $T$ is the exponential random variable with parameter $\lambda(p) =  2(1-p) \sqrt{\frac{2}{\pi}}$ introduced in \eqref{eq: lambda}. Moreover, $T$ is independent of $\xi^{\sfS}$, so that
	\[
	\bbE\left[F^{U}(H^{U}_{\eps})^{\zeta} \cdot \mathds{1}_{\calE_{\eps}^{U}} \right]
	=
	\bbE\left[\mathrm{e}^{-\zeta \xi^{\sfS}_{T_{-\log \eps}}-\lambda(p) T_{-\log \eps}}\right].
	\]
	An application of \cref{lem: asymptotics subordinator} (with $a=\lambda(p)> 0$ and $b=\zeta>-\lambda_*(p)$) yields the desired estimates. 
\end{proof}

\subsection{Two-point function estimate}
We now consider two independent uniform points $U$ and $V$ in $(0,1)$ also independent from all the other random quantities. Recall from \eqref{eq:key_event} the notation $\calE^U_\eps$ and $\calE^V_\eps$ respectively for the events that the fragments containing $U$ and $V$ survive in the strategy $\sfS$ (defined in \eqref{eq:selection}) until getting smaller than $\eps>0$. We are interested in the correlation between the two events $\calE^U_\eps$ and $\calE^V_\eps$. The aim of this section is to prove the following second moment estimate, which should be compared to \cref{prop: estimate killing}.
\begin{prop} \label{prop: two point function estimates}
There exists a constant $c>0$ such that
	\[
	\bbP\left(\calE^U_\eps \cap \calE^V_\eps\right) \underset{\eps\to 0}{\sim} c\, \eps^{2\lstar(p)},
	\]
	where $\lstar(p)$ is the only positive solution to the equation $\Phi^\sfS(-\lstar(p)) = -\lambda(p)$, with $\Phi^\sfS$ as in \eqref{eq:Phi^S}.
\end{prop}

\begin{proof}
	Introduce, for $\eps>0$, the event that $U$ and $V$ split before reaching $\eps$, namely
\[
\calG_{\eps}^{U,V} 
:=
\left\{ I^{U}(H_{\eps}^{U}) \ne I^{V}(H_{\eps}^{V}) \right\}.
\]
Let $\eps>0$. We split the event $\calE^U_\eps \cap \calE^V_\eps$ according to $\calG_{\eps}^{U,V}$ and its complement.

We first deal with the two-point function on the event $(\calG_{\eps}^{U,V})^c$. In this case, we condition on $(\efrak,\sfrak,U)$ to obtain
\begin{equation}\label{eq: two point function complement G}
\bbP\left(\calE^U_\eps \cap \calE^V_\eps \cap (\calG_{\eps}^{U,V})^c\right)
=
\bbP\left(\calE^U_\eps \cap \{V\in I^U(H_\eps^U)\}\right)
=
\bbE\left[\mathds{1}_{\calE^U_\eps} \cdot \bbP\left(V\in I^U(H_\eps^U)|(\efrak,\sfrak, U)\right)\right].
\end{equation}
We now argue conditionally on $(\efrak,\sfrak,U)$. Since $H_\eps^U$ is a stopping time for the filtration $\sigma(I^U(h'), h'\leq h)$, defined for all $h > 0$, and $V$ is independent of $(\efrak, \sfrak, U)$, we a.s.\ have $\bbP\left(V\in I^U(H_\eps^U)|(\efrak,\sfrak, U)\right) = F^U(H_\eps^U)$. Plugging this identity into \eqref{eq: two point function complement G}, we infer that 
\[
\bbP\left(\calE^U_\eps \cap \calE^V_\eps \cap (\calG_{\eps}^{U,V})^c\right)
=
\bbE\left[ \mathds{1}_{\calE_\eps^U} \cdot F^U(H_\eps^U)\right].
\]
Using \cref{prop: first moment estimate}, we therefore conclude that there exists a constant $c_1>0$ such that 
\begin{equation} \label{eq: two point function equivalent event G^c}
\bbP\left(\calE^U_\eps \cap \calE^V_\eps \cap (\calG_{\eps}^{U,V})^c\right)
\underset{\eps\to 0}{\sim}
c_1 \eps^{1+\lstar(p)}.
\end{equation}

We next deal with the two-point function on the event $\calG_{\eps}^{U,V}$. Remark that this event can be rephrased as the existence of a branching height $a \le H_\eps^U \wedge H_\eps^V$ separating $U$ and $V$. We note that $a$ is a stopping time with respect to the filtration $\sigma(I^U(h'), I^V(h')\,,\, h'\le h)$, defined for all $h>0$. Moreover, given $I^U(a)$, the law of $U$ is uniform in $I^U(a)$ (and the same applies to $V$). Therefore, by conditioning at height $a$ and using the branching property of excursions above $a$,
	\begin{equation} \label{eq: branching height U,V}
		\bbP \Big( \calE_{\eps}^{U}\cap \calE_{\eps}^{V}\cap \calG^{U,V}_{\eps} \Big) \\
		=
		\bbE\left[\mathds{1}_{\{a<H^{U}_\sfS \wedge H^{V}_\sfS \wedge H^{U}_{\eps}\wedge H^{V}_{\eps}\}} \cdot \bbP_{F^{U}(a)}\big( \calE_{\eps}^{\widetilde{U}} \big) \cdot \bbP_{F^{V}(a)}\big( \calE_{\eps}^{\widetilde{V}}\big) \right],
	\end{equation}
		where $\bbP_\ell$ describes the law of a Brownian excursion with duration $\ell>0$, and conditionally on $F^{U}(a)$ and $F^{V}(a)$, $\widetilde{U}, \widetilde{V}$ are independent and uniform in $(0, F^{U}(a))$ and $(0, F^{V}(a))$ respectively. By Brownian scaling, for $\ell>0$ we have
	\[
	\bbP_{\ell}\big(\calE_{\eps}^{\widetilde{U}}\big)
	=
	 \bbP\big( \calE_{\eps/\ell}^{U'}\big) \quad \text{and} \quad 
	\bbP_{\ell}\big(\calE_{\eps}^{\widetilde{V}}\big)
	=
	 \bbP\big( \calE_{\eps/\ell}^{V'}\big),
	\]
	with $U', V'$ uniform in $(0,1)$. Equation \eqref{eq: branching height U,V} now boils down to
	\begin{equation} \label{eq: branching height U,V scaling}
		\bbP \Big( \calE_{\eps}^{U}\cap \calE_{\eps}^{V}\cap \calG^{U,V}_{\eps} \Big) 
		=
		\bbE\left[\mathds{1}_{\{a<H^{U}_\sfS \wedge H^{V}_\sfS \wedge H^{U}_{\eps}\wedge H^{V}_{\eps}\}} \cdot \bbP\left( \calE_{\eps/F^{U}(a)}^{U'} \,\middle|\, F^{U}(a) \right) \cdot \bbP\left( \calE_{\eps/F^{V}(a)}^{V'} \,\middle|\, F^{V}(a)\right) \right].
	\end{equation}
	We now take $\eps \to 0$. We claim that for some constant $c_2>0$,
	\begin{equation} \label{eq: two point function equivalent event G}
		\bbP \Big( \calE_{\eps}^{U}\cap \calE_{\eps}^{V}\cap \calG^{U,V}_{\eps} \Big)
		\underset{\eps\to 0}{\sim} 
		c_2 \, \eps^{2\lstar(p)} \cdot \bbE\left[\mathds{1}_{\{a<H^{U}_\sfS \wedge H^{V}_\sfS\}} F^{U}(a)^{-\lstar(p)} F^{V}(a)^{-\lstar(p)}\right].
	\end{equation}
	Indeed, from \eqref{eq: Ecal_eps equivalent} in \cref{prop: estimate killing}, we know that there exists a constant $c>0$ such that $\bbP\big( \calE_{\eps/F^{U}(a)}^{U'} \big| F^{U}(a) \big) \underset{\eps\to 0}{\sim} c \cdot (\eps/F^{U}(a))^{\lstar(p)}$ and likewise $\bbP\big( \calE_{\eps/F^{V}(a)}^{V'} \big| F^{V}(a) \big) \underset{\eps\to 0}{\sim} c \cdot (\eps/F^{V}(a))^{\lstar(p)}$. The result then follows from \eqref{eq: branching height U,V scaling} provided we can apply dominated convergence. We now justify that we can apply it. 
	
	First of all, recall from \eqref{eq: Ecal_eps uniform upper bound} that we have, for all $\eps>0$, the upper-bounds 
	\[
	\bbP\big( \calE_{\eps/F^{U}(a)}^{U'} \big| F^{U}(a) \big) \le (\eps/F^{U}(a))^{\lstar(p)} \quad \text{and} \quad \bbP\big( \calE_{\eps/F^{V}(a)}^{V'} \big| F^{V}(a) \big) \le (\eps/F^{V}(a))^{\lstar(p)}.
	\]
	The domination will therefore follow if we prove that
	\begin{equation} \label{eq: two point function domination}
	\bbE\left[\mathds{1}_{\{a<H^{U}_\sfS \wedge H^{V}_\sfS\}} F^{U}(a)^{-\lstar(p)} F^{V}(a)^{-\lstar(p)} \right] < \infty.
	\end{equation}
	To do so, we argue conditionally on $F^U$, by first noting that one can construct $F^V(a)$ from $F^U$ as follows. Recall that $a$ corresponds to the branching height separating $U$ and $V$, hence $F^V(a)$ is equal to the length of the interval not containing $U$ split by the jump of $F^U$ at height $a$, that is  $F^V(a)=-\Delta F^{U}(a)$ (recall \cref{fig-selection-rule}, p.\ \pageref{fig-selection-rule}). Now observe that given some fixed jump time $h$ of $F^U$, the probability that the branching height $a$ is equal to $h$ is given by the probability that $V$ belongs to the interval corresponding to the fragment $-\Delta F^U(h)$. In other words, conditionally on $F^{U}$, using the independence of $V$ and $U$,  one can build $F^{V}(a)$ by selecting the jump $-\Delta F^{U}(h)$ of $F^{U}$ at time $h$ with probability $-\Delta F^{U}(h)$. By removing the indicator and conditioning \eqref{eq: two point function domination} on $F^{U}$, one therefore gets
		\[
		\bbE\left[ \mathds{1}_{a<H_{\sfS}^{U}\wedge H_{\sfS}^{V}} (F^{U}(a)F^{V}(a))^{-\lstar(p)} \right] \\
		\le
		\bbE\left[ \sum_{0<h<h^U} (F^{U}(h))^{-\lstar(p)} (-\Delta F^{U}(h))^{1-\lstar(p)} \right].
		\]
	Using the Lamperti representation from \cref{prop: law tagged}, and noting that $-\Delta F^{U}(h)=\mathrm{e}^{-\xi_{b^-}}(1-\mathrm{e}^{-\Delta\xi_{b}})$ with $b=\rho(h)$, this becomes
	\begin{align*}
		\bbE\left[ \mathds{1}_{a<H_{\sfS}^{U}\wedge H_{\sfS}^{V}} (F^{U}(a)F^{V}(a))^{-\lstar(p)} \right]
		&\le
		\bbE\left[ \sum_{b>0} \mathrm{e}^{\lstar(p)\xi_b} \mathrm{e}^{-(1-\lstar(p))\xi_{b^-}} (1-\mathrm{e}^{-\Delta\xi_b})^{1-\lstar(p)} \right] \\
		&=
		\bbE\left[ \sum_{b>0} \mathrm{e}^{-(1-2\lstar(p))\xi_{b^-}}  \mathrm{e}^{\lstar(p)\Delta\xi_b}(1-\mathrm{e}^{-\Delta\xi_b})^{1-\lstar(p)} \right].
	\end{align*}
	Then an application of the compensation formula for $\xi$ (see \textit{e.g.} \cite[Theorem 4.4]{kyprianou2014fluctuations}) provides 
		\[
		\bbE\left[ \mathds{1}_{a<H_{\sfS}^{U}\wedge H_{\sfS}^{V}} (F^{U}(a)F^{V}(a))^{-\lstar(p)}\right] \\
		\le
		\bbE\left[ \int_0^\infty \mathrm{e}^{-(1-2\lstar(p))\xi_{b}}\mathrm{d}b\right] \cdot \int_0^\infty \Lambda(\mathrm{d}x) \mathrm{e}^{\lstar(p) x}(1-\mathrm{e}^{-x})^{1-\lstar(p)}.
		\]
	On the one hand, since $\lstar(p)<1/2$, it is clear from the expression of $\Lambda$ in \cref{prop: law tagged} that the second integral is finite. On the other hand, the first expectation is simply
	\[
		\bbE\left[\int_0^\infty \mathrm{e}^{-(1-2\lstar(p))\xi_{b}}\mathrm{d}b\right]
		=
		 \int_0^\infty \mathrm{e}^{-b\Phi(1-2\lstar(p))} \mathrm{d}b.
	\]
	As $\lstar(p)<1/2$, the observation that $\Phi(1-2\lstar(p))>0$ concludes the proof of \eqref{eq: two point function domination}. The domination is thus established, which proves \eqref{eq: two point function equivalent event G}.
	
	\cref{prop: two point function estimates} finally follows from the two asymptotics \eqref{eq: two point function equivalent event G^c} and \eqref{eq: two point function equivalent event G}.
	\end{proof}

\section{Lower bound for sequences sampled from the Brownian separable permutons} \label{sec: lower bound seq permuton}

The main goal of this section is to prove the lower bound in \cref{thm:upper_lower_perm}. As already pointed out in \cref{rem:expr-bounds}, $\alpha_*(p)>1/2$ for all $p\in(0,1)$ (indeed, $\alpha_*(0^+) = 1/2$ and $\alpha_*(1^-) = 1$, and $\alpha_*(p)$ is strictly increasing in $p\in(0,1)$). Hence, it is enough to prove the lower bound in the second item in the theorem statement. This is done in the following proposition.

\begin{prop}\label{prop:lwbound}
	Fix $p\in(0,1)$ and let $\alpha_*(p)=1-\lambda_*(p)$ be as in \cref{rem:expr-bounds}. Let $\sigma_n$ be a random permutation of size $n\in \bbN$ sampled from the Brownian separable permuton $\bm{\mu}_p$. Then for all $\alpha<\alpha_*(p)$, the following convergence in probability holds
	\begin{equation}\label{eq:conv-prob-inf}
		\frac{\LIS(\sigma_n)}{n^{\alpha}}\to \infty.
	\end{equation}
\end{prop}

\begin{proof}
	Fix $n\in \bbN$ and set $\eps=1/n$. Let $(U_i)_{i\leq n}$ be a sequence of i.i.d.\ uniform random variables in $(0,1)$, independent of all the other random quantities, and recall from \cref{sect:sampling} that $\sigma_n=\Perm(\efrak,\sfrak,(U_i)_{i\leq n})$.
	Recall also the event $\calE_{\eps}^{U}$ from \eqref{eq:key_event}, and introduce $S_n := \sum_{i=1}^n\mathds{1}_{\calE_{\eps}^{U_i}}$. Thanks to \cref{prop: estimate killing}, there exist two constants $c,C>0$ such that, for $n$ large enough,
	\begin{equation}\label{eq:first mom bound}
		\bbE[S_n]=\sum_{i=1}^n\bbP(\calE_{\eps}^{U_i})\in (c\cdot n^{\alpha_*(p)},C\cdot  n^{\alpha_*(p)}),
	\end{equation}
	where we used that $\eps=1/n$ and $\alpha_*(p)=1-\lstar(p)$. Moreover, from \cref{prop: two point function estimates}, if $U$ and $V$ are independent uniform random variables in $(0,1)$ also independent of all the other random quantities, there exists another constant $c>0$ such that, for $\eps=1/n$ small enough,
	\begin{equation}\label{eq:second mom bound}
		\bbP\bigg( \calE_{\eps}^{U} \cap \calE_{\eps}^{V}\bigg)\leq  c \cdot \eps^{2\lstar(p)}=c \cdot n^{2\alpha_*(p)-2}.
	\end{equation}
	Using \eqref{eq:first mom bound} and \eqref{eq:second mom bound} we deduce that for $n$ large enough,
	\begin{equation}\label{eq:second mom bound2}
		\bbE[S_n^2]=\sum_{i,j=1}^n\bbP(\calE_{\eps}^{U_i}\cap\calE_{\eps}^{U_j})\leq c\cdot n^{2\alpha_*(p)},
	\end{equation}
	for some (other) constant $c>0$.
	The Paley–Zygmund inequality with the bounds in \eqref{eq:first mom bound} and \eqref{eq:second mom bound2} implies that there exist two (other) constants $c,q>0$ such that for $n$ large enough,
	\begin{equation}\label{eq:prestimete}
		\bbP(S_n \geq c\cdot n^{\alpha_*(p)})\geq q.
	\end{equation}
	Now, recalling that $S_n = \sum_{i=1}^n\mathds{1}_{\calE_{\eps}^{U_i}}$, denote by $\{U^*_\ell,\ell\in[S_n]\}$ the $S_n$ random variables among $\{U_i,i\le n\}$ such that $\calE_{\eps}^{U^*_\ell}$ occurs. 
	We now want to extract a large subset $\{U^*_\ell,\ell\in \Lambda\}$ of $\{U^*_\ell,\ell\in[S_n]\}$ such that the corresponding intervals $\{I^{U^*_\ell}_{\eps},\ell\in \Lambda\}$ are disjoint (recall that $I^t_\eps:=I^t(H^t_\eps)$ denotes the largest interval containing $t$ of size smaller than $\eps$ in the fragmentation process). The reason to be interested in such a large subset resides in the following result. 
	
	\begin{lem}\label{clm:clm1}
		Almost surely, if $\{I^{U^*_\ell}_{\eps},\ell\in \Lambda\}$ are disjoint, then the permutation $\Perm(\efrak,\sfrak,(U^*_\ell)_{\ell\in \Lambda})$ is increasing.
	\end{lem}
	
	\begin{proof}[Proof of \cref{clm:clm1}]
		Assume that $\{I^{U^*_\ell}_{\eps},\ell\in \Lambda\}$ are disjoint and recall that for all $\ell$, the variable $U^*_\ell$ is chosen so that $\mcl E^{U^*_\ell}_{\eps}$ occurs. By definition of $\mcl E^{U}_{\eps}$ in \eqref{eq:key_event}, if $\mcl E^{U}_{\eps}$ occurs then the interval $I^U_{\eps}$ containing $U$ is not discarded by the selection rule $\sfS$. Now suppose for a contradiction that there exist two variables $U^*_{\ell_1}$ and $U^*_{\ell_2}$, with $U^*_{\ell_1}<U^*_{\ell_2}$ and $\ell_1,\ell_2 \in \Lambda$, such that the two corresponding values in $\Perm(\efrak,\sfrak,(U^*_\ell)_{\ell\in\Lambda})$ form an inversion.
		Since $I^{U^*_{\ell_1}}_{\eps}$ and $I^{U^*_{\ell_2}}_{\eps}$ are disjoint, there must exist (by \cref{lem:prop_support_perm}) a local minimum of $\efrak$ at some height $h=\efrak(t)$ which separates $U^*_{\ell_1}$ and $U^*_{\ell_2}$, in particular $t\notin I^{U^*_{\ell_1}}_{\eps}\cup I^{U^*_{\ell_2}}_{\eps}$, 
		$F^{U^*_{\ell_1}}(h^-)\geq \eps$ and $F^{U^*_{\ell_2}}(h^-)\geq \eps$. See \cref{fig-discarding}. Moreover, the sign corresponding to such a local minimum must be a $\ominus$ (since we assumed that the two elements corresponding to $U^*_{\ell_1}$ and $U^*_{\ell_2}$ in $\Perm(\efrak,\sfrak,(U^*_\ell)_{\ell\in\Lambda})$ form an inversion). We deduce that either $U^*_{\ell_1}$ or $U^*_{\ell_2}$ must be discarded by $\sfS$ at height $h$. Since $F^{U^*_{\ell_1}}(h^-)\geq \eps$ and $F^{U^*_{\ell_2}}(h^-)\geq \eps$, this contradicts the fact that $\mcl E^{U^*_{\ell_1}}_{\eps}$ and $\mcl E^{U^*_{\ell_2}}_{\eps}$ occur. 
	\end{proof}

	\begin{figure}[ht!]
		\begin{center}
			\includegraphics[width=.6\textwidth]{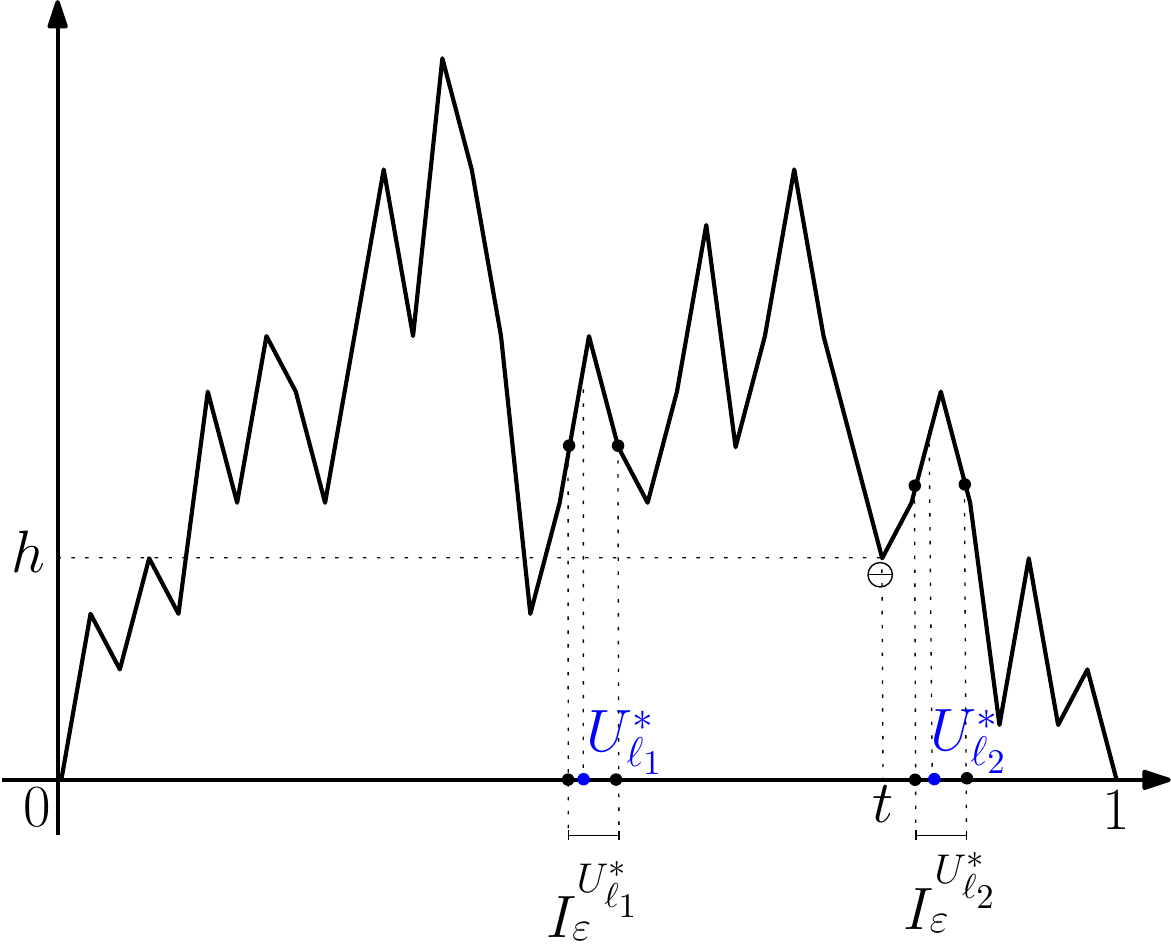}  
			\caption{\label{fig-discarding} An illustration explaining the proof of \cref{clm:clm1}.}
		\end{center}
		\vspace{-3ex}
	\end{figure}
	
	In order to guarantee that the size of $\Lambda$ (and hence the size of the increasing permutation $\Perm(\efrak,\sfrak,(U^*_\ell)_{\ell\in \Lambda})$) is large, we also need the following estimate.
	
	\begin{lem}\label{clm:clm2}
		Let $\eps=1/n$ and fix $\zeta>0$. There exist two constants $c_1,c_2>0$ (which may depend on $\zeta$ but not on $n$) such that
		\[\bbP\left(\forall i \in [n], \; \#\{j\in[n], \, U_j\in I^{U_i}_\eps\}<n^{\zeta}\right)\geq 1-c_1 n \cdot\exp(-c_2n^{\zeta}), \qquad \text{for all }n\geq 1.\]
	\end{lem}
	
	\begin{proof}[Proof of \cref{clm:clm2}]
		This is a standard binomial concentration argument. Let $\eps=1/n$ and fix $\zeta>0$. 
		We recall that the number of uniform variables among $(U_i)_{i\leq n}$ which fall in some prescribed interval $A$ follows a binomial distribution $\text{Bin}(n,|A|)$ with success probability $|A|$. Note that for all $n\ge 1$,
		\begin{align*}
			\bbP\left( \forall i \in [n], \, \#\{j\in[n], \, U_j\in I^{U_i}_\eps\}<n^{\zeta}\right)
			&\geq 1-\sum_{i=1}^n\bbP\big(\#\{j\in[n]\setminus\{i\}, \, U_j\in I^{U_i}_\eps\} \geq n^{\zeta}-1\big)\\
			&\geq1-\sum_{i=1}^n\bbP(\text{Bin}(n-1,\eps)\geq n^{\zeta}-1),
		\end{align*}
		where in the last inequality we used that by definition $|I^{U_i}_\eps|=F^{U_i}(H_\eps^{U_i})\leq \eps$. Now recalling that $\eps=1/n$, we get by Chernov's bound that
		\begin{align}\label{eq:bound-bin}
			\bbP\left( \forall i \in [n], \, \#\{j\in[n], \, U_j\in I^{U_i}_\eps\}<n^{\zeta}\right)
			&\geq1-c_1 n\cdot\min_{\gamma>0} \{\exp(-\gamma n^{\zeta})\bbE[\exp(\gamma \text{Bin}(n-1,1/n))]\} \notag\\
			&=1-c_1 n\cdot\min_{\gamma>0}\{ \exp(-\gamma n^{\zeta}) (1-1/n+\exp(\gamma)/n)^{n-1}\} \notag\\
			&\geq
			1-c_1 n \cdot \exp(-c_2n^{\zeta}),
		\end{align}
		where $c_1,c_2>0$ are two constants.
	\end{proof}

	Note that on the event $\left\{S_n\geq c\cdot n^{\alpha_*(p)}\right\}\cap\left\{\forall i \in [n], \; \#\{j\in[n], \, U_j\in I^{U_i}_\eps\}<n^{\zeta}\right\}$, we can extract a subset $\{U^*_\ell,\ell\in \Lambda\}$ of $\{U^*_\ell,\ell\in[S_n]\}$ such that the corresponding intervals $\{I^{U^*_\ell}_{\eps},\ell\in \Lambda\}$ are disjoint and the size of $\Lambda$ is large, in the sense that $\#\Lambda\geq c \cdot n^{\alpha_*(p)-\zeta}$. Hence, on this event, thanks to \cref{clm:clm1}, we have that $\Perm(\efrak,\sfrak,(U^*_\ell)_{\ell\in\Lambda})$ is an increasing subsequence in $\sigma_n=\Perm(\efrak,\sfrak,(U_i)_{i\leq n})$ of size at least $c\cdot n^{\alpha_*(p)-\zeta}$.
	
	Using \eqref{eq:prestimete} and \cref{clm:clm2}, we deduce that for all $\zeta>0$ there exist two (other) constants $c,q>0$ (which may depend on $\zeta$) such that, 
	\begin{equation}\label{eq:non_zero_proba}
		\bbP\left(\LIS(\sigma_n)\geq  c\cdot n^{\alpha_*(p)-\zeta} \right)\geq q, \qquad \text{for all } n\geq 1.
	\end{equation}
	Note that this estimate holds for all $n\geq 1$ since we are allowed to choose a (possibly smaller) constant $c>0$.
	
	The rest of the proof is devoted to upgrading \eqref{eq:non_zero_proba} by proving that for all $\zeta>0$ there exists (another) constant $c>0$  (which may depend on $\zeta$)  such that 
		\begin{equation}\label{eq:goal-final-part}
			\lim_{n\to \infty}\bbP\left(\LIS(\sigma_n)\geq  c\cdot n^{\alpha_*(p)-\zeta} \right)=1.
		\end{equation}
	To do this, we use a zero-one law type argument.
	Recall the signed excursion $(\efrak,\sfrak,p)$ and the definition of the tagged fragment from \cref{prop: law tagged}. 
	Given a branching height $h>0$, recall that $\sfrak(h)$ is the corresponding sign.
	For $\delta>0$, let
	\begin{equation}\label{eq:defn_H}
	H=H(\delta):=\inf\left\{h>0 \, \middle|\; \sfrak(h)=\oplus,\,\,  \min\{F^{U}(h^-)-F^{U}(h),F^{U}(h)\}\geq \delta,\,\, F^U(h)>1/2\right\},
	\end{equation}
	with the convention that $\inf\emptyset=\infty$. We claim the following.
	
	\smallskip 
	
	\begin{lem}\label{clm:clm3} 
		For every $\eps'>0$, there exists $\delta>0$ such that $\bbP(H<\infty)\geq 1-\eps'$. 
	\end{lem} 
	
	\begin{proof}[Proof of \cref{clm:clm3}]
		Fix $\eps'>0$. 
		It suffices to show that a.s.\ there exists a random $\delta > 0$ such that $H(\delta) < \infty$. Indeed, by the continuity of the probability for increasing families of events, this implies that there is some deterministic $\delta$ such that $\bbP(H(\delta) < \infty) \geq 1 -\eps'$. 
			
			Recall that from \cref{prop: law tagged} we have that $F^{[U]}(h) = \exp(-\xi_{s})$ with $s=\rho(h)$. Since $\xi$ is a subordinator with Levy measure $\Lambda(0,\infty)=\infty$, then $F^U(h)$ does not immediately jump below $1/2$ (by right-continuity of $F^U(h)$), and it has infinitely many downward jumps in every non-trivial interval of times (because $\Lambda(0,\infty)=\infty$; see \textit{e.g.} \cite[Section I.5]{bertoin1996levy}). Hence there are infinitely many $h > 0$ such that $\min\{F^{U}(h^-)-F^{U}(h),F^{U}(h)\} > 0$ and $F^U(h) >1/2$. Since the signs $\sfrak$ at different branching heights are i.i.d., a.s.\ there exists $h >0$ satisfying the conditions of the previous sentence such that $\sfrak(h)=\oplus$. We then take $\delta := \min\{F^{U}(h^-)-F^{U}(h),F^{U}(h)\}$ for the latter choice of $h$.
	\end{proof}

	We now fix $\eps'>0$ and take $\delta>0$ so that $\bbP(H<\infty)\geq 1-\eps'$. When $H<\infty$, by the definition of $H$ in \eqref{eq:defn_H}, the sign of the branching height $H$ is $\sfrak(H)=\oplus$ and $H$ splits $I^U(H)$ into the two intervals $L$ and $R$ with $\min\{|L|,|R|\}=\min\{F^{U}(H^-)-F^{U}(H),F^{U}(H)\}\geq \delta$, and none of them is discarded by  the selection rule $\sfS$ (recall how the selection rule $\sfS$ works from \eqref{eq:selection}). When $H=\infty$, we set $L=R=\emptyset$. With the convention that $\Perm(\efrak,\sfrak,\emptyset)$ is the empty permutation, we set
	\[\sigma^L_n:=\Perm(\efrak,\sfrak,(U_i)_{i\leq n}\cap L)\qquad\text{and}\qquad\sigma^R_n:=\Perm(\efrak,\sfrak,(U_i)_{i\leq n}\cap R),\] 
	and observe that $|\sigma^L_n|=\#\big((U_i)_{i\leq n}\cap L\big)$ and $|\sigma^R_n|=\#\big((U_i)_{i\leq n}\cap R\big)$.  We have the following bound on the size of the two permutations.
	
	\smallskip 
	
	\begin{lem}\label{clm:clm4}
		For all $\delta>0$, there exists two constants $c_1,c_2>0$ (which may depend on $\delta$ but not on $n$) such that 
		\[\bbP\left(\min\{|\sigma^L_n|,|\sigma^R_n|\}\geq \delta n/3\,\middle|\, H<\infty \right)\geq  1-c_1\exp(-c_2n), \qquad \text{for all }n\geq 1.\]
	\end{lem}
	
	\begin{proof}[Proof of \cref{clm:clm4}]
		The proof uses standard binomial concentration arguments. Fix $\delta>0$. 
		We have that
		\begin{multline*}
			\bbP\left(\min\{|\sigma^L_n|,|\sigma^R_n|\}\geq \delta n/3\,\middle|\, H<\infty\right) 
			=
			1-\bbP\left(\min\{|\sigma^L_n|,|\sigma^R_n|\}< \delta n/3\,\middle|\, H<\infty\right)\\ 
			\geq1-\bbP\left(|\sigma^L_n|< \delta n/3\,\middle|\, H<\infty\right)-\bbP\left(|\sigma^R_n|< \delta n/3\,\middle|\, H<\infty\right)
			\geq
			1-2\bbP\left(\text{Bin}(n,\delta)< \delta n/3\right),
		\end{multline*}
		where the first inequality is a union bound, and the last inequality comes from the fact that $\min\{|L|,|R|\}\geq \delta$ when $H<\infty$.
		Standard binomial concentration bounds (as in \eqref{eq:bound-bin}) then provide the existence of two constants $c_1,c_2>0$ such that
			\[
			\bbP\left(\min\{|\sigma^L_n|,|\sigma^R_n|\}\geq \delta n/3\,\middle|\, H<\infty \right) \ge 1-c_1\exp(-c_2n), \quad \text{for all } n\ge 1,
			\]
		which is our claim.
	\end{proof}

	We now conclude the proof of \cref{prop:lwbound}. By standard self-similarity properties of the Brownian excursion $\efrak$, under the conditional law given $(|L|,|R|,|\sigma^L_n|,|\sigma^R_n|)$, the permutations $\sigma^L_n$ and $\sigma^R_n$ have the same law as two independent copies of $\Perm(\efrak,\sfrak,(U_i)_{i\leq n})$ but of size $|\sigma^L_n|$ and $|\sigma^R_n|$ and defined in terms of two independent excursions of duration $|L|$ and $|R|$, respectively.
	Taking $c>0$ as in \eqref{eq:non_zero_proba}, and setting 
		\begin{equation*}
			G := \{H<\infty,\, \min\{|\sigma^L_n|,|\sigma^R_n|\}\geq \delta n/3 \},
		\end{equation*}
		we get that
	\begin{align}\label{eq:cond_bound}
		&\bbP\left(\{\LIS(\sigma^L_n) \geq  c(\delta n/3)^{\alpha_*(p)-\zeta}\}\cup\{\LIS(\sigma^R_n) \geq  c(\delta n/3)^{\alpha_*(p)-\zeta}\} \right)\\
		&= 1-
		\bbE\left[\bbP\left(\LIS(\sigma^L_n) < c(\delta n/3)^{\alpha_*(p)-\zeta}\, \middle| \, |L|,|\sigma^L_n|\right) \cdot \bbP\left(\LIS(\sigma^R_n) < c(\delta n/3)^{\alpha_*(p)-\zeta}\, \middle| \, |R|,|\sigma^R_n|\right)\right]\notag\\
		&\geq
		1-
		\bbE\left[\bbP\left(\LIS(\sigma^L_n) < c  |\sigma^L_n|^{\alpha_*(p)-\zeta}\, \middle| \, |L|,|\sigma^L_n|\right) \cdot \bbP\left(\LIS(\sigma^R_n) < c  |\sigma^R_n|^{\alpha_*(p)-\zeta}\, \middle| \, |R|,|\sigma^R_n|\right)\cdot \mathds{1}_G\right] - \bbP(G^c),\notag
	\end{align}
	where in the last inequality we used that on the event $G$ we have that $\min\{|\sigma^L_n|,|\sigma^R_n|\}\geq \delta n/3$.
	Now, observe that by Brownian scaling,
	\begin{equation*}
		\left(\sigma_n^L\,\middle|\,|L|,|\sigma^L_n|\right)\stackrel{d}{=}\left(\Perm\left(\efrak,\sfrak,(\overline U_i)_{i\leq |\sigma^L_n|}\right)\,\middle|\,  |\sigma^L_n|\right),
	\end{equation*}
	where $(\overline U_i)_{i}$ is a new sequence of i.i.d.\ uniform random variables in $(0,1)$, independent of all other quantities.
	Recalling that $\sigma_n=\Perm(\efrak,\sfrak,(U_i)_{i\leq n})$, we get that almost surely, for all $n\geq 1$, 
	\begin{equation*}
		\bbP\left(\LIS(\sigma^L_n) < c \cdot |\sigma^L_n|^{\alpha_*(p)-\zeta}\, \middle| \, |L|,|\sigma^L_n|\right)   \cdot \mathds{1}_G =\bbP\Big(\LIS(\sigma_{|\sigma^L_n|}) < c\cdot {|\sigma^L_n|}^{\alpha_*(p)-\zeta}\, \Big| \, |\sigma^L_n|\Big)  \cdot \mathds{1}_G \leq 1-q, 
	\end{equation*}
	where the last inequality follows from the bound in \eqref{eq:non_zero_proba} (recall that we chose the same constant $c>0$ as in \eqref{eq:non_zero_proba}).
	Noting that the same argument holds with $R$ instead of $L$, we conclude from \eqref{eq:cond_bound} that, for all $n\geq 1$,
	\begin{equation*}
		\bbP\left(\{\LIS(\sigma^L_n) \geq  c \cdot(\delta n/3)^{\alpha_*(p)-\zeta}\}\cup\{\LIS(\sigma^R_n) \geq  c\cdot(\delta n/3)^{\alpha_*(p)-\zeta}\}\right)
		\geq 1-(1-q)^2-\bbP(G^c).
	\end{equation*}
	Claims \ref{clm:clm3} and \ref{clm:clm4} imply that, for each large enough $n$,
	\begin{equation*}
		\bbP\left(\{\LIS(\sigma^L_n) \geq  c\cdot(\delta n/3)^{\alpha_*(p)-\zeta}\}\cup\{\LIS(\sigma^R_n) \geq  c\cdot(\delta n/3)^{\alpha_*(p)-\zeta}\}\right)
		\geq 1-(1-q)^2-2\eps'.
	\end{equation*}
	Since if $\LIS(\sigma^L_n) \geq  c(\delta n/3)^{\alpha_*(p)-\zeta}$ or $\LIS(\sigma^R_n) \geq  c(\delta n/3)^{\alpha_*(p)-\zeta}$ then $\LIS(\sigma_n) \geq  c(\delta n/3)^{\alpha_*(p)-\zeta}$, the last estimate implies that for each large enough $n$,
	\begin{equation*}
		\bbP\left(\LIS(\sigma_n) \geq  c\cdot(\delta/3)^{\alpha_*(p)-\zeta} \cdot n^{\alpha_*(p)-\zeta}\right)\geq 1-(1-q)^2-2\eps',
	\end{equation*}
	where we recall that $\eps'$ and $\delta$ were fixed right after the statement of \cref{clm:clm3}.
	By possibly choosing a smaller constant $c$, we deduce that for each $\eps'>0$ there exists a constant $c(\eps')>0$ such that
	\begin{equation}\label{eq:new_bound}
		\bbP\left(\LIS(\sigma_n) \geq  c(\eps') \cdot n^{\alpha_*(p)-\zeta}\right)\geq 1-(1-q)^2-2\eps', \qquad\text{for all } n\geq 1. 
	\end{equation}
	Now, set 
	\[q^*:=\sup\left\{r\in[0,1]\,\middle|\,\text{there exists }c(r)>0 \text{ s.t. }\bbP(\LIS(\sigma_n)\geq c(r)\cdot n^{\alpha_*(p)-\zeta})\geq r\text{ for all } n\geq 1\right\},\]  
	and note that $q^*>0$ thanks to \eqref{eq:non_zero_proba}.  The two bounds in \eqref{eq:non_zero_proba} and \eqref{eq:new_bound} and our definition of $q^*$ show that if $q < q^*$, then  $1-(1-q)^2\leq q^*$. Taking $q\to q^*$, we get $1-q^* \le (1-q^*)^2$. Noting that the only strictly positive solution of the latter inequality is $q^*=1$, we conclude the proof of \eqref{eq:goal-final-part}. 
	
	Since \eqref{eq:goal-final-part} holds for any arbitrary $\zeta>0$, we deduce the convergence in \eqref{eq:conv-prob-inf}.
\end{proof}

\section{Estimates for the upper bound}
\label{sec-sec3}

In this section we prove two fundamental estimates (see Propositions \ref{prop:first step2} and \ref{prop:second step2} below) that will be later used in \cref{sect:up_bound_seq_perm} to upper bound the length of the longest increasing subsequence in permutations sampled from the Brownian separable permutons. 

\subsection{The two main estimates}
\label{sec:strategy_upper_bound}

Recall from \cref{sec: strategy lower bound} that $\mcl B_{\ominus}$ and $\mcl B_{\oplus}$ stand for the collections of negative and positive branching heights in the signed excursion $(\efrak, \sfrak, p)$ respectively, \textit{i.e.}\ the collection of heights corresponding to the local minima of $\efrak$ decorated by a $\ominus$--sign and a $\oplus$--sign respectively. Moreover $\mcl B=\mcl B_{\oplus}\cup\mcl B_{\ominus}$. 
Recall also that $I^{t}(h)$, $h\ge 0$, denotes the open interval which contains $t$ in the interval fragmentation $\Ffrak_0(h)$ introduced in \eqref{eq: def Ffrak}. We set $I^{t}(h^-)$ to be equal to the interior of $\bigcap_{\ell<h}I^{t}(\ell)$.

Of particular importance will be the case of negative branching heights $\mcl B_{\ominus}$, for which we introduce the following additional notation. Since $\efrak$ has almost surely distinct local minima (\cref{lem:prop_support_perm}), almost surely, for all $b\in \mcl B_{\ominus}$, there is a unique time $t_b\in[0,1]$ when $\efrak$ has a local minimum with $b=\efrak(t_b)$. 
We define
\begin{equation}\label{eq:notation_l_r}
 (\ell_b,r_b) := I^{t_b}(b^-) ,\qquad L_b:=(\ell_b,t_b), \qquad \text{and} \qquad R_b:=(t_b,r_b) ,
\end{equation}   
\textit{i.e.}\ $L_b$ and $R_b$ are the two intervals in which the interval $I^{t_b}(b^-)$ is split at the branching height $b$. Finally, given $t\in[0,1]$, we also denote by $\mcl B^t_{\ominus}$ the collection of negative branching heights $b\in\mcl B_{\ominus}$ such that $t\in(\ell_b,r_b)$. See \cref{fig:excursion-split2} for an illustration.

\begin{figure}[ht!]
	\begin{center}
		\includegraphics[width=.8\textwidth]{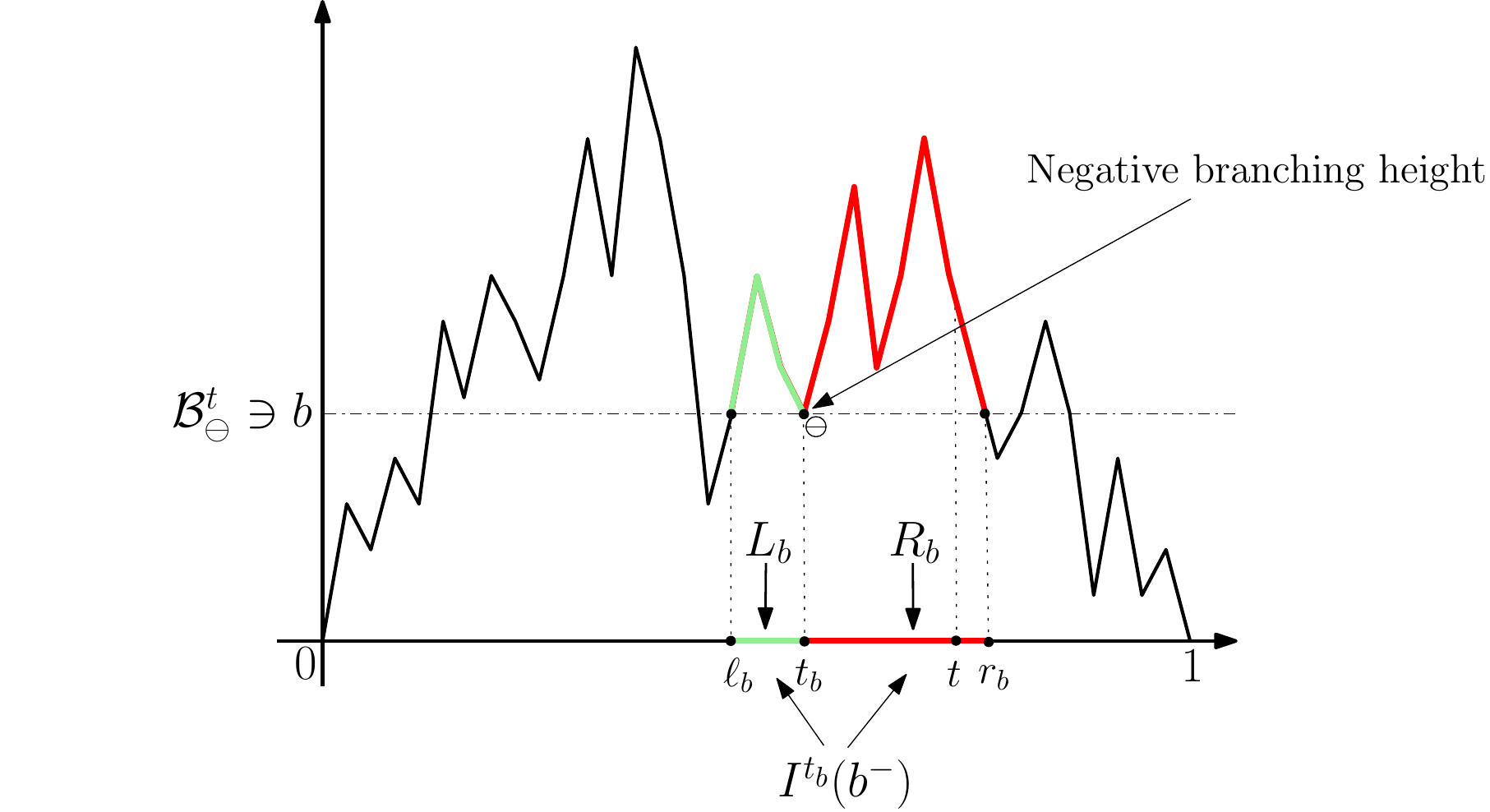}  
		\caption{\label{fig:excursion-split2} A sketch for the notation introduced in \cref{sec:strategy_upper_bound}.}
	\end{center}
	\vspace{-3ex}
\end{figure}

We now introduce a particular type of sequences; the motivation will be clarified right after their definition. Fix $p\in(0,1)$, \textit{i.e.}\ fix the parameter for the signs of the signed excursion $(\efrak, \sfrak,p)$. Conditioning on $(\efrak, \sfrak)$, let $\mathcal{D}$ (for \emph{discarding rules}, see explanations below) be the set of sequences $\eta=\{\eta_b\}_{b\in \mcl B_{\ominus}}$, with $\eta_b\in \{L_{b},R_{b},\diamond\}$, such that $\eta_b=\diamond$ if and only if there exists $b'\in \mcl B_{\ominus}$ with $b'< b$ such that $t_b\in \eta_{b'}$. 

Note that given $(\efrak, \sfrak)$, we can think of $\eta=\{\eta_b\}_{b\in \mcl B_{\ominus}}$
as a deterministic \textbf{discarding rule}, where each $\eta_{b}$ determines which side of the interval $(\ell_b,r_b)$ we are going to discard (the case $\eta_b=\diamond$ corresponds to the case when the interval $(\ell_b,r_b)$ is contained in an interval that was already discarded at some smaller negative branching height).

We introduce some notation. Given $t\in[0,1]$ and $b\in \mcl B^t_{\ominus}$, we say that \textbf{$\eta_b$ does not discard $t$} if $t \notin \eta_b$. 
We want to estimate the probability of the following events defined for all $t\in(0,1)$ and $\eps\in(0,1)$,
\begin{equation}\label{eq:event_key}
	E_{\eps}^{t}=E_{\eps}^{t}(\eta):=\left\{\eta_b\text{ does not discard } t \text{ for all } b\in\mcl B^t_{\ominus} \text{ such that } F^t(b)\geq \eps \right\}. 
\end{equation}
More precisely, we are interested in the case when $t$ is a uniform point $U$ in $(0,1)$ independent of all other random quantities. The ideas used in this section to estimate the probability of the above event will then play a key role in \cref{sect:up_bound_seq_perm}.

Recall the subordinator $\xi$ introduced in \cref{prop: law tagged}. In this section we use the decomposition $\xi=\xi^{\oplus}+\xi^{\ominus}$, where $\xi^{\oplus}$ (resp.\ $\xi^{\ominus}$) is the process determined by the jumps of $\xi$ corresponding to the local minima of $\efrak$ decorated by $\oplus$--signs (resp.\ $\ominus$--signs).  
Then, arguing as in \eqref{eq: Lambda^S}, the process $\xi^{\ominus}$ is a Lévy process with intensity measure on $(0,+\infty)$
\begin{equation} \label{eq: sec 4 Lambda^-}
	\Lambda^\ominus(\mathrm{d}x) 
	:= (1-p)\Lambda(\mathrm{d}x)
	=  \frac{2 (1-p)\,\mathrm{e}^x\mathrm{d}x}{\sqrt{2\pi(\mathrm{e}^x-1)^3}}.
\end{equation}
Now recall from \cref{rem:expr-bounds} the definition of $\lambda^*(p)=1-\beta^*(p)$ as 
\begin{equation} \label{eq: sec 4 def lambda*}
	\lambda^*(p)=\sup_{\beta\in(0,\log(2)),\delta>0}\min\left\{\beta\delta\,,\,\sup_{\gamma<0}\{\gamma \delta + \kappa^{*}_{\gamma,\mathrm{e}^{-\beta}}(p)\}\right\},
\end{equation}
and note that for all $\gamma<0$ and all $r\in(1/2,1)$, by simple calculations, the equation \eqref{eq:expr_kappa} satisfied by $\kappa^{*}_{\gamma,r}(p)$ can be rephrased more conveniently as
\begin{equation} \label{eq: sec 4 def kappa}
	\Phi(-\kappa^{*}_{\gamma,r}(p))=-(1-\mathrm{e}^{\gamma})\Lambda^\ominus\big(-\log(r),-\log(1-r)\big),
\end{equation}
where we recall from \cref{rem:lap_exp_expl} that $\Phi(q) = 2\sqrt{2} \frac{\Gamma(q+1/2)}{\Gamma(q)}$. Note that the above equation has a unique positive solution because as explained in \cref{rem:lap_exp_expl}, the function $\Phi(\cdot)$ is increasing, $\Phi(0)=0$, $\lim_{q\to (-1/2)^+}\Phi(q) = -\infty$ and $-(1-\mathrm{e}^{\gamma})\Lambda^\ominus\big(-\log(r),-\log(1-r)\big)<0$ for all $\gamma<0$ and all $r\in(1/2,1)$.

\medskip

We are going to prove (after stating and proving two additional complementary lemmas) the following upper bound for the probability of the event $E_{\eps}^{U}$ introduced in \eqref{eq:event_key}.

\begin{prop}\label{prop:est_up_bound_lis(q)}
	Fix $p\in(0,1)$. Let $\eta\in \mathcal{D}$ be a discarding rule chosen in a $(\efrak, \sfrak )$--measurable manner. Then,  
	\begin{equation*}
		\bbP(E_{\eps}^{U})\leq 2\, \eps^{\lambda^*(p)}, \quad\text{for all } \eps\in(0,1).
	\end{equation*}	
\end{prop}

The following two results complement the bound obtained in the above proposition. 

\begin{lem}\label{lem:param_is pos}
	For all $p\in(0,1)$, we have that $\lambda^*(p)>0$.
\end{lem}

\begin{proof}[Proof of \cref{lem:param_is pos}]
	Fix $p\in(0,1)$. Recall the definition of $\lambda^*(p)$ from \eqref{eq: sec 4 def lambda*}.
	We show that there exist $\beta\in(0,\log(2))$ and $\delta>0$ such that $\sup_{\gamma<0}\{\gamma \delta + \kappa^{*}_{\gamma,\mathrm{e}^{-\beta}}(p)\}>0$, then the desired result follows. Fix any $\gamma_0<0$ and $\beta_0\in(0,\log(2))$. Note that since  $\kappa^{*}_{\gamma,\mathrm{e}^{-\beta}}(p)$ is defined as the only positive solution to \eqref{eq: sec 4 def kappa} with $r=\mathrm{e}^{-\beta}$, we have that $\kappa^{*}_{\gamma_0,\mathrm{e}^{-\beta_0}}(p)>0$. Therefore there always exists $\delta>0$ small enough such that $\gamma_0 \delta + \kappa^{*}_{\gamma_0,\mathrm{e}^{-\beta_0}}(p)>0$. This concludes the proof.
\end{proof}

\begin{lem}\label{lem: sup delta-beta attained}
The supremum in \eqref{eq: sec 4 def lambda*} is attained, \textit{i.e.}\ there exist $\beta=\beta(p)\in(0,\log (2))$ and $\delta=\delta(p)>0$ such that 
\[
\lambda^*(p)
=
\min\left\{\beta\delta\,,\,\sup_{\gamma<0}\{\gamma \delta + \kappa^{*}_{\gamma,\mathrm{e}^{-\beta}}(p)\}\right\}.
\]
\end{lem}

\begin{proof}
To see this, let 
\[
\mu(\delta,\beta)
:=
\min\left\{\beta\delta\,,\,\sup_{\gamma<0}\{\gamma \delta + \kappa^{*}_{\gamma,\mathrm{e}^{-\beta}}(p)\}\right\},
\quad
\delta>0, \; \beta\in(0,\log(2)),
\]
and $\delta_n>0,\beta_n\in(0,\log 2)$ define two sequences such that $\mu(\delta_n,\beta_n) \to \lambda^*(p)$ as $n\to \infty$. Our argument is to show that we can extract from $(\delta_n, n\in\bbN)$ and $(\beta_n, n\in\bbN)$ converging subsequences with limits in $(0,\infty)$ and $(0,\log(2))$ respectively (note that the extremities of the intervals are excluded).

First of all, $\mu(\delta_n,\beta_n) \le \beta_n \delta_n \le \delta_n\log(2)$, and since $\mu(\delta_n,\beta_n)\to \lambda^*(p)>0$ (see \cref{lem:param_is pos}), it must be the case that $(\delta_n, n\in\bbN)$ is bounded away from $0$. Secondly, it is fairly easy to see that $(\beta_n, n\in\bbN)$ is bounded away from $\log(2)$. In fact, for all $n\in\bbN$, $\mu(\delta_n,\beta_n)\le\sup_{\gamma<0}\{\gamma \delta_n + \kappa^{*}_{\gamma,\mathrm{e}^{-\beta_n}}(p)\} \le \sup_{\gamma<0} \kappa^{*}_{\gamma,\mathrm{e}^{-\beta_n}}(p)$. But note that for all $r\in(1/2,1)$, $\sup_{\gamma<0}\kappa^{*}_{\gamma,r}(p) = \kappa^{*}_{\infty,r}(p)$, where $\kappa^{*}_{\infty,r}(p)$ is the unique positive solution to 
\[
	\Phi(-\kappa^{*}_{\infty,r}(p))=-\Lambda^\ominus\big(-\log(r),-\log(1-r)\big).
\]
It is clear that $\kappa^{*}_{\infty,r}(p)\to 0$ as $r\to 1/2$. Therefore, $\beta_n\to\log(2)$ (up to extraction) would imply $\mu(\delta_n,\beta_n) \le \kappa^{*}_{\infty,\mathrm{e}^{-\beta_n}}(p) \to 0$ as $n\to\infty$. This contradicts $\lambda^*(p)>0$.

Next, we claim that $(\beta_n, n\in\bbN)$ is bounded away from $0$. Suppose that this is not the case; without loss of generality we can assume that $\beta_n\to 0$. Take $n$ large enough so that $\mu(\delta_n,\beta_n) \ge \frac{\lambda^*(p)}{2}$. The bound $\frac{\lambda^*(p)}{2}\le \mu(\delta_n,\beta_n) \le \beta_n \delta_n$ shows that $\delta_n\to\infty$. Then, for all $n$, we pick $\gamma_n<0$ such that, 
\begin{equation}\label{eq: def gamma_n}
\gamma_n \delta_n + \kappa^*_{\gamma_n,\mathrm{e}^{-\beta_n}}(p) \ge \frac12 \sup_{\gamma<0}\{\gamma \delta_n + \kappa^{*}_{\gamma,\mathrm{e}^{-\beta_n}}(p)\}. 	
\end{equation}
This is possible since $\sup_{\gamma<0}\{\gamma \delta_n + \kappa^{*}_{\gamma,\mathrm{e}^{-\beta_n}}(p)\} \ge \mu(\delta_n,\beta_n)\ge \frac{\lambda^*(p)}{2}>0$. We first remark that, since $\kappa^*_{\gamma_n,\mathrm{e}^{-\beta_n}}(p) \le \frac12$ (recall that $\lim_{q\to (-1/2)^+}\Phi(q) = -\infty$),
\begin{equation} \label{eq: lb gammadelta}
\gamma_n\delta_n 
\stackrel{\eqref{eq: def gamma_n}}{\ge}
\frac12 \Big(\sup_{\gamma<0}\{\gamma \delta_n + \kappa^{*}_{\gamma_n,\mathrm{e}^{-\beta_n}}(p)\}-1\Big)
\ge
\frac12 \Big(\frac{\lambda^*(p)}{2}-1\Big),
\end{equation}
and hence $\gamma_n \to 0$ because $\gamma_n<0$, $\delta_n \to \infty$ and $\frac12 \Big(\frac{\lambda^*(p)}{2}-1\Big)\ge-\frac{1}{2}$. Now by definition of $\kappa^*_{\gamma_n,\mathrm{e}^{-\beta_n}}(p)$, 
\[
	\Phi(-\kappa^*_{\gamma_n,\mathrm{e}^{-\beta_n}}(p))
	=
	-(1-\mathrm{e}^{\gamma_n})\Lambda^\ominus\big(\beta_n,-\log(1-\mathrm{e}^{-\beta_n})\big).
\]
A simple integral estimate shows that, as $n\to \infty$, $\Lambda^\ominus\big(\beta_n,-\log(1-\mathrm{e}^{-\beta_n})\big) \sim c/\sqrt{\beta_n}$ for some positive constant $c$, whence $\Phi(-\kappa^*_{\gamma_n,\mathrm{e}^{-\beta_n}}(p)) \sim c\gamma_n /\sqrt{\beta_n}$ as $n\to\infty$. Note that by the inequality $\frac{\lambda^*(p)}{2}\le \mu(\delta_n,\beta_n) \le \beta_n \delta_n$ and the fact that $\gamma_n<0$,
\[
\frac{\gamma_n}{\sqrt{\beta_n}} 
\ge 	
\sqrt{\frac{2}{\lambda^*(p)}} \cdot \gamma_n \sqrt{\delta_n}
\stackrel{\eqref{eq: lb gammadelta}}{\ge} 
\frac12\sqrt{\frac{2}{\lambda^*(p)}} \cdot \frac{\frac{\lambda^*(p)}{2}-1}{\sqrt{\delta_n}}.
\]
Since $\gamma_n<0$, this entails $\gamma_n /\sqrt{\beta_n}\to 0$ and hence $\Phi(-\kappa^*_{\gamma_n,\mathrm{e}^{-\beta_n}}(p)) \to 0$. Therefore $\kappa^*_{\gamma_n,\mathrm{e}^{-\beta_n}}(p)\to 0$. But this is a contradiction: indeed,
\begin{equation} \label{eq: contradiction kappa^*}
	\frac{\lambda^*(p)}{2}
	\le 
	\mu(\delta_n,\beta_n)
	\le
	\sup_{\gamma<0}\{\gamma \delta_n + \kappa^{*}_{\gamma_n,\mathrm{e}^{-\beta_n}}(p)\}
	\stackrel{\eqref{eq: def gamma_n}}{\le} 
	2\gamma_n\delta_n + 2\kappa^*_{\gamma_n,\mathrm{e}^{-\beta_n}}(p)
	\le
	2\kappa^*_{\gamma_n,\mathrm{e}^{-\beta_n}}(p),
\end{equation}
so that we would have $\lambda^*(p)=0$.

Finally, we prove that $(\delta_n, n\in\bbN)$ is upper bounded using a similar argument. We argue by contradiction, assuming that $\delta_n \to \infty$ (this is possible up to extraction). We first use the inequality $\mu(\delta_n,\beta_n) \le \sup_{\gamma<0}\{\gamma \delta_n + \kappa^{*}_{\gamma,\mathrm{e}^{-\beta_n}}(p)\}$ for all $n\in\bbN$, and we consider $\gamma_n<0$ as in \eqref{eq: def gamma_n}. Now the same argument as in \eqref{eq: lb gammadelta} shows that $\gamma_n\to 0$. Since we know from the previous paragraph that $(\beta_n, n\in\bbN)$ is bounded away from $0$, we deduce that, as $n\to\infty$, 
\[
	\Phi(-\kappa^*_{\gamma_n,\mathrm{e}^{-\beta_n}}(p))
	=
	-(1-\mathrm{e}^{\gamma_n})\Lambda^\ominus\big(\beta_n,-\log(1-\mathrm{e}^{-\beta_n})\big) \longrightarrow 0.
\]
Therefore $\kappa^*_{\gamma_n,\mathrm{e}^{-\beta_n}}\to 0$ and we conclude as in \eqref{eq: contradiction kappa^*} that this is impossible.

We have proved that $(\delta_n, n\in\bbN)$ is upper bounded and bounded away from $0$, and that $(\beta_n,n\in\bbN)$ is bounded away both from $\log(2)$ and $0$. We may now conclude the proof. Up to extraction, we can assume that $(\delta_n,n\in\bbN)$ and $(\beta_n,n\in\bbN)$ converge to $\delta_\infty \in(0,\infty)$ and $\beta_\infty\in(0,\log(2))$ respectively. Taking a limit, we finally get $\lambda^*(p) = \underset{n\to\infty}{\lim} \mu(\delta_n,\beta_n) = \mu(\delta_\infty,\beta_\infty)$, which proves our statement.
\end{proof}

We now turn to the proof of \cref{prop:est_up_bound_lis(q)}. We highlight that the two main estimates used in this proof will be also used in \cref{sect:up_bound_seq_perm}.

\begin{proof}[Proof of \cref{prop:est_up_bound_lis(q)}]
	For a constant $r\in(1/2,1)$, we consider the sequence of branching heights defined for all $m>0$ by
	\begin{equation}\label{eq:seq_time}
		\tau_{m}^U=\tau_{m}^U(r):=\text{$m$-th smallest $b\in\mcl B^U_{\ominus}$ such that $\frac{\max\{|L_b|,|R_b|\}}{|L_b|+|R_b|}\leq r$}.
	\end{equation}
	The random variables $\tau_{m}^U$ are a.s.\ well defined for all $m\in\bbN$, because we claim that there are infinitely many branching heights  $b\in\mcl B^U_{\ominus}$ such that $\max\{|L_b|,|R_b|\}\leq r (|L_b|+|R_b|)$. Indeed, we first note that $\max\{|L_b|,|R_b|\}\leq r (|L_b|+|R_b|)$ if and only if $ \max\{\mathrm{e}^{-\xi^\ominus_{s^-}}-\mathrm{e}^{-\xi^\ominus_s}, \mathrm{e}^{-\xi^\ominus_s}\}\leq r \cdot \mathrm{e}^{-\xi^\ominus_{s^-}}$, where $s=\rho(b)$, and that this is equivalent to $\Delta\xi^\ominus_s\in(-\log(r),-\log(1-r))$. Since the number of jumps of the process $\xi^\ominus$ by time $t$ and with sizes in $(-\log(r),-\log(1-r))$ forms a Poisson process with mean $\Lambda^\ominus(-\log(r),-\log(1-r))\cdot t$ (see \textit{e.g.} \cite[Section I.5]{bertoin1996levy}), letting $t\to \infty$ we get the claim.
	
	We also set for all $m>0$
	\begin{equation}\label{eq:itermediate}
		A^U_{m}=A^U_{m}(r):=\{\eta_b\text{ does not discard } U \text{ for all } b\in\mcl B^U_{\ominus} \text{ such that } b\leq\tau^U_{m}\},
	\end{equation}
	and we claim that for all $r\in(1/2,1)$, $m>0$ and $\eps\in(0,1)$, 
	\begin{equation}\label{eq:incl_eve}
		E_{\eps}^{U}\subseteq \left(A^U_m \cap\left\{F^U(\tau_m^U)\geq \eps\right\}\right)\cup \left\{F^U(\tau_m^U)< \eps\right\}.
	\end{equation}
	Indeed, if $E_{\eps}^{U}$ occurs -- \textit{i.e.}\ $\eta_b$ does not discard  $U$  for all  $b\in\mcl B^U_{\ominus}$ such that $F^U(b)\geq \eps$ -- and $F^U(\tau_m^U)\geq \eps$, then $F^U(b)\geq \eps$ for all $b\in\mcl B^U_{\ominus}$ such that $b\leq \tau_m^U$. 
	Hence $\eta_{b}$ must not discard $U$ for all for all $b\in\mcl B^U_{\ominus}$ such that $b\leq \tau_m^U$, which is the event $A^U_m$. 
	
	As a consequence, for all $r\in(1/2,1)$, $m>0$ and $\eps\in(0,1)$, 
	\begin{equation}\label{eq:upper_bnd_proba2}
		\bbP\left(E_{\eps}^{U}\right)
		\leq \bbP\left(A^U_m\right) 
		+\bbP\left( F^U\big(\tau_m^U\big) < \eps\right).
	\end{equation}
	In the next two crucial propositions (whose proofs are postponed to \cref{sect:proof_lemmas}), we upper-bound the two terms on the right-hand side of the last equation. Later, we will choose some specific values for the constants $r\in(1/2,1)$ and $m>0$ which will optimize our upper-bound.

	\begin{prop}\label{prop:first step2}
		For all $r\in(1/2,1)$, $m>0$, 
		\begin{equation*}
			\bbP\left(A^U_m\right) \leq r^m.
		\end{equation*}
	\end{prop}
	
	\begin{prop}\label{prop:second step2}
		For all $r\in(1/2,1)$, $m>0$ and $\eps\in(0,1)$,
		\[\bbP\left(F^U\big(\tau^U_m\big)< \eps\right)\leq \inf_{\gamma<0}\left\{\mathrm{e}^{-\gamma m}\cdot \eps^{\kappa^{*}_{\gamma,r}(p)}\right\},\]
		where $\kappa^{*}_{\gamma,r}(p)$ was defined in \eqref{eq: sec 4 def kappa}.
	\end{prop}

	We now conclude the proof of \cref{prop:est_up_bound_lis(q)}. Combining the estimates in the last two propositions with \eqref{eq:upper_bnd_proba2}, we get that for all $r\in(1/2,1)$, $m>0$ and $\eps\in(0,1)$, 
	\begin{equation*}
		\bbP\left(E_{\eps}^{U}\right)\leq r^m +
		\inf_{\gamma<0}\left\{\mathrm{e}^{-\gamma m}\cdot \eps^{\kappa^{*}_{\gamma,r}(p)}\right\}.
	\end{equation*}
	Setting $r=\mathrm{e}^{-\beta}$ with $\beta\in(0,\log(2))$ and $m=\delta\log(1/\eps)$ with $\delta>0$, we deduce that for all $\eps\in(0,1)$
	\begin{equation*}
		\bbP\left(E_{\eps}^{U}\right)\leq \eps^{\beta\delta}+
		\inf_{\gamma<0}\left\{\eps^{\gamma \delta + \kappa^{*}_{\gamma,\mathrm{e}^{-\beta}}(p)}\right\}\leq 2\, \eps^{\min\{\beta\delta\,,\,\sup_{\gamma<0}\{\gamma \delta + \kappa^{*}_{\gamma,\mathrm{e}^{-\beta}}(p)\}\}}.
	\end{equation*} 
	In particular, for all $\eps\in(0,1)$,
	\begin{equation*}
		\bbP\left(E_{\eps}^{U}\right)\leq 2\, \eps^{\lambda^*(p)},
	\end{equation*} 
	where we recall from \eqref{eq: sec 4 def lambda*} that
	\begin{equation*}
	\lambda^*(p)=\sup_{\beta\in(0,\log(2)),\delta>0}\min\left\{\beta\delta\,,\,\sup_{\gamma<0}\{\gamma \delta + \kappa^{*}_{\gamma,\mathrm{e}^{-\beta}}(p)\}\right\}.
	\end{equation*}
	This concludes the proof of \cref{prop:est_up_bound_lis(q)}.
\end{proof}

An immediate consequence of the last proof and Lemmas \ref{lem:param_is pos} and \ref{lem: sup delta-beta attained} is the following result, which contains the key estimates for the results in \cref{sect:up_bound_seq_perm}.

\begin{cor}\label{cor:key_estimates}
	Fix $p\in(0,1)$. Let $\eta\in \mathcal{D}$ be a discarding rule chosen in a $(\efrak, \sfrak )$--measurable manner and $U$ be a uniform random variable in $(0,1)$ independent of all other random quantities. Recall the definitions of $\tau^U_m$ and $A^U_m$ from \eqref{eq:seq_time} and \eqref{eq:itermediate}.
	
	There exist a choice of $r = r(p)\in(1/2,1)$ and $\delta = \delta(p)>0$ such that for all $\eps\in(0,1)$, setting $m = \delta \log(1/\eps)$, we have that
	\begin{equation*}
		\bbP\left(A^U_m\right) 
		\leq 
		2\, \eps^{\lambda^*(p)},
	\end{equation*}
	and that
	\begin{equation*}
		\bbP\left(F^U\big(\tau^U_m\big)< \eps\right) 
		\leq 
		2\, \eps^{\lambda^*(p)},
	\end{equation*}
	where $\lambda^*(p)>0$ is defined as in \cref{eq: sec 4 def lambda*}.
\end{cor}

\subsection{Proofs of the two main estimates}\label{sect:proof_lemmas}

In this section we prove the two propositions that we left behind in the previous section, \textit{i.e.}\ Propositions \ref{prop:first step2} and \ref{prop:second step2}. For convenience, we recall at the beginning of each proof what is the statement that we need to prove.

\begin{proof}[Proof of \cref{prop:first step2}]
	Recall that $\eta\in \mathcal{D}$ is a discarding rule chosen in a $(\efrak, \sfrak )$--measurable manner.
	Also recall from \eqref{eq:itermediate} that for all $m>0$,
	\begin{equation}
		A^U_{m}=A^U_{m}(r):=\{\eta_b\text{ does not discard } U \text{ for all } b\in\mcl B^U_{\ominus} \text{ such that } b\leq\tau^U_{m}\},
	\end{equation}
	where $\tau_m^U=\tau_m^U(r)$ was defined in \eqref{eq:seq_time}.
	We want to prove that for all $r\in(1/2,1)$, $m>0$, 
	\begin{equation}\label{eq:rn_bound}
		\bbP\left(A^U_m\right) \leq  r^m.
	\end{equation}
	Setting 
	\[B^U_{m}=B^U_{m}(r):=\{\eta_b\text{ does not discard } U \text{ for all } b\in\mcl B^U_{\ominus} \text{ such that } b<\tau^U_m\},\]
	(note that the inequality is strict here), we have that
	\begin{equation}\label{eq:simply_event}
		A^U_{m}=B^U_{m}\cap \left\{\eta_{\tau_m^U}\text{ does not discard } U\right\}.
	\end{equation}
	We consider the following filtration, defined for all $m>0$ by
	\[\mcl F_m:=\sigma\left(\efrak, \sfrak, (I^U(h))_{h<\tau^U_{m+1}},(\tau^U_{s})_{s\leq m+1}\right).\] 
	We emphasize that, importantly, the interval $I^U(h)$ at level $h=\tau_{m+1}^U$ is not included in the definition of $\mcl F_{m}$. Fix $m>0$. Note that:
	\begin{itemize}
		\item $B^U_{m}$ is $\mcl F_{m-1}$--measurable.
		\item $I^U((\tau^U_m)^-)$ is $\mcl F_{m-1}$--measurable because by definition it is equal to the interior of the decreasing intersection $\bigcap_{h<\tau^U_m}I^{U}(h)$.
		\item $L_{\tau_m^U}$ and $R_{\tau_m^U}$ are $\mcl F_{m-1}$--measurable because $\tau_m^U$ is  $\mcl F_{m-1}$--measurable and $L_{\tau_m^U}$ and $R_{\tau_m^U}$ are a deterministic function of $\tau_m^U$ and $\efrak$.
		\item $|I^U((\tau^U_m)^-)|=|L_{\tau_m^U}|+|R_{\tau_m^U}|$.
		\item Conditioning on $\mcl F_{m-1}$, $U$ is uniform in $I^U((\tau^U_m)^-)$ because $I^U(h)$ at level $h=\tau_{m}^U$ is not included in the definition of $\mcl F_{m-1}$. 
		\item Conditioning on $\mcl F_{m-1}$, the probability that $U$ falls in $L_{\tau_m^U}$ (resp. $R_{\tau_m^U}$) is therefore $\frac{|L_{\tau_m^U}|}{|L_{\tau_m^U}|+|R_{\tau_m^U}|}$ (resp. $\frac{|R_{\tau_m^U}|}{|L_{\tau_m^U}|+|R_{\tau_m^U}|}$).
	\end{itemize}
	The above observations entail that, almost surely, for all $m>0$, 
	\begin{align}\label{eq: sec 4 uniform re-sampling}
		\mathds{1}_{B^U_{m}}\,\bbE&\left[\mathds{1}_{\{\eta_{\tau_m^U}\text{ does not discard } U\}}\middle|\mcl F_{m-1}\right]\notag\\
		&=
		\mathds{1}_{B^U_{m}}\,\bbP\left(U\text{ is not in the interval discarded by }\eta_{\tau_m^U} \middle|\mcl F_{m-1}\right)\notag\\
		&=
		\mathds{1}_{B^U_{m}}\,\left(\frac{|L_{\tau_m^U}|}{|L_{\tau_m^U}|+|R_{\tau_m^U}|}\delta_{\{\eta_{\tau_m^U}=R_{\tau_m^U}\}}+\frac{|R_{\tau_m^U}|}{|L_{\tau_m^U}|+|R_{\tau_m^U}|}\delta_{\{\eta_{\tau_m^U}=L_{\tau_m^U}\}}\right)\notag\\
		&\leq
		\mathds{1}_{B^U_{m}}\,\frac{\max\{|L_{\tau_m^U}|,|R_{\tau_m^U}\}|}{|L_{\tau_m^U}|+|R_{\tau_m^U}|}
		\leq \mathds{1}_{B^U_{m}}\,r,
	\end{align}
	where the last inequality follows by definition of $\tau_m^U$.
	Note also that in the second equality of the above equation we used that if $\eta_{b}$ does not discard  $U$ for all $b\in\mcl B^U_{\ominus}$ such that $b <\tau_m^U$ (which is the event $B^U_{m}$) then $\eta_{\tau_m^U}\neq \diamond$ by the definition of discarding rule. 
	Therefore, we get that for all $m>0$,
	\begin{multline*}
		\bbP(A^U_m)\stackrel{\eqref{eq:simply_event}}{=}\bbE\left[\mathds{1}_{B^U_m}\mathds{1}_{\{\eta_{\tau_m^U}\text{ does not discard } U\}}\right]
		=
		\bbE\left[\mathds{1}_{B^U_m}\bbE\Big(\mathds{1}_{\{\eta_{\tau_m^U}\text{ does not discard } U\}}\Big|\mathcal F_{m-1}\Big)\right]\\
		\stackrel{\eqref{eq: sec 4 uniform re-sampling}}{\leq}
		r\cdot\bbE\left[\mathds{1}_{B^U_m}\right]
		\leq
		r\cdot\bbE\left[\mathds{1}_{B^U_{m-1}}\mathds{1}_{\{\eta_{\tau_{m-1}^U}\text{ does not discard } U\}}\right]\stackrel{\eqref{eq:simply_event}}{=} r\cdot \bbP(A^U_{m-1}),
	\end{multline*}
	where in the second equality we used that $\mathds{1}_{B_m}$ is $\mcl F_{m-1}$--measurable.
	Iterating the same argument, we retrieve \eqref{eq:rn_bound}. 
\end{proof}

	\begin{proof}[Proof of \cref{prop:second step2}]
		We want to prove that for all $r\in(1/2,1)$, $m>0$ and $\eps\in(0,1)$,
		\[\bbP\left(F^U\big(\tau^U_m\big)< \eps\right)\leq \inf_{\gamma<0}\left\{\mathrm{e}^{-\gamma m}\cdot \eps^{\kappa^{*}_{\gamma,r}(p)}\right\},\]
		where $\kappa^{*}_{\gamma,r}(p)$ was defined in \eqref{eq: sec 4 def kappa}.
		We fix $r\in(1/2,1)$, $m>0$ and $\eps\in(0,1)$. Let 
		\[
		\mcl J^U_{\eps,r}=\left\{b\in \mcl B^U_{\ominus}\;\middle| \; F^U(b^-)\geq \eps \text{ and } \frac{\max\{|L_b|,|R_b|\}}{|L_b|+|R_b|}\leq r\right\}.
		\]
		We start by noticing that $\{F^U(\tau^U_m)< \eps \} \subseteq \{ \#\mcl J^U_{\eps,r}\leq m \}$. By Chernoff bound, for all $\gamma<0$, we have
		\begin{equation} \label{eq: sec 4 Chernoff q2}
			\bbP\left(F^U\big(\tau^U_m\big)< \eps\right)
			\leq
			\bbP\left(\#\mcl J^U_{\eps,r}\leq m \right)
			\le
			\mathrm{e}^{-\gamma m}\bbE\left[ \mathrm{e}^{\gamma \#\mcl J^U_{\eps,r}}\right].
		\end{equation}
		Fix $\gamma<0$. Recall the notation $\Delta\xi^{\ominus}_s=\xi^{\ominus}_s-\xi^{\ominus}_{s^-}$, where the process $\xi^{\ominus}_s$ was introduced above \eqref{eq: sec 4 Lambda^-}.
		We claim that
		\begin{equation}\label{eq:size_set_eq}
			\#\mcl J^U_{\eps,r}=\sum_{s > 0}\mathds{1}_{\xi_{s^-}\leq \log(1/\eps)}\mathds{1}_{\Delta\xi^{\ominus}_s\in(r_0,r_1)},
		\end{equation}
		where $r_0=-\log(r)$ and $r_1=-\log(1-r)$.
		Indeed, note that for all $b\in \mcl B^U_{\ominus}$,  $F^U(b^-)\geq \eps$ if and only if $\xi_{s^-}\leq \log(1/\eps)$, with $s=\rho(b)$.
		Moreover, we have that $\max\{|L_b|,|R_b|\}\leq r (|L_b|+|R_b|)$ if and only if $\Delta\xi^\ominus_s\in(-\log(r),-\log(1-r))$, as already explained below \eqref{eq:seq_time}.
		Hence we proved the equality \eqref{eq:size_set_eq}. The latter implies that 
		\begin{equation} \label{eq: sec3 formula sum2}
			\bbE\left[\mathrm{e}^{\gamma \#\mcl J^U_{\eps,r}}\right]
			=
			\bbE\left[\exp\left\{\gamma\sum_{s > 0}\mathds{1}_{\xi_{s^-}\leq \log(1/\eps)}\mathds{1}_{\Delta\xi^{\ominus}_s\in(r_0,r_1)}\right\}\right].
		\end{equation}
		By definition of $\xi^\ominus$, the law of $\xi^\ominus$ is obtained from the law of $\xi$ by keeping the jumps of $\xi$ according to i.i.d.\ coin tosses with success probability $1-p$. More precisely, if $(\chi_s, s>0)$ denotes a collection of i.i.d.\ coin flips, with $\bbP(\chi=1)=1-p = 1 - \bbP(\chi=0)$, and that is further independent of $\xi$, we have the identity in law
		\begin{equation} \label{eq: sec3 thinning ominus2}
			\sum_{s > 0}\mathds{1}_{\xi_{s^-}\leq \log(1/\eps)}\mathds{1}_{\Delta\xi^{\ominus}_s\in(r_0,r_1)}
			\overset{d}{=}
			\sum_{s > 0}\chi_s\mathds{1}_{\xi_{s^-}\leq \log(1/\eps)}\mathds{1}_{\Delta\xi_s\in(r_0,r_1)}.
		\end{equation}
		We combine equations \eqref{eq: sec3 formula sum2} and \eqref{eq: sec3 thinning ominus2} into
		\[
		\bbE\left[\mathrm{e}^{\gamma \#\mcl J^U_{\eps,r}}\right]
		=
		\bbE\left[\exp\left\{\gamma\sum_{s > 0}\chi_s\mathds{1}_{\xi_{s^-}\leq \log(1/\eps)}\mathds{1}_{\Delta\xi_s\in(r_0,r_1)}\right\}\right].
		\]
		Using the independence of $\chi$ and $\xi$, and the fact that $\chi$ is a collection of i.i.d.\ Bernoulli variables, we obtain
		\begin{align*}
			\bbE\left[\mathrm{e}^{\gamma \#\mcl J^U_{\eps,r}}\right] 
			&=
			\bbE\left[\prod_{s>0}\left(p+(1-p)\exp\left\{\gamma\mathds{1}_{\xi_{s^-}\leq \log(1/\eps)} \mathds{1}_{\Delta\xi_s\in(r_0,r_1)}\right\}\right)\right] \\
			&= \bbE\left[ \exp\left\{\sum_{s>0}\log\left(p+(1-p)\exp\left\{\gamma\mathds{1}_{\xi_{s^-}\leq \log(1/\eps)} \mathds{1}_{\Delta\xi_s\in(r_0,r_1)}\right\}\right)\right\}\right].
		\end{align*}
		Using the exponential formula for $\xi$ (see for instance \cite[Proposition XII.1.12]{revuz2013continuous}), the previous display reduces to
		\begin{align*}
			\bbE\left[\mathrm{e}^{\gamma \#\mcl J^U_{\eps,r}}\right]
			&=
			\bbE\left[\exp\left\{-\int_0^\infty ds\int_0^\infty\Lambda(dx)\left(1-\left(p+(1-p)\exp\left\{\gamma\mathds{1}_{\xi_{s^-}\leq \log(1/\eps)} \mathds{1}_{x\in(r_0,r_1)}\right\}\right)\right)\right\}\right]\\
			&=
			\bbE\left[\exp\left\{- T_{\log(1/\eps)} \int_0^\infty \Lambda(dx)\left(1-p-(1-p)\exp\left\{\gamma \mathds{1}_{x\in(r_0,r_1)}\right\}\right)\right\}\right]\\
			&=\bbE\left[\exp\left\{- T_{\log(1/\eps)} (1-\mathrm{e}^{\gamma}) \Lambda^{\ominus}(r_0,r_1)\right\}\right],
		\end{align*}
		where in the last equality we used that for an indicator $\mathds{1}_{A}$, $1-p-(1-p)\mathrm{e}^{\gamma\mathds{1}_{A}}=(1-p)(1-\mathrm{e}^{\gamma})\mathds{1}_{A}$  and that $\Lambda^\ominus(\mathrm{d}x) = (1-p)\Lambda(\mathrm{d}x)$ from \eqref{eq: sec 4 Lambda^-}.
		Now let $\kappa^{*}_{\gamma,r}(p)$ be defined as in \eqref{eq: sec 4 def kappa}, \textit{i.e.}\ as the only positive solution to the equation $\Phi(-\kappa^{*}_{\gamma,r}(p))=-(1-\mathrm{e}^{\gamma})\Lambda^\ominus(r_0,r_1)$. Using the second claim \eqref{eq: lem asymptotics subordinator upper bound} in \cref{lem: asymptotics subordinator} with $a=(1-\mathrm{e}^{\gamma}) \Lambda^{\ominus}(r_0,r_1)$ and $b=0$, we deduce the upper bound $\bbE\left[\mathrm{e}^{\gamma \#\mcl J^U_{\eps,r}}\right] \le \eps^{\kappa^*_{\gamma,r}(p)}$.
		Going back to our Chernoff estimate in \eqref{eq: sec 4 Chernoff q2}, we deduce that $\bbP\left(F^U\big(\tau^U_m\big)< \eps\right)\leq\mathrm{e}^{-\gamma m} \eps^{\kappa^*_{\gamma,r}(p,q)}$.
		Optimizing over $\gamma<0$ yields the desired claim.
	\end{proof}

\begin{remark}
Our estimate in \cref{prop:second step2} can be slightly improved by dealing directly with the left-hand side of \eqref{eq: sec 4 Chernoff q2}. The idea is to use the Lamperti representation (\cref{prop: law tagged}) to rephrase the event $\{F^U(\tau_m^U) <\eps\}$ in terms of the underlying Lévy process, and then use a thinning argument to select the special jumps corresponding to the times $\tau_i^U$. This would lead to a slightly better (but uglier) upper bound for the exponent for $\LIS(\Perm(\bm\mu_p,n))$.

We did not follow this other route, since we found the current argument slightly cleaner and in any case, both arguments provide bounds which are quite far from the actual behavior of $\LIS(\Perm(\bm\mu_p,n))$.
\end{remark}

\section{Upper bound for sequences sampled from the Brownian separable permutons}\label{sect:up_bound_seq_perm}

The main goal of this section is to prove the upper bound in \cref{thm:upper_lower_perm}, completing the proof of the theorem. Note that in \cref{lem:param_is pos} we proved that $\lambda^*(p)>0$, and so that $\beta^*(p)=1-\lambda^*(p)<1$ for all $p\in(0,1)$. Hence, in order to complete the proof of \cref{thm:upper_lower_perm} it remains to prove the upper bound in the second item in the theorem statement. This is done in the next proposition.

\begin{prop}\label{prop:upbound}
	Fix $p\in(0,1)$ and let $\beta^*(p)=1-\lambda^*(p)$ be as in \cref{rem:expr-bounds}. Let $\sigma_n$ be a random permutation of size $n\in \bbN$ sampled from the Brownian separable permuton $\bm{\mu}_p$. Then for all $\beta>\beta^*(p)$, the following convergence in probability holds
	\begin{equation*}
		\frac{\LIS(\sigma_n)}{n^{\beta}}\to 0.
	\end{equation*}
\end{prop}

\begin{proof}
	
	We split the proof in six steps. In what follows, \emph{w.h.p.}\ means \emph{with probability tending to one when $n\to\infty$}.

	\medskip
	
	\noindent\emph{\underline{Step 0}: Fixing the notation and setting our goal.}
	
	\medskip
	
	\begin{enumerate}[(a)]
		\item Fix  $n\in \bbN$. Let $(U_j)_{j\leq n}$ be a sequence of i.i.d.\ uniform random variables on $(0,1)$ and recall from \cref{sect:sampling} that $\sigma_n=\Perm(\efrak, \sfrak, (U_j)_{j\leq n})$;
		
		\item Fix $\zeta>0$ (small), and set $\eps=n^{-1+\zeta}$;
		
		\item Let  $r = r(p)\in(1/2,1)$ and $\delta = \delta(p)>0$ be as in \cref{cor:key_estimates};
		
		\item Finally, with the above choice of $\eps$ and $\delta$, set $m = \delta \log(1/\eps)$.
	\end{enumerate}	

	We will show that w.h.p.\ $\LIS(\sigma_n)$ is at most $4\, n^{1 - \lambda^*(p) +  3 \zeta}$. Note that this would be enough to conclude the proof.
		
	\medskip
 
	\noindent\emph{\underline{Step 1}: Introducing the discarding rule $\eta$.} 
	
	\medskip
	
	Recall the definition of the set of possible discarding rules $\mathcal{D}$ from the beginning of \cref{sec:strategy_upper_bound} and the notation in \cref{fig:excursion-split2}.
	Let  $\{U^*_\ell\}$ be the set of points in $(U_j)_{j\leq n}$ corresponding to the longest increasing subsequence in $\sigma_n$ (if there are multiple ones, we choose one arbitrarily).
	Let $\eta=\{\eta_b\}_{b\in \mcl B_{\ominus}}$ be the selection rule corresponding to this longest increasing subsequence. More precisely, let us choose $\eta\in \mathcal{D}$ in the following way for all $b\in \mcl B_{\ominus}$:
	\begin{itemize}
		\item if $\{U^*_\ell\}\cap L_b\neq \emptyset$ then set $\eta_b=R_{b}$;
		\item if $\{U^*_\ell\}\cap R_b\neq \emptyset$ then set $\eta_b=L_{b}$;
		\item if $\{U^*_\ell\}\cap L_b = \emptyset$ and $\{U^*_\ell\}\cap R_b= \emptyset$ and there exists $b'\in \mcl B_{\ominus}$ with $b'\leq b$ such that $b\in \eta_{b'}$ then set $\eta_b=\diamond$, otherwise set $\eta_b=L_b$ (note that this is an arbitrary choice).
	\end{itemize}
Here we emphasize that $\{U^*_\ell\}\cap L_b$ and $\{U^*_\ell\}\cap R_b$ cannot both be non-empty, since $\{U^*_\ell\}$ is increasing and $b\in\mcl B_{\ominus}$. Note also that $\eta$ is chosen in a $(\efrak,\sfrak,(U_j)_{j\leq n})$--measurable manner.

Recall now the definition of $\tau_{m}^U=\tau_m^U(r)$ from  \eqref{eq:seq_time} and recall that $I^{U}(\tau_m^{U})$ denotes the interval containing $U$ at height $\tau_m^{U}$ in the interval fragmentation introduce in \eqref{eq: def Ffrak}. For ease of notation, we introduce the more compact notation $I^j_m:=I^{U_j}(\tau_m^{U_j})$ for intervals and $F^j_m:=F^{U_j}(\tau_m^{U_j})$ for their lengths.

\medskip

\noindent\emph{\underline{Step 2}: W.h.p., at most $n^{1 - \lambda^*(p) + \zeta}$ of the intervals in $\left\{I^{j}_m,j\in[n]\right\}$ have length smaller than $\eps$.}

\medskip

From \cref{cor:key_estimates} and our choice of the constants in Step 0, we have that for all $j\leq n$,
\begin{equation}\label{eq:bound_2}
	\bbP\left(F^{j}_m< \eps\right)\leq 2 \, \eps^{\lambda^*(p)}.
\end{equation}
Now set $S_n:=\sum_{j=1}^n \mathds{1}_{F^{j}_m< \eps}$, \textit{i.e.}\ $S_n$ counts the number of intervals in $\left\{I^{j}_m,j\in[n]\right\}$ having length smaller than $\eps$. By Markov's inequality,
\begin{equation}\label{eq:bound_1}
	\bbP\left(S_n>n^{1 - \lambda^*(p) + \zeta}\right)\leq	\frac{\bbE\left[S_n\right]}{n^{1 - \lambda^*(p) + \zeta}}\stackrel{\eqref{eq:bound_2}}{\leq}\frac{n \cdot 2 \, \eps^{\lambda^*(p)}}{n^{1 - \lambda^*(p) + \zeta}}=2\,n^{-\zeta(1-\lambda^*(p))},
\end{equation}
where in the last equality we used that $\eps=n^{-1+\zeta}$.

\medskip

\noindent\emph{\underline{Step 3}: W.h.p., the total length of the non-discarded intervals in $\left\{I^{j}_m, j\in [n]\right\}$ is at most $\eps^{\lambda^*(p)-\zeta}$.}

\medskip

Let $J$ be the set of indexes
\begin{equation*}
	J:=\left\{j\in [n]\,\middle| \,\eta_b\text{ does not discard } U_j \text{ for all } b\in\mcl B^{U_j}_{\ominus} \text{ such that } b\leq\tau^{U_j}_m\right\},
\end{equation*}
and $\mathcal{J}$ be a subset of $J$ such that 
\begin{equation*}
\left\{I^{j}_m,j\in\mathcal{J}\right\}=\left\{I^{j}_m,j\in J\right\},
\end{equation*}
and all the intervals in $\left\{I^{j}_m,j\in\mathcal{J}\right\}$ are pairwise disjoint (note that by definition if $U_i \in I^{j}_m$ then $I^{i}_m=I^{j}_m$). Note also that by definition, $\LIS(\sigma_n)$ is upper bounded by the cardinality of $J$.

Let now $V$ be an additional uniform random variables on $(0,1)$ sampled independently from all other random quantities. 
Note that if $V$ is contained in $I^{j}_m$ for some $j \in \mathcal{J}$, then the event
\begin{equation*}
	A^V_m=\{\eta_b\text{ does not discard } V \text{ for all } b\in\mcl B^V_{\ominus} \text{ such that } b\leq\tau^V_m\},
\end{equation*} 
occurs. 
Hence, since $V$ is uniform and independent of everything else, setting $\mathcal{L}_m=\sum_{j\in\mathcal{J}} F^{j}_m$, \textit{i.e.}\ $\mathcal{L}_m$ is the total length of the non-discarded intervals among $\left\{I^{j}_m,j\in[n]\right\}$, we get
\begin{equation*}
	\bbE\left[\mathcal{L}_m\right]=\bbP\Bigg(V\in\bigcup_{j\in\mathcal{J}}I^{j}_m\Bigg)\leq\bbP(A^V_m).
\end{equation*} 
Recall that the discarding rule $\eta$ from Step 1 depends only on $(\efrak,\sfrak)$ and $(U_j)_{j\leq n}$, and $V$ is uniform and independent from everything else. By \cref{cor:key_estimates} (applied under the conditional law given $(U_j)_{j\leq n}$) and our choice of the constants in Step 0, 
\begin{equation*}
	\bbE\left[\mathcal{L}_m\right]\leq\bbP(A^V_m)=\bbE\left[\bbP\left(A^V_m\,\middle|\,(U_j)_{j\leq n} \right)\right]\leq 2\,\eps^{\lambda^*(p)}.
\end{equation*} 
Hence, using Markov's inequality, we get that
 \begin{equation*}
 	\bbP(\mathcal{L}_m>\eps^{\lambda^*(p)-\zeta})\leq\frac{\bbE\left[\mathcal{L}_m\right]}{\eps^{\lambda^*(p)-\zeta}}\leq 2 \, {\eps^{\zeta}}.
 \end{equation*}

\medskip

\noindent\emph{\underline{Step 4}: W.h.p., the cardinality of $\mathcal{J}$ is at most $ 2 \,n^{1 - \lambda^*(p) + \zeta}$. Moreover, w.h.p.\ the union of the intervals in $\left\{I^{j}_m,j\in\mathcal{J}\right\}$ can be covered by at most $4\, n^{1 - \lambda^*(p) +  \zeta}$ intervals of length $\eps$ with endpoints in $\eps \mathbb{Z}$.}

\medskip

By  Step 2, w.h.p., at most $n^{1 - \lambda^*(p) + \zeta}$ of the intervals in $\left\{I^{j}_m,j\in\mathcal{J}\right\}$ have length smaller than $\eps$. And so, the union of all these intervals can be covered by at most $2\, n^{1 - \lambda^*(p) +  \zeta}$ intervals of length $\eps$ with endpoints in $\eps \mathbb{Z}$.

By Step 3, w.h.p., the total size $\mathcal{L}_m=\sum_{j\in\mathcal{J}} F^{j}_m$ of the non-discarded intervals in $\left\{I^{j}_m,j\in\mathcal{J}\right\}$ is at most $\eps^{\lambda^*(p)-\zeta}$. Hence, recalling that $\eps=n^{-1+\zeta}$, there are at most $\eps^{-1 + \lambda^*(p)-\zeta}\leq 
n^{1 - \lambda^*(p) +  \zeta}$ intervals in $\left\{I^{j}_m,j\in\mathcal{J}\right\}$  having length bigger than $\eps$ and the union of all these intervals can be covered by at most $2\, n^{1 - \lambda^*(p) +  \zeta}$ intervals of length $\eps$ with endpoints in $\eps \mathbb{Z}$.

\medskip

\noindent\emph{\underline{Step 5}: W.h.p., the cardinality of $J$ (which is an upper bound for $\LIS(\sigma_n)$) is at most $ 4 \cdot n^{1 - \lambda^*(p) + 3 \, \zeta}$.}

\medskip

It remains to deal with the possible discrepancy between the cardinality of $\mathcal{J}$ and $J$. 
Let $I_\eps$ be a deterministic interval of size $\eps$ and with endpoints in $\eps \mathbb{Z}$. Note that the number of uniform variables among $(U_j)_{j\in[n]}$ which fall in $I_\eps$ follows a binomial distribution $\text{Bin}(n,\eps)$. Therefore, for all $n\ge 1$,
\begin{align*}
	\bbP\left(\#\{j\in[n], \, U_j\in I_\eps\}<n^{2\zeta}\right)
	&= 1-\bbP\big(\#\{j\in[n], \, U_j\in I_\eps\} \geq n^{2\zeta}\big)\\
	&= 1-\bbP(\text{Bin}(n,\eps)\geq n^{2\zeta}).
\end{align*}
Now recalling that $\eps=n^{-1+\zeta}$, we get by Chernov's bound that
\begin{align}\label{eq:chern_bnd}
	\bbP\left(\#\{j\in[n], \, U_j\in I_\eps\}<n^{2\zeta}\right)
	&\geq1- \min_{\gamma>0} \{\exp(-\gamma n^{2\zeta})\bbE[\exp(\gamma \text{Bin}(n,n^{-1+\zeta}))]\} \notag\\
	&=1-\min_{\gamma>0}\{ \exp(-\gamma n^{2\zeta}) (1-n^{-1+\zeta}+\exp(\gamma)n^{-1+\zeta})^{n}\} \notag\\
	&\geq
	1-c_1  \cdot \exp(-c_2n^{\zeta}),
\end{align}
where $c_1,c_2>0$ are two constants. Letting $(I^i_\eps)_{i\in[\eps^{-1}]}$ denote the collection of $\lceil\eps^{-1}\rceil$ intervals of size $\eps$ with endpoints in $\eps \mathbb{Z}$ covering $(0,1)$, we get that
\begin{align*}
	\bbP\left(\forall i\in[\eps^{-1}],\#\{j\in[n], \, U_j\in I^i_\eps\}<n^{2\zeta}\right)
	&\,\,\geq\,\,1- \lceil\eps^{-1}\rceil \bbP\left(\#\{j\in[n], \, U_j\in I_\eps\} \geq n^{2\zeta}\right) \notag\\
	&\stackrel{\eqref{eq:chern_bnd}}{\geq}1- c_1 n^{1-\zeta} \cdot \exp(-c_2n^{\zeta}), \notag
\end{align*}
where in the last inequality we also used that $\eps=n^{-1+\zeta}$.

By this estimate and the estimate in Step 4, combined with a union bound, we get that w.h.p.\ the total number of points $(U_j)_{j\leq n}$ in the union of the intervals $\{I^{j}_m : j \in \mathcal{J}\}$ is at most $n^{2\zeta} \cdot 4\, n^{1 - \lambda^*(p) +  \zeta}$. In particular, w.h.p., the cardinality of $J$ (which is an upper bound for $\LIS(\sigma_n)$) is at most $ 4\, n^{1 - \lambda^*(p) +  3 \zeta}$.
\end{proof}

\appendix
\section{Numerical simulations}\label{sect:num-sim}

A natural question in light of our results in \cref{thm:upper_lower_perm} and \cref{thm:upper_lower_graphons} is to determine the exact exponent for the polynomial growth of $\LIS(\Perm(\bm{\mu}_p,n))$ and $\LHS(\Graph(\bm{W}_p,n))$.
Since the answers to these questions are the same (recall \cref{lem:perm_to_graph}), in this section, we focus on $\LIS(\Perm(\bm{\mu}_p,n))$ for $p\in(0,1)$.

In \cref{conj:first-conj}, we conjectured that with probability tending to 1 as $n\to\infty$,
\begin{equation*}
	\LIS(\Perm(\bm{\mu}_p,n))=n^{d(p)+o(1)}.
\end{equation*}
We did several numerical simulations to estimate the exact values of $d(p)$. Our simulations were done in the following way. For $p\in(0,1)$ fixed, we sampled one million independent permutations $\Perm(\bm{\mu}_p,2^k)$ of size $2^k$ and we computed $\LIS(\Perm(\bm{\mu}_p,2^k))$, for all $k=10,\dots,18$. Then for each fixed $k$, we computed the average of $\LIS(\Perm(\bm{\mu}_p,2^k))$, denoted by $\overline{\LIS_k(p)}$, over the one million samples. Finally, we performed a linear regression on the points $\left(k\log(2),\log(\overline{\LIS_k(p)})\right)$ for $k=10,\dots,18$, obtaining a linear function $a(p)+\overline{d(p)}x$. It is quite straightforward to realize that assuming \cref{conj:first-conj} one should have that 
\begin{equation*}
	d(p)\approx \overline{d(p)}.
\end{equation*}
See \cref{fig-regression} for the linear regression when $p=1/2$. See also the table and the plot in \cref{table:regressions} for a summary of our numerical simulations for different values of the parameter $p\in (0,1)$.

We highlight that the discrepancy between our lower bound $\alpha_*(p)$ and the numerical values for $d(p)$ is tiny (see the fourth column in the table in  \cref{table:regressions}). 
Since these simulations look only at the \emph{length} of the longest increasing subsequence, we also tested if the real longest increasing subsequence and the subsequence obtained through our selection rule $\sfS$ in \eqref{eq:selection} are close, getting a positive answer. See the results in \cref{fig-simulations-lis}, p.\ \pageref{fig-simulations-lis}. 

\begin{figure}[ht!]
	\centering
	\includegraphics[width=.46\textwidth]{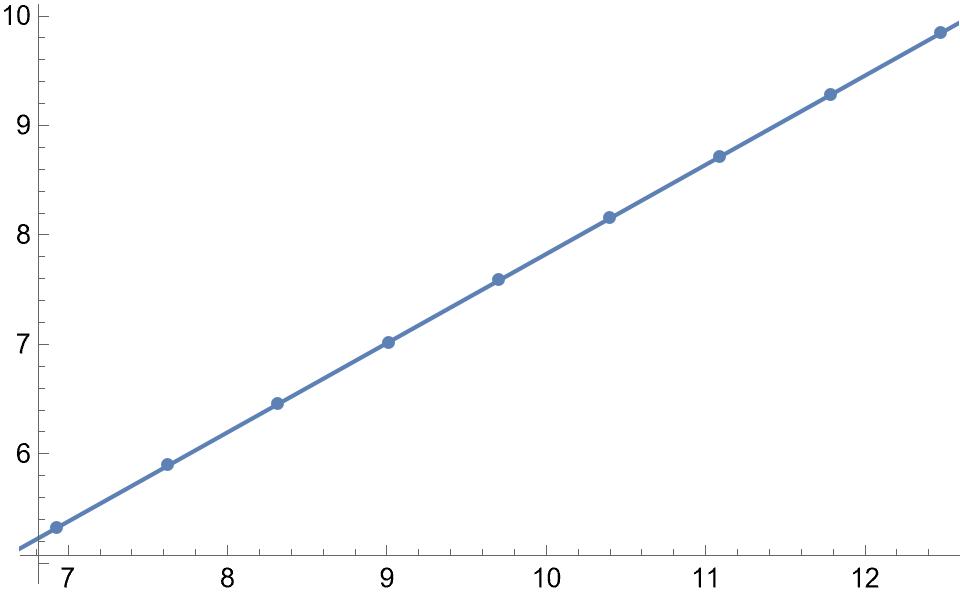}
	\caption{\label{fig-regression} The blue dots are the points $\left(k\log(2),\log(\overline{\LIS_k(p)})\right)$ for $k=10,\dots,18$. The linear regression of these points gives the line $a(1/2)+\overline{d(1/2)}x=-0.323+0.815 x$. Therefore we estimate that $d(1/2)\approx0.815$.
	}
\end{figure}

\begin{figure}[ht!]
	\begin{minipage}{.495\textwidth}
		\begin{tabular}{ |P{0.8cm}||P{1cm}|P{2.3cm}|P{2.2cm}| }
			\hline
			\multicolumn{4}{|c|}{ $\LIS(\Perm(\bm{\mu}_p,n))\approx C(p)\cdot n^{d(p)}$} \\
			\hline
			$p$ & $\alpha_*(p)$ & $d(p)$ & $|d(p)-\alpha_*(p)|$ \\
			\hline
			0.1 & 0.584 & $0.58\pm 10^{-2}$ & $< 2\cdot10^{-2}$ \\
			0.2 & 0.653 & $0.656\pm 7\cdot10^{-3}$ & $< 1.4\cdot10^{-2}$ \\
			0.3 & 0.712 & $0.715\pm 5\cdot10^{-3}$ & $< 10^{-2}$ \\
			0.4 & 0.765 & $0.766\pm 4\cdot10^{-3}$ & $< 8\cdot10^{-3}$ \\
			0.5 & 0.812 & $0.814\pm 3\cdot10^{-3}$ & $< 6\cdot10^{-3}$ \\
			0.6 & 0.855 & $0.857\pm 2\cdot10^{-3}$ & $< 4\cdot10^{-3}$ \\ 
			0.7 & 0.895 & $0.897\pm 2\cdot10^{-3}$ & $< 4\cdot10^{-3}$ \\
			0.8 & 0.932 & $0.933\pm10^{-3}$ & $<  2\cdot10^{-3}$ \\ 
			0.9 & 0.967 & $0.967\pm10^{-3}$ & $<  2\cdot10^{-3}$ \\
			\hline
		\end{tabular}
	\end{minipage}
	\begin{minipage}{.495\textwidth}
		\includegraphics[width=\textwidth]{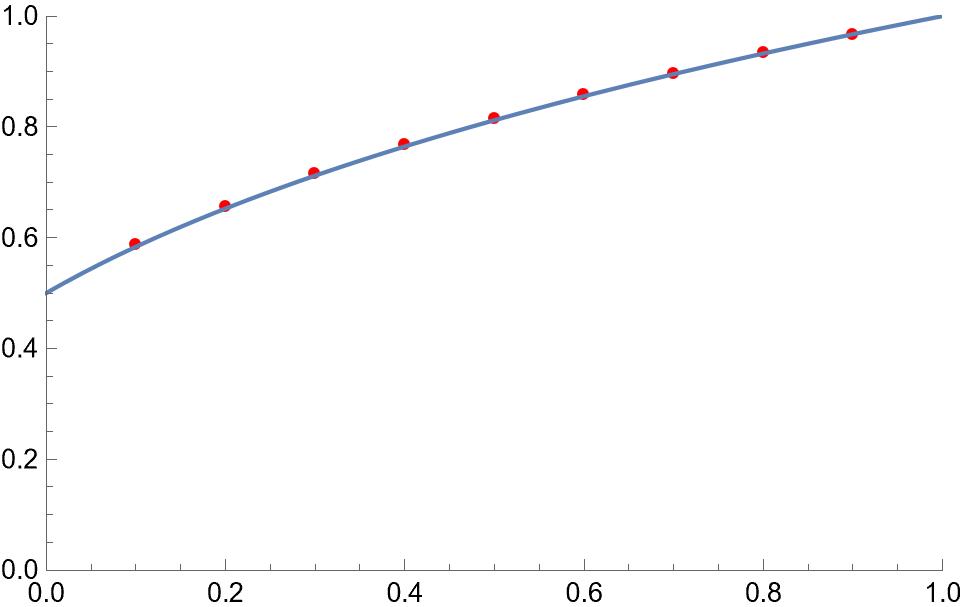}
	\end{minipage}
	\caption{\label{fig-simulations-lis-diagram} \textbf{Left:} For various values of the parameter $p$ (first column) we indicate the value of our lower bound $\alpha_*(p)$ (second column), the value of the exponent $d(p)$ estimated from simulations (third column), and the difference between these two values (fourth column).\label{table:regressions} \textbf{Right:} In blue, the plot of the function $\alpha_*(p)$. In red, the numerical values for the exponents $d(p)$ (computed numerically).}
\end{figure}

\begin{figure}[ht!]
	\includegraphics[width=.195\textwidth]{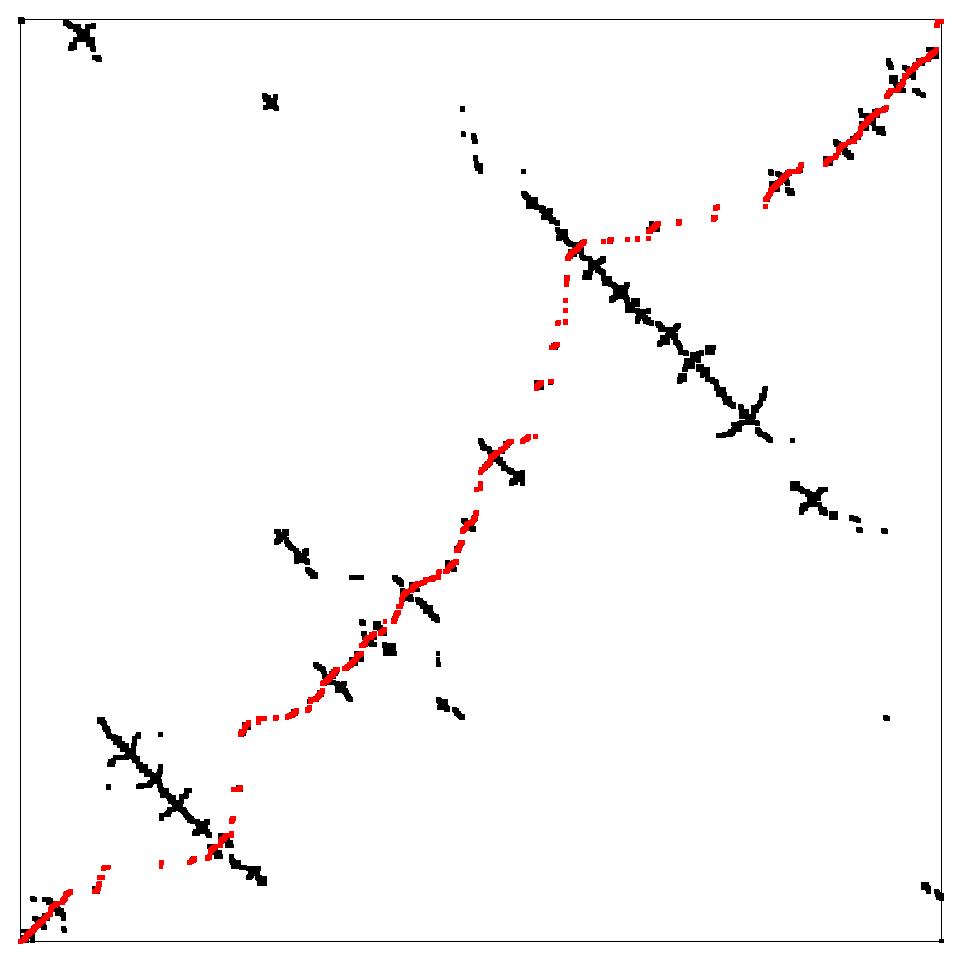}
	\includegraphics[width=.195\textwidth]{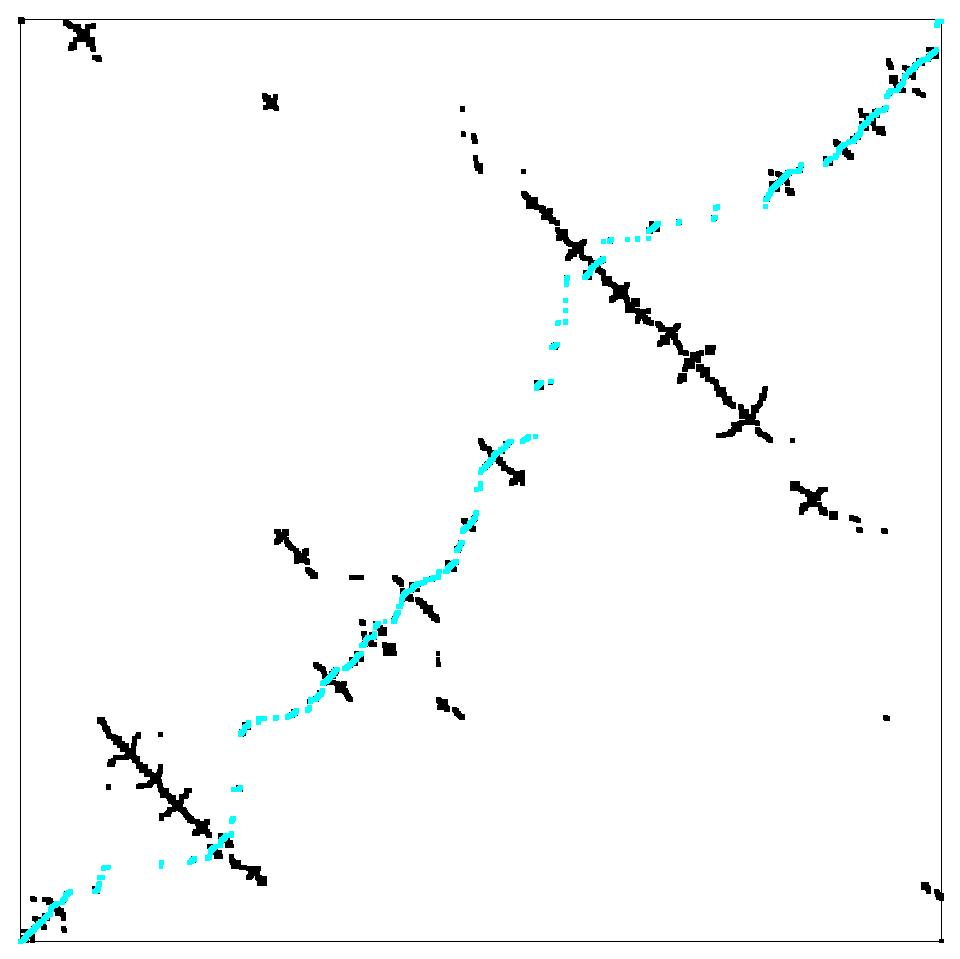}
	\includegraphics[width=.195\textwidth]{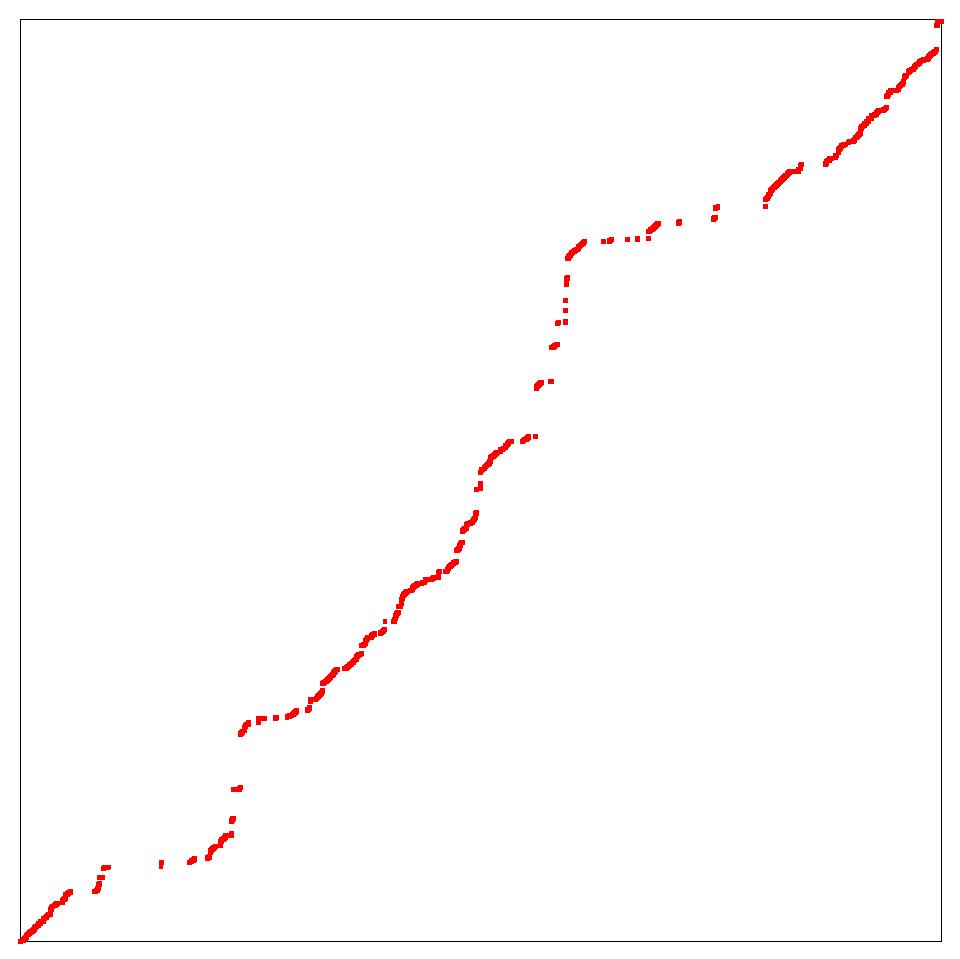}
	\includegraphics[width=.195\textwidth]{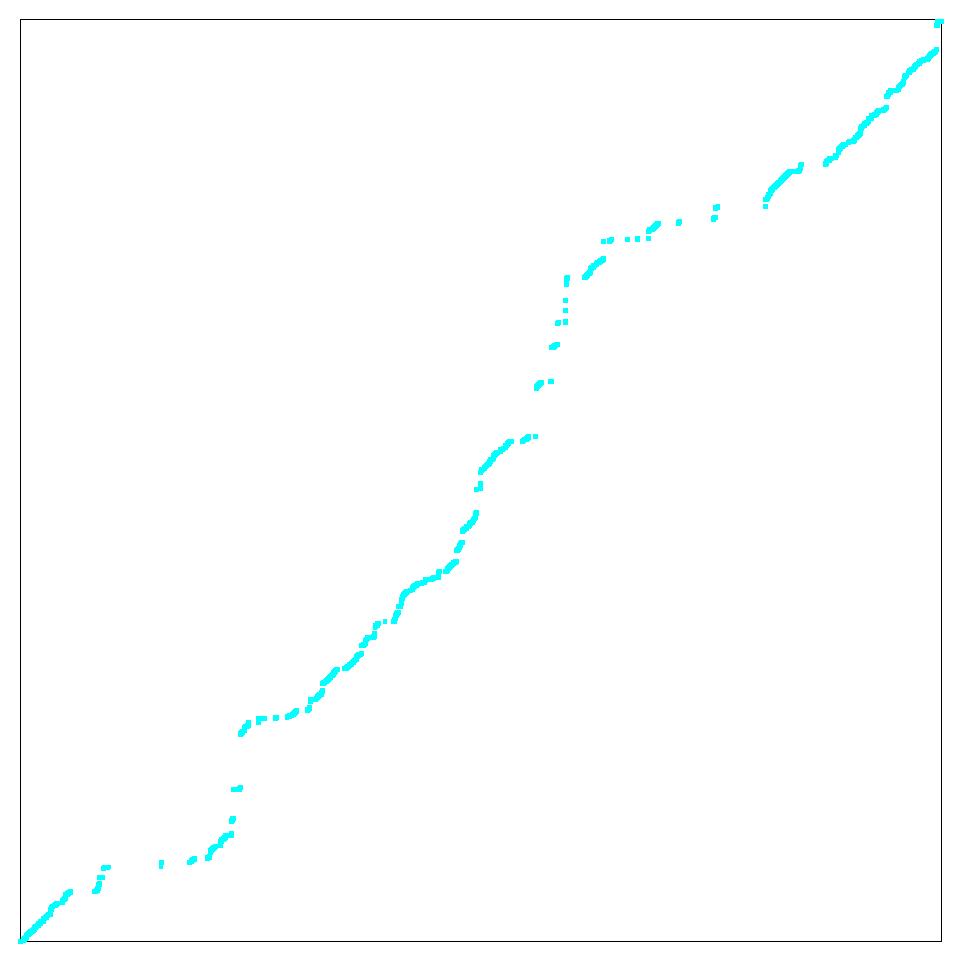}   
	\includegraphics[width=.195\textwidth]{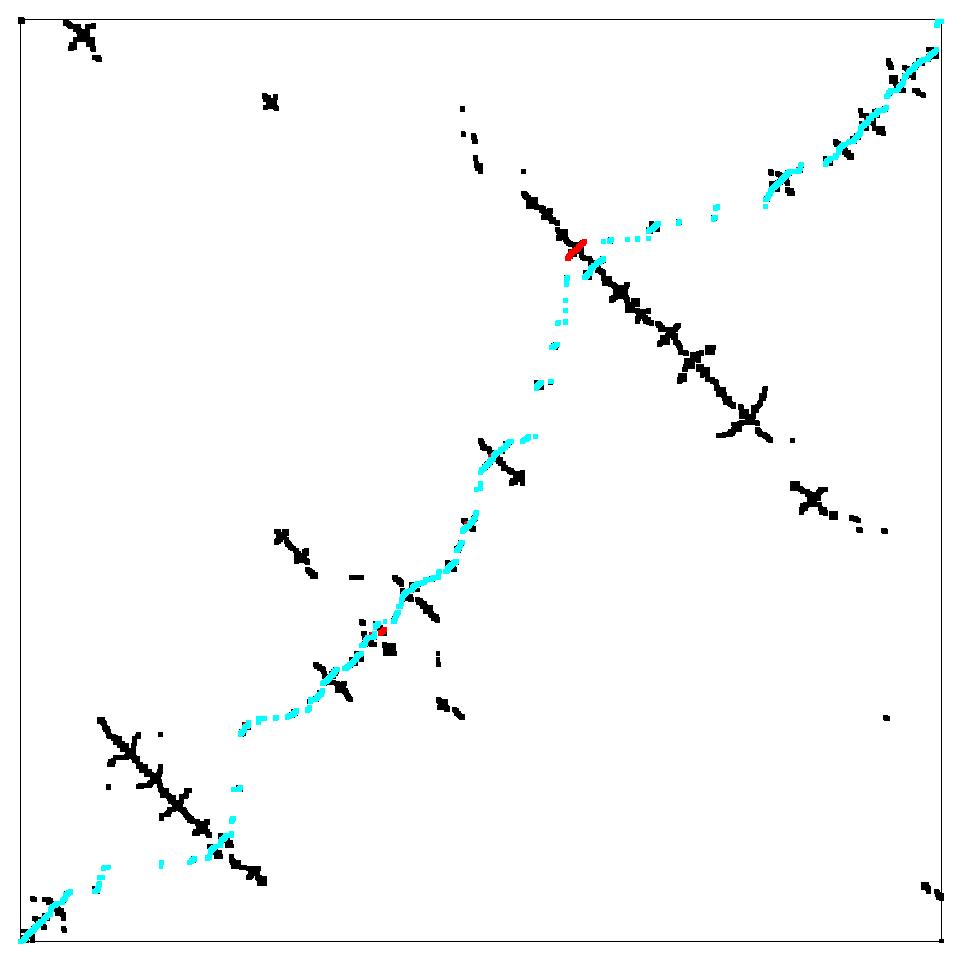}
	\caption{\label{fig-simulations-lis} \textbf{From left to right:} (1) A permutation of length $262144$ sampled from the Brownian separable permuton $\bm\mu_{1/2}$ with one longest increasing subsequence in red of length $22546$. (2) The same permutation with one increasing subsequence of length $21751$ in cyan computed using our selection rule $\sfS$ in \eqref{eq:selection}. (3-4) The two diagrams in (1) and (2) with only the two increasing subsequences. (5) The cyan increasing subsequence is plotted on top of the red increasing subsequence. Note that the two sequences are very similar since the cyan subsequence almost completely covers the red subsequence.}
\end{figure}

\bibliography{cibib,cibib2}
\bibliographystyle{hmralphaabbrv}

\end{document}